\documentclass[11pt,reqno]{amsart}
\usepackage[top=4.2cm, bottom=4cm, left=2.4cm, right=2.4cm]{geometry}
\usepackage[utf8]{inputenc}
\usepackage[USenglish]{babel}
\usepackage[T1]{fontenc} 
\usepackage{mathrsfs}
\usepackage{wasysym}
\usepackage{tikz-cd}
\usepackage{mathtools}
\usepackage{esint}
\usepackage{amsmath}
\usepackage{amssymb,epsfig}
\usepackage{amsthm}
\usepackage[absolute]{textpos}
\usepackage{mathtools,emptypage}
\usepackage{tikz}
\mathtoolsset{showonlyrefs}  % uncomment to tag only the referred equations (post-production)

\usetikzlibrary{decorations.markings, arrows.meta}

\usepackage[bookmarks=true]{hyperref}
\usepackage{xcolor}
\hypersetup{
    colorlinks,
    linkcolor={red!50!black},
    citecolor={blue!50!black},
    urlcolor={blue!80!black},
}
\numberwithin{equation}{section}

\newcommand{\sfd}{{\sf d}}
\newcommand{\restr}[1]{\lower3pt\hbox{$|_{#1}$}}

\newcommand{\Kliminf}{K\kern-3pt-\kern-2pt\mathop{\rm lim\,inf}\limits}  % Kuratowski liminf di insiemi
\newcommand{\dist}{\mathop{\rm dist}\nolimits}
\newcommand{\Lip}{\mathop{\rm Lip}\nolimits}          %Lipschitz constant

\newcommand{\R}{{\mathbb R}}
\newcommand{\N}{\mathbb N}
\newcommand{\diam}{{\rm diam}}

\newcommand{\loc}{{\rm loc}}

\newcommand{\mm}{\mathfrak m}                                %misura di riferimento

\newtheorem{theorem}{Theorem}[section]

\newtheorem{corollary}[theorem]{Corollary}
\newtheorem{lemma}[theorem]{Lemma}
\newtheorem{proposition}[theorem]{Proposition}
\newtheorem{fact}[theorem]{Fact}

\theoremstyle{definition}
\newtheorem{remark}[theorem]{Remark}
\newtheorem{definition}[theorem]{Definition}

\newtheorem{example}[theorem]{Example}
\newtheorem{question}[theorem]{Question}
\newtheorem{claim}[theorem]{Claim}

\begin{document}

\title[]{Extension properties of planar subsets supporting a weak $(1,1)$-Poincar\'e inequality}

\author{Miguel Garc\'ia-Bravo}

\address{Department of Mathematical Analysis and Applied Mathematics, Faculty of Mathematics,
University Complutense of Madrid, 28040 Madrid, Spain, Spain.  ICMAT (CSIC-UAM-UC3-UCM).}

\email{miguel05@ucm.es}

 \author{Tapio Rajala}        
\address{University of Jyv\"askyl\"a \\
         Department of Mathematics and Statistics \\
         P.O. Box 35  \\
         FI-40014 University of Jyv\"askyl\"a \\
         Finland}

\email{tapio.m.rajala@jyu.fi}

\thanks{The first named author acknowledges the support provided by the grant PID2022-138758NB-I00 from the Ministerio de Ciencia e Innovación of Spain.}
\date{\today}

%%%%%%%%%%%%%%%%%%%%%%%%%%%%%%%%%%%%%%%%%%%%%%%%%%%%%%%%%%%%%%%%%%%%%

%%%%%%%%%%%%%%%%%%%%%%%%%%%%%%%%%%%%%%%%%%%%%%%%%%%%%%%%%%%%%%%%%%%%%

\begin{abstract}
We study Sobolev and $BV$-extension properties of planar subsets. In particular, we prove that fat Sierpi\'nski carpets that support a weak $(1,1)$-Poincaré inequality are $W^{1,1}$-extension sets, and we provide an explicit linear extension operator for them. We also show that 
for planar domains satisfying a weak $(1,1)$-Poincaré inequality the $W^{1,1}$-extension property and the $BV$-extension property are equivalent. Finally, as a consequence of the previous result, we show that if a simply-connected planar domain is Ahlfors regular and supports a weak $(1,1)$-Poincaré inequality, then it is a (linear) $W^{1,1}$-extension domain.

\end{abstract}

\maketitle

\section{Introduction}

We study extension properties for open and closed subsets $E \subset \R^2$ that satisfy a weak Poincar\'e inequality and are Ahlfors $2$-regular. We always understand  $\R^2$ endowed with the Euclidean distance and the $2$-dimensional Lebesgue measure. The extensions we are interested are for Sobolev $W^{1,p}$-functions and $BV$-functions.

\subsection{Basic notions}

In a more general context, in this paper $(X,\sfd,\mm)$ denotes a separable metric measure space where $\mm$ is a nonnegative Borel measure which is finite and positive on every ball. We say that a Borel set $E \subset X$ has the $W^{1,p}$-extension property if there exists a bounded (not necessarily linear) operator $T \colon W^{1,p}(E) \to W^{1,p}(X)$ such that $Tu|_E = u$ for all $u \in W^{1,p}(E)$. The same definition applies for $BV$ functions. For us, the space of Sobolev functions $W^{1,p}$ is understood in the sense of Shanmugalingam \cite{S00} using $p$-weak upper gradients, and the space of functions with bounded variation $BV$ is understood in the sense of Miranda \cite{Mir03} through approximations by Lipschitz functions. In Section \ref{sec:preliminaries} we give the precise definition of these notions. However, we warn the reader that all our results are established for either open or closed subsets of $\R^2$.

For a metric measure space $(X,\sfd,\mm)$ recall that $\mm$ is said to be doubling if there exists a constant $c_\mm>0$ so that $\mm(2B)\leq c_\mm \mm(B)$ for every open ball $B\subset X$. Also, $X$ is said to support a weak $(1,p)$-Poincaré inequality  for some $1\leq p<\infty$ if there are constants $C>0$, $\lambda\geq 1$ such that for all balls $B\subset X$, all integrable functions $u\colon \lambda B\to\R$ and all $L^p$ integrable upper gradients $\rho\colon \lambda B \to [0,\infty]$ of $u$ we have
    \begin{equation}\label{eq:1-PI}
   \dfrac{1}{\mm(B)} \int_B |u-u_B|\, d\mm\leq \dfrac{C\diam (B)}{\mm(\lambda B)}\int_{\lambda B} \rho^p\, d\mm.
    \end{equation}
For the definition of upper gradients we refer to Section \ref{sec:preliminaries}.
If \eqref{eq:1-PI} only holds for balls $B$ with radii at most $r_0>0$ we say that  $X$ supports a weak $(1,p)$-Poincaré inequality up to scale $r_0>0$.
A set $E \subset X$ satisfies a weak $(1,p)$-Poincar\'e inequality if $(E,\sfd|_{E\times E},\mm|_E)$ does as a metric measure space. A metric measure space is called a $(1,p)$-PI space if $X$ is complete, 
$\mm$ is doubling and $X$ supports a weak $(1,p)$-Poincaré inequality. Moreover, a set $E\subset X$ is said to be Ahlfors $s$-regular if there exists $s\geq 1$ and $C\geq 1$ so that for every  $r\in (0,\diam(X)]$ and $x\in E$ we have
\begin{equation}
C^{-1}r^s\leq \mm(B(x,r)\cap E)\leq C r^s.
\end{equation}
In particular, if $E$ is Ahlfors $s$-regular for some $s\geq 1$ then $\mm|_E$ is doubling.
A set $E\subset X$ is said to satisfy a measure density condition if there exists $c\in (0,1]$ so that for all $x\in E$ and $r\in (0,1]$ we have
\begin{equation}
    \mm(E\cap B(x,r))\geq c\,\mm (B(x,r)).
\end{equation}
Observe that in Euclidean spaces with the Lebesgue measure (denoted here by $|\cdot|$), a bounded set $E\subset \R^n$  is Ahlfors $n$-regular if and only if $E$ has the measure density condition. From now on we will often call a set $E\subset \R^n$ Ahlfors regular whenever it is Ahlfors $n$-regular. In case $E\subset \R^n$ is an open Ahlfors regular set, the Lebesgue differentiation theorem implies that $|\partial E|=0$. However, this conclusion does not always hold if $E$ is not open; see for instance the fat Sierpi\'nski carpets from Section \ref{sec:Siperpinski carpet}.

\subsection{Motivation}

The starting point for our discussion is the following result contained in \cite[Proposition 1.10]{GIZ23}, see also \cite[Proposition 5.1]{Bjo:Sha:07}.
We point out that the result still holds for $p=\infty$, by \cite{G-BIZ25}.

\begin{proposition}{\cite[Proposition 1.10]{GIZ23}}\label{prop:plargerthan1}
Let $(X,\sfd,\mm)$ be a $(1,p)$-PI space, let $E\subset X$ be measurable and let $1<p<\infty$. If $E$ satisfies a measure density condition and  supports a weak $(1,p)$-Poincaré inequality up to some scale then we have that $E$ is a $W^{1,p}$-extension set.
\end{proposition}

The main question for us is to know if the above result holds for $p=1$; Question \ref{question}. Before considering this question, let us comment on some converse statements of Proposition \ref{prop:plargerthan1}.

On the one hand, if in Proposition \ref{prop:plargerthan1} the space $X$ is further assumed to be Alhfors $s$-regular for some $s\geq 1$, and we have that $E\subset X$ is a $W^{1,p}$-extension set for some $1\leq p<\infty$, then $E$ has the measure density condition. We refer to \cite{HKT2008b} for a proof of this result, and also to \cite{Koskela,HKT2008a} for the Euclidean situation. We also mention that in \cite[Proposition 4.2]{CKLR25} the authors prove the measure density condition without requiring any Ahlfors regularity of the set $E$. On the other hand, the validity of a weak $(1,p)$-Poincaré inequality up to some scale holds for $W^{1,p}$-extensions domains $\Omega \subset \R^n$ if $n\leq p<\infty$. The proof of this fact appears in \cite[Theorem 1.11]{GIZ23}, which heavily relies on \cite{K1998} and \cite[Theorem 5.1]{BB2019}. Up to our knowledge, it remains an open question whether \cite[Theorem 1.11]{GIZ23} also holds for $(1,p)$-PI spaces $X$ that are Ahlfors $s$-regular for some $1\leq s\leq p$. We also remark that there exist $W^{1,p}$-extension domains $\Omega\subset \R^n$ for $1\leq p<n$ which do not satisfy any weak $(1,q)$-Poincaré inequality; for instance Figure \ref{fig:triangle-square} in Section \ref{sec:example} illustrates  a typical situation of a $W^{1,p}$-extension set for $1\leq p<2$ which is not quasiconvex and hence does not satisfy any weak Poincaré inequality. Other examples can be found on \cite[Example 2.5]{Koskela}, \cite[Remark 5.2]{Bjo:Sha:07} or \cite[Example 7.9]{GIZ23}. 

As previously mentioned, our initial motivation for this work was to extend Proposition \ref{prop:plargerthan1} to the case $p=1$. The method for proving Proposition \ref{prop:plargerthan1} for the cases $p>1$ does not work for $p=1$. A crucial step in that proof, which fails for $p=1$, relies on the equality $M^{1,p}(X)=W^{1,p}(X)$ for $p>1$, where $M^{1,p}$ denotes the Haj{\l}asz-Sobolev space (see \cite{HKST15} and \cite[Example 23]{KS2008}). 

We now formulate this question in a general setting.

\begin{question}\label{question}
Suppose that $(X,\sfd,\mm)$ is an Ahlfors $s$-regular metric measure space satisfying a weak $(1,1)$-Poincar\'e inequality and let $E \subset X$ be a closed set or a domain so that $(E,\sfd,\mm|_E)$ is also Ahlfors $s$-regular and satisfies a weak  $(1,1)$-Poincar\'e inequality. Is then $E$ a $W^{1,1}$-extension set?
\end{question}

We have not been able to answer this question even in the Euclidean plane. However, we have an affirmative answer for an important class of planar sets, called fat Sierpi\'nski carpets.

\subsection{Fat Sierpi\'nski carpets}

Given a sequence $\mathbf a  = (a_i)_{i=1}^\infty$ of reciprocals of odd integers that are at least $3$, the Sierpi\'nski carpet $\mathcal S_\mathbf a$ is defined by dividing the unit square to a $1/a_1 \times 1/a_1$ grid of subsquares, leaving out the middle one and continuing in each of the remaining squares with a $1/a_2 \times 1/a_2$ grid, and so on. When $\mathbf a\in c_0$ the Hausdorff dimension of $\mathcal S_{\mathbf a}$ is $2$ (see \cite[Proposition 3.1]{Sierpinski}), and the set $\mathcal S_{\mathbf a}$ is called a {\em fat Sierpi\'nski carpet}. See Section \ref{sec:Siperpinski carpet} for the more precise definition. Our main result of the paper reads as follows.

\begin{theorem}\label{thm:Sierp-W11_ext._char}
    Let $\mathcal S\subset \R^2$ be a fat Sierpi\'nski carpet. Then there exists a bounded extension operator $T\colon W^{1,1}(\mathcal S)\to W^{1,1}(\R^2)$ if and only if $\mathcal S$ satisfies a weak $(1,1)$-Poincar\'e inequality.
    In this case, the extension operator $T$ can be chosen to be a linear.
\end{theorem}

The validity of weak $(1,p)$-Poincar\'e inequalities on fat Sierpi\'nski carpets was studied in \cite{Sierpinski}, by Mackay, Tyson and Wildrick. The most relevant result in  our context is \cite[Theorem 1.5]{Sierpinski} according to which 
$\mathcal S_{\mathbf a}$ satisfies a weak $(1,1)$-Poincar\'e inequality if and only if $\mathbf a \in \ell^1$. Moreover, by \cite[Theorem 1.6]{Sierpinski}, fat Sierpi\'nski carpets $\mathcal S_{\mathbf a}$ satisfy a weak $(1,p)$-Poincaré inequality for some or any $p>1$ if and only if $ \mathbf a\in\ell_2$, and in this case $\mathcal S_{\mathbf a}$ are known to be Ahlfors regular, c.f. \cite[Proposition 3.1]{Sierpinski}. We recall that if a set satisfies a weak $(1,p)$-poincaré inequality for some $p\geq 1$ then it satisfies a weak $(1,q)$-Poincaré inequality for any $q\geq p$. For some generalizations of \cite{Sierpinski}  we refer to \cite{EG21,EG22}. 

The paper \cite{EG21} is of particular interest as it provides a necessary geometric condition satisfied by Ahlfors regular planar sets with a weak $(1,1)$-Poincaré inequality. As noted in \cite[Remark 1.6]{EG22}, this condition is nearly optimal. Namely, \cite[Theorem 4.40]{EG21} asserts\footnote{The proof of \cite[Theorem 4.40]{EG21} concerns Loewner carpets and relies on \cite[Theorem 4.32]{EG21}, which assumes that the closed set $E$ has an empty interior. However, the result remains valid even if $E$ has a nonempty interior. Indeed, this assumption is only used in \cite[Theorem 4.32]{EG21} to guarantee that the connected components $\Omega_i$ of $E^c$ are Jordan domains. If $E$ has a nonempty interior, one can introduce additional holes within $E$ to force an empty interior without destroying its other properties, after which \cite[Theorem 4.40]{EG21} applies.} that if $E\subset\R^2$ is a closed Ahlfors regular set satisfying a weak $(1,1)$-Poincaré inequality, 
and if $E^c = \bigcup_{i \in I} \Omega_i$ is the decomposition of its complement into open, connected components, then each $\Omega_i$ is a Jordan uniform domain (in the extended plane) and there exists some $\lambda>0$ so that for $i\neq j$,
$$\dist (\Omega_i,\Omega_j)\geq \lambda \min\{\diam(\Omega_i),\diam(\Omega_j) \}. $$
The converse does not hold in general, as demonstrated by the fat Sierpiński carpets $\mathcal{S}_{\mathbf{a}}$ for $\mathbf{a} \in \ell_2 \setminus \ell_1$ (see \cite{Sierpinski}), which satisfy this geometric condition and support a weak $(1,p)$-Poincaré inequality for every $p > 1$, yet fail to do so for $p = 1$.

Another classical property of sets supporting a weak Poincaré inequality is quasiconvexity. For the particular case of $p=1$, 
if a Borel set $E\subset \R^n$ is Ahlfors regular and satisfies a weak $(1,1)$-Poincaré inequality, then $E$ is quasiconvex. We refer to \cite[Theorem 9.4.1]{HKST15}  for a proof of this result in the case of complete sets and to \cite[Remark 3.3]{D-CJS2016} for the non-complete case. On the other hand, quasiconvexity is not a sufficient condition for the validity of weak Poincaré inequalities; a prime example is the classical Sierpiński carpet $\mathcal S_{\mathbf a}$ with $\mathbf a=(1/3,1/3,1/3,\dots)$.

Unfortunately, our method of proof for Theorem \ref{thm:Sierp-W11_ext._char} does not extend to the general case of Ahlfors regular closed sets $E\subset \R^2$ that satisfy a weak $(1,1)$-Poincaré inequality. This is illustrated in Example \ref{ex:method_fails}.

\subsection{$BV$ and $W^{1,1}$-extension sets with the weak $(1,1)$-Poincaré inequality}

We now move the discussion to $BV$-extension sets. In particular in Question \ref{question} with the same assumptions one can ask if $E$ is a $BV$-extension set. We will see in Theorem \ref{prop:equiv_of_BV_and W11} that for bounded domains $\Omega \subset \R^2$ supporting a weak $(1,1)$-Poincaré inequality the $W^{1,1}$-extension property and the $BV$-extension property are equivalent. As a consequence, a positive answer to Question \ref{question} in the $BV$-situation  for planar bounded domains would imply a positive answer in the $W^{1,1}$ case too.

The theory of $BV$-extension sets starts with the work of Burago and Mazy'a. The result in \cite{BM1967}, together with \cite[Lemma 2.1]{KMS2010}, establishes that a bounded domain $\Omega\subset \R^n$ is a $BV$-extension domain if and only if there exists a constant $C>0$ so that any set $F\subset \Omega$ of finite perimeter in $\Omega$ admits an extension $\widetilde F\subset\R^n$, modulo measure zero sets, that is 
\begin{equation}\label{eq:intro_1}
|((\widetilde F\cap \Omega) \setminus F)\cup (F\setminus (\widetilde F\cap \Omega))|=0,
\end{equation}
so that the following estimate for the perimeters hold 
\begin{equation}\label{eq:intro_2}
P(\widetilde F,\R^n)\leq CP(F,\Omega).\end{equation} 

Moreover, if we restrict ourselves to simply connected planar domains $\Omega\subset \R^2$ we know from \cite{KMS2010} that the $BV$-extension property is equivalent to the quasiconvexity of $\R^2\setminus \Omega$.

The results in \cite{BM1967} have been generalized very recently in \cite{CKR23} to metric measure spaces $X$. See Section \ref{sec:preliminaries} for the definitions of $BV$ functions and perimeter of a set both in the Euclidean setting and in metric measure spaces.

\begin{proposition}{\cite[Proposition 3.4]{CKR23}}\label{prop:BV_ext_per_full_norm}
   Let $(X,\sfd,\mm)$ be a complete separable metric measure space where $\mm$ is a nonnegative Borel measure finite on bounded sets. Then a Borel set  $E\subset X$ is a $BV$-extension set if and only if there is a constant $C>0$ so that for every set $F\subset E$ of finite perimeter in $E$ there exists $\widetilde F\subset X$ with 
\begin{enumerate}
    \item[(PE1)] $\mm(( F\setminus (\widetilde F\cap E))\cup((\widetilde F\cap E)\setminus F))=0  $.
    \item[(PE2)] $\mm(\widetilde F)+ P_X(\widetilde F)\leq C(\mm(F)+P_E(F)).$
\end{enumerate}
\end{proposition}
This result will be useful in the proof of Theorem \ref{thm:Sierp-W11_ext._char} for showing that fat Sierpinksi carpets $\mathcal S_\mathbf a$ for $a\notin \ell_1$ cannot be $BV$-extension set, and hence neither $W^{1,1}$-extension set.

The relation between $W^{1,1}$-extension sets and $BV$-extension sets has been studied in  \cite{KMS2010,BR21,CKR23,CKLR25}. For instance, it is known that every $W^{1,1}$-extension set $E\subset \R^n$ is a $BV$-extension set whenever $E$ is a domain (see \cite[Lemma 2.4]{KMS2010}) or a closed set (see \cite[Proposition 3.4]{CKLR25}). On the contrary, there exist examples of $BV$-extension domains, as for example the slit disc, which are not $W^{1,1}$-extension sets. An example of a metric measure space admitting a closed $BV$-extension subset that fails to be a $W^{1,1}$-extension set was given in \cite[Example 3.7]{CKLR25}. However, in the Euclidean setting, the situation remains unclear: \cite[Question 1.3]{CKLR25} asks whether a set $E\subset \R^2$ that is a $BV$-extension closed set must also be a $W^{1,1}$-extension set. A partial answer is given in  \cite[Theorem 1.3]{CKLR25}, establishing that if $E\subset \R^2$ is a compact $BV$-extension set with $\R^2\setminus E$ consisting on finitely many components then $E$ is a $W^{1,1}$-extension set.

As we just have explained, it is well known that $BV$ extension domains or closed sets do not necessarily have to be, in general, $W^{1,1}$-extension sets. However, any existing counterexample in this direction does not satisfy a weak $(1,1)$-Poincaré inequality, and it seems therefore interesting to study whether the $BV$-extension property and the $W^{1,1}$-extension property could be equivalent under this further assumption. In this paper we have managed to give a positive solution for the case of planar bounded domains $\Omega\subset\R^2$.  This is stated as our  second main result.
\begin{theorem}
\label{prop:equiv_of_BV_and W11}
       Let $\Omega\subset\R^2$ be a bounded domain satisfying the weak $(1,1)$-Poincaré inequality. Then $\Omega$ is a $BV$-extension domain if and only if $\Omega$ is a $W^{1,1}$-extension domain.
\end{theorem}
As an immediate application, we state the following interesting corollary and include its straightforward proof for completeness.
\begin{corollary}
    Let $\Omega\subset \R^2$ a bounded planar simply connected domain. Suppose that $\Omega$ is Ahlfors $2$-regular and that it satisfies a weak $(1,1)$-Poincaré inequality. Then there exists a linear bounded extension operator $T\colon W^{1.1}(\Omega)\to W^{1,1}(\R^2)$.
\end{corollary}
\begin{proof}
    Since $\Omega$ satisfies a weak $(1,1)$-Poincaré inequality it follows from the Hölder inequality that $\Omega$ satisfies a weak $(1,p)$-Poincaré inequality for every $1< p\leq \infty$. Therefore, by Proposition \ref{prop:plargerthan1}, we know that $\Omega$ is a $W^{1,p}$-extension set for all $1<p\leq \infty$. Now, we present two possible arguments to conclude the proof:
    \begin{enumerate}
        \item By the results in \cite{KRZ25} the complement of $\Omega$ (which is just one component) is quasiconvex. For bounded planar simply connected domains \cite[Theorem 1.1]{KMS2010} gives that $\Omega$ must be a $BV$-extension domain. Finally, we use Theorem \ref{prop:equiv_of_BV_and W11} to conclude that $\Omega$ must be a $W^{1,1}$-extension domain as well. The linearity of this $W^{1,1}$-extension set can be achieved by using \cite{KRZ25'}.
        \item Bounded planar simply connected $W^{1,2}$-extension domains $\Omega$ are known  from \cite{GR1990} to be uniform domains, i.e. quasidisks, and in particular by using Jones' result \cite{Jones} $\Omega$ must be a (linear) $W^{1,p}$-extension domain for every $1\leq p\leq \infty$.
    \end{enumerate}
    Note that the above argument implies that Ahlfors regular bounded planar simply connected domains $\Omega\subset \R^2$ with the weak $(1,1)$-Poincaré inequality must be Jordan uniform domains.    
\end{proof}

Let us next give some ideas behind the proof of the sufficiency in Theorem \ref{prop:equiv_of_BV_and W11}. In the Euclidean setting  \cite[Theorem 1.3]{BR21} states that a bounded domain $\Omega\subset \R^n$ is a $W^{1,1}$-extension set if and only if $\Omega$ has the strong extension property for sets of finite perimeter. This last property means that there exists a constant $C>0$ so that for any set $E\subset\Omega$ of finite perimeter in $\Omega$ there exists a set $\widetilde E\subset\R^n$ satisfying \eqref{eq:intro_1}, \eqref{eq:intro_2} and so that $P(\widetilde E,\partial\Omega)=0$. In a more general setting of open subsets $\Omega\subset X$ of a separable metric measure space $X=(X,\sfd,\mm)$ ($\mm$ a Borel nonnegative measure finite on bounded sets) a similar characterization through the strong extension property of sets of finite perimeter is true whenever $\mm(\partial\Omega)=0$. 
\begin{theorem}{\cite[Theorem 1.5]{CKR23}}\label{prop:W11-chara-strong_ext_per}
     Let $(X,\sfd,\mm)$ be a complete separable metric measure space where $\mm$ is a nonnegative Borel measure finite on bounded sets. Then, an open set  $\Omega\subset X$ with $\mm(\partial \Omega)=0$ is a $W^{1,1}$-extension set if and only if there exists a constant $C>0$ so that for every set $E\subset \Omega$ of finite perimeter in $\Omega$ there exists a set $\widetilde E\subset X$ satisfying (PE1), (PE2) and so that $P_X(\widetilde E,\partial\Omega)=0$.
\end{theorem}
For closed sets, a characterization of the $W^{1,1}$-extension property via the strong extension of sets of finite perimeter does not hold. This is shown in Remark \ref{rem:stron_ext_per_not_possible} using the fat Sierpiński carpets introduced in Theorem \ref{thm:Sierp-W11_ext._char}.

Going back to the Euclidean setting, an application of \cite[Theorem 1.3]{BR21} for the case of planar bounded domains $\Omega\subset\R^2$ gives that (see \cite[Theorem 1.4]{BR21}) if $\Omega$ has the $BV$-extension property then $\Omega$ has the $W^{1,1}$-extension property if and only if the set $$\partial \Omega\setminus \bigcup_{i\in I}\overline \Omega_i$$ is purely $1$-unrectifiable, being $\{\Omega_i\}_{i\in I}$ the open connected components of $\R^2\setminus \overline\Omega$. This characterization, together with Proposition \ref{prop:(1,1)_to_purely_unrect}, allows us to establish Theorem \ref{prop:equiv_of_BV_and W11}.

\subsection{Outline of the paper}

The paper is organized as follows. Section \ref{sec:preliminaries} is dedicated to preliminaries. In Section \ref{sec:Siperpinski carpet} we deal with fat Sierpi\'nski carpets and in particular prove Theorem \ref{thm:Sierp-W11_ext._char}. In this section, we also provide two examples of closed and open sets that are $W^{1,p}$-extension sets for every $p>1$, but which are not $W^{1,1}$-extension sets. These are included in Propositions \ref{prop:W1p-ext-closed-NOT-W11} and \ref{prop:W1p-ext-dom-NOT-W11}.  In the next Section \ref{sec:example} we include Example \ref{ex:method_fails}, which shows that the techniques that work well for the fat Sierpi\'nski carpet case do not work in general. We end with Section \ref{sec:equiv_BV_W11}, which contains the proof of Theorem \ref{prop:equiv_of_BV_and W11}.

\section{Preliminaries}\label{sec:preliminaries}

We denote open balls of $\R^n$ by $B$ or $B(x,r)$ if we want to specify the center $x\in\R^n$ and radius $r>0$. Squares $Q\subset\R^2$ are just sets $Q=[a_1,a_2]\times [b_1,b_2]$, with $a_2-a_1 = b_2-b_1$, which are understood to be closed unless otherwise mentioned. By $Q(x,r)=[x_1-r/2,x_1+r/2]\times [x_2-r/2,[x_2+r/2]$ we denote the square with center $x=(x_1,x_2)\in \R^2$ and side-length $r>0$. The Lebesgue measure on $\R^n$ will be denoted  by $|\cdot|$. The $s$-dimensional Hausdorff measures, $s\geq 0$, are defined as $\mathcal H^s(A)=\lim_{\delta\to 0}\mathcal H^{s}_{\delta}(A),\; A\subset\R^n ,$
where $\mathcal H^{s}_{\delta}$ stands for the $s$-dimensional Hausdorff $\delta$-content of $A$ given by the formula
$$\mathcal H^{s}_{\delta}(A)=\inf\left\{ \sum^{\infty}_{i=1}\diam(U_i)^{s}:\, A\subset \bigcup^{\infty}_{i=1}U_i,\; \diam (U_i)\leq\delta\right\}. $$
Recall that on $\R^n$, the Lebesgue measure $|\cdot|$ and Hausdorff measure $\mathcal H^n$ are equivalent (see \cite{EG2015}). In general, in the rest of the paper, when working with subsets of $\R^n$ we will use the symbol $\Omega$ for open sets and $E$ for closed sets.

For a metric measure space $(X,\sfd,\mm)$, the integral average of an integrable function $u$ over some set $A\subset X$ is denoted by
$$u_A=\dfrac{1}{\mm(A)}\int_A u(x)\, d\mm=\fint_A u(x)\, d\mm. $$
When doing estimations we normally use $C>0$ to denote a constant, that may vary within a chain of inequalities and may depend on some parameters that arise from the context. We sometimes specify the dependence of those parameters by writing $C(\cdot)$ as a function. By $a\lesssim b$ we mean that $a \leq Cb$ for some constant $C \geq  1$. Similarly for $a\gtrsim b$. Then $a \sim b$ means that both $a\lesssim b$ and $a\gtrsim b$
hold.

We next define the space of Sobolev functions $W^{1,p}$ and the space of functions of bounded variation $BV$. In this work, these functions are defined exclusively on either open or closed subsets of $\R^n$. Consequently, we adopt two distinct approaches to introduce these spaces: a classical approach tailored to arbitrary open subsets $\Omega\subset\R^n$, and a more general framework for metric measure spaces $(X, \sfd, \mm) $. The latter can be applied to closed sets $E\subset \R^n$ by viewing $E$ as a metric measure space in its own right, endowed with the Euclidean distance and the restricted $n$-dimensional Lebesgue measure.

\subsection{Sobolev functions}

The Sobolev space $W^{1,p}(\Omega)$, being $\Omega\subset\R^n$ any open set, is defined as those $L^p$ integrable functions $u\colon\Omega\to\R$ whose weak derivatives $\nabla u\colon\Omega\to\R^n$ are also in $L^p$. 

The space, after identifying functions which are equal almost everywhere, is endowed with the norm
$$
\|u\|_{W^{1,1}(\Omega)}=\|u\|_{L^{p}(\Omega)}
+ \|\nabla u\|_{L^p(\Omega)} . $$
We assume %from granted that
the reader to be familiar with %masters 
the basic properties of these spaces (see \cite{EG2015} otherwise).

\medskip

If one wishes to define Sobolev functions on measurable subsets of 
$\R^n$, there now exists a vast body of literature addressing this problem. Most of the existing approaches are equivalent (see \cite{AILD24}) and here we will give the definition proposed by Shanmugalingam in 2000, based on upper gradients and which is referred as the Newtonian-Sobolev space. In her pioneering work \cite{S00} the notation $N^{1,p}$ is used, but here, in order to  unify the notation, we prefer to use $W^{1,p}$.

In what follows let $(X,\sfd,\mm)$ be a  metric measure space and $1\leq p<\infty$. A curve on $X$ is a continuous function $\gamma\colon[a,b]\to X$. Given a family of curves $\Gamma$, the $p$-modulus of $\Gamma$ is defined as
$$\text{Mod}_p(\Gamma)=\inf \left\{\int_{X} \rho(x)^p\, d\mm:\, \rho \text{ is nonnegative Borel function so that } \int_{\gamma}\rho\, ds\geq 1 \text{ for all } \gamma\in\Gamma  \right\}.$$
If a property fails only for a curve family with $p$-modulus zero, we say that it holds for $p$-almost every curve. We recall that integrals of functions $\rho\colon X\to [0,\infty]$ along rectifiable curves $\gamma\colon[a,b]\to X$ are defined as $$\int_{\gamma}\rho\, ds=\int^{\ell(\gamma)}_{0} \rho(\widehat \gamma(t))\, dt.$$
Here $\ell(\gamma)$ denotes the length of $\gamma$, that is,
\begin{equation}
    \label{eq:rectifiable}
    \ell( \gamma )
    =
    \sup¤%_{ \left\{ t_{i} \right\}_{ i = 1 }^{N+1} }
    \sum_{ i = 1 }^{ N } \dist( \gamma(t_i), \gamma( t_{i+1} ) ),
\end{equation}
where the supremum is taken over finite partitions $t_1 = a,$ $t_{i} \leq t_{i+1}$, $t_{N+1} = b$. And the function $\widehat{\gamma}:[0,\ell(\gamma)]\to X$ is the arc-length parametrization of $\gamma$, that is, the unique continuous map for which $\widehat{\gamma}( \ell( \gamma|_{ \left[a,t\right] } )) = \gamma( t )$; see \cite[Eq. (5.1.6)]{HKST15}.  Given a function $u\colon X\to\R$ we say that $\rho\colon X\to [0,\infty]$ is a $p$-weak upper gradient if $\rho$ is Borel and for $p$-almost every curve $\gamma\colon [0,1]\to X$ we have
$$|u(\gamma(1))-u(\gamma(0))|\leq \int_{\gamma}\rho\, ds. $$

\begin{definition}[$W^{1,p}$ spaces] For a measurable function $u\colon X\to\R$ let
$$ \|u\|_{ W^{1,p}(X)}=\|u\|_{L^p(X)}+\inf\left\{\|\rho\|_{L^p(X)}:\, \rho \;\text{is an}\; L^p\text{-integrable}\; p\text{-weak upper gradient of $u$} \right\}.$$
The Sobolev space $W^{1,p}(X)$ is defined as a quotient space
$$W^{1,p}(X)=\left\{u: \|u\|_{W^{1,p}(X)}<\infty
 \right\}/\sim ,\quad \quad \text{where $u\sim v$ if and only if $\|u-v\|_{W^{1,p}(X)}=0$.}$$

\end{definition}

It is well known (see \cite[Theorem 6.3.20]{HKST15}) that if a function $u$ has an $L^p$-integrable $p$-weak upper gradient then there exists a minimal $p$-weak upper gradient $\rho_u$. This means that $\rho_u$ is itself an $L^p$-integrable $p$-weak upper gradient and satisfies  $\rho_u\leq \rho$ almost everywhere for any other such gradient. We may therefore write
$$\|u\|_{W^{1,p}(X)}=\|u\|_{L^p(X)}+\|\rho_u\|_{L^p(X)}.$$
 We also define the homogenoeus Sobolev space $L^{1,p}(X)=\left\{u\in L^{1}_{loc}(X):\, \|\rho_u\|_{L^p(X)}<\infty\right\}$ together with the seminorm $\|\rho_u\|_{L^p(X)}$. 

We point out that the spaces $W^{1,1}(X)$ and $L^{1,1}(X)$ coincide with the classical definitions of Sobolev functions, by means of weak derivatives, whenever $X$ is an open subset of $\R^n$, setting $|\nabla u|=\rho_u$. For more information about Sobolev functions on metric measure spaces we refer to \cite{S00,BB11,HKST15,AILD24}.

\subsection{$BV$-functions}

For open sets $\Omega\subset\R^n$ the  definition of the space of bounded variation $BV(\Omega)$ is given as the set of all integrable functions $u\in L^1(\Omega)$ whose total variation
\begin{equation}\label{eq:def_BV_1}
 \|Du\|(\Omega)=\sup \left\{\int_{\Omega} u \,\text{div}(v)\, dx:\, v\in C^{\infty}_{0}(\Omega),\, |v|\leq 1  \right\} 
 \end{equation}
is finite. The norm of $u\in BV(\Omega)$ is given by $\|u\|_{L^1(\Omega)}+\|Du\|(\Omega)$. It is well known that $\|Du\|(\cdot)$ defines a Radon measure over all measurable subsets of $F \subset\Omega$ through the formula 
$$\|Du\|(F)=\inf\left\{ \|Du\|(U):\, F\subset U\subset \Omega,\, U \, \text{open}\right\} .$$
A measurable set $F\subset \Omega$ is said to have finite perimeter on $\Omega$ if $\|D\chi_F\|(\Omega)<\infty $. In this case, we define the perimeter of $F$ in $\Omega$ as $P(F,\Omega) = \|D\chi_F\|(\Omega)$; otherwise, we set $P(F,\Omega) = \infty$. For sets of finite perimeter $F\subset \Omega$ and for a given measurable subset $A\subset \Omega$ we can also understand the perimeter of $F$ on $A$ as the value $P(F,A)=\|D\chi_F\|(A) $. Thanks to De Giorgi \cite{Giorgi} and Federer  it is nowadays a standard fact that if $F\subset \Omega$ has finite perimeter on $\Omega$ then for every measurable subset $A\subset \Omega$ we have
\begin{equation}\label{eq:Per-H^n-1}
P(F,A)\sim\mathcal H^{n-1}(\partial^M F\cap A).
\end{equation}
Here $\partial^M F$ denotes the  measure theoretic boundary  of  $F$  (also called essential boundary), that is 
$$\partial^M F=\left\{x\in \R^n:\, \limsup_{r\to 0}\dfrac{|F\cap B(x,r)|}{|B(x,r)|}>0\;\text{and}\;\limsup_{r\to 0}\dfrac{|F\setminus B(x,r)|}{|B(x,r)|}>0\right\}.$$

There exist alternative, but equivalent, definitions of the space of $BV$ functions that are more convenient when working with arbitrary measurable subsets of $\R^n$.
Following the approach of \cite{A-DM14}, for a metric measure space $(X,\sfd,\mm)$, for a function $u\in L^{1}_{loc}(X)$ and for an open subset $U\subset X$ we define
\begin{equation}\label{eq:def_BV_2} 
|Du|(U)= \inf\left\{ \liminf_{n\to\infty}\int_{U}  \text{lip} f_n\, d\mm:\, f_n\in \Lip_{\text{loc}}(U),\, \|f_n- u\|_{L^{1}_{loc}(U)}\to 0\right\}. 
\end{equation}
 Here $\Lip_{\text{loc}}(U)$ denotes the space of locally Lipschitz functions on $U$, which are those functions $f\colon U\to\R$, for which the pointwise Lipschitz constant (also called slope) is finite everywhere on every accumulation point $x\in U$, that is,
 
$$\text{lip}f(x)=\limsup_{y\to x} \dfrac{|f(y)-f(x)|}{|y-x|}<\infty.$$
The space of Lipschitz functions over a set $S\subset X$ is denoted by $\Lip(S)$.
We can also extend $|Du|$ to all Borel subsets $B\subset X$, hence defining a Borel measure \cite[Theorem 3.4]{Mir03}, by letting
$$|Du|(B)=\inf \{ |Du|(U):\, B\subset U\subset X,\; U \,\text{open} \} .$$
For a Borel set $B\subset X$ and $u\in L^{1}_{\loc}(B)$ we denote by $|Du|_B$ the total variation when computed in the metric measure space $(X,\sfd,\mm|_B)$.

 \begin{definition}[$BV$ spaces]
   For a Borel set $B\subset X$, the space $BV(B)$ is the set of those functions $u\in L^1(B)$ for which their total variation $|Du|_B(B)$ is finite. We endow the space with the norm
     $$\|u\|_{BV(B)}=\|u\|_{L^1(B)}+|Du|_B(B) ,$$
     and as usual we identify functions $u,v\in BV(B)$ so that $\|u-v\|_{BV(B)}=0$.
    We also define the homogeneous space $$\overset{\circ}{BV}(B)=\{u\in L^{1}_{\loc}(B):\, |Du|_{B}(B)<\infty\} $$
     together with the seminorm $|Du|_{B}(B)$.
 \end{definition}
We always have that $W^{1,1}(X)\subset BV(X)$  with $\|\rho_u\|_{L^1(X)}\sim |Du|_X(X)$ for every $u\in W^{1,1}(X)$; see \cite[Section 8]{A-DM14}. 

In the Euclidean situation of $X=\R^n$ with the Euclidean distance and the $n$-dimensional Lebesgue measure, we note that whenever $E\subset X$ is closed, we have $BV (X,\sfd,\mm|_E) =BV(E,\sfd|_{E\times E},\mm|_E)$. Moreover, for an open set $\Omega\subset X= \R^n$ and $u\in L^{1}_{\loc}(\Omega)$ it can be proved that (see \cite{AFP2000,Mir03,A-DM14})  \begin{equation}\label{eq:Variation_equality}
\|Du\|(U)=|Du|_\Omega(U) =|Du|(U) \end{equation} 
for every open subset $U\subset\Omega$. Consequently, the resulting $BV(\Omega)$ spaces are identical regardless of whether \eqref{eq:def_BV_1} or \eqref{eq:def_BV_2} is used to define the total variation.

In the context of metric measure spaces $X$, we define sets of finite perimeter as follows. 

\begin{definition}[Sets of finite perimeter]
Let $E\subset X$ be a Borel subset of a metric measure space. A  Borel set $F\subset E$ has finite perimeter on $E$ if $|D \chi_F|_E(E)<\infty$, writing in this case $P_E(F)=|D \chi_F|_E(E)$ and $P_E(F)=\infty$ otherwise. Moreover, for any Borel subset $B\subset E$ we denote $P_E(F,B)=|D\chi_F|_E(B)$. 
\end{definition}
Observe that, from \eqref{eq:Variation_equality}, for open subsets $\Omega\subset\R^n$ we have $P_\Omega(F)=P(F,\Omega)$ for every Borel subset $F\subset\Omega$ (recall that $P(F,\Omega)=\|D\chi_F\|(\Omega)$ should be understood in the sense of \eqref{eq:def_BV_1}).  Similarly, if $B\subset\Omega$ is another Borel set we also have $P_\Omega(F,B)=P(F,B)=\|D\chi_F\|( B)$.

\begin{remark}\label{rem:stron_ext_per_not_possible} We show here that Proposition \ref{prop:W11-chara-strong_ext_per} does not hold for closed sets. The fat Sierpi\'nski carpets $S\subset \R^2$ considered in Theorem \ref{thm:Sierp-W11_ext._char} are closed $W^{1,1}$-extension sets which do not satisfy the strong extension property for sets of finite perimeter. Indeed, for these sets we have $S=\partial S$ and $|S|>0$. Next, take a set $F\subset S$ of finite perimeter on $E$ with $P_S(F)>0$, and consider any extension $\widetilde F$ for which (PE1) holds. In that case we have

   $$ P_{\R^2}(\widetilde F, \partial S)=P_{\R^2}(\widetilde F, S)=\left|D \chi_{\widetilde F}\right|_{\R^2}(S)\geq \left|D\chi_{\widetilde F\cap S}\right|_S(S)=P_S(\widetilde F\cap S).$$
The inequality $|D \chi_{\widetilde F}|_{\R^2}(S)\geq |D\chi_{\widetilde F\cap S}|_S(S) $ follows because, since $S$ is closed, globally Lipschitz functions can be used instead of locally Lipschitz ones to compute the total variation. Next, since (PE1) gives that $|( F\setminus (\widetilde F\cap E))\cup((\widetilde F\cap E)\setminus F)|=0 $, by using  \cite[Proposition 2.8]{CKR23} we get 
$ P_S(\widetilde F\cap S)=P_S(F)$ and hence
   $$P_{\R^2}(\widetilde F, \partial S)\geq P_S(\widetilde F\cap S) = P_S(F)>0.$$
\end{remark}
We refer the interested reader to \cite{AFP2000,Mir03,A-DM14,Panu} for more information about $BV$-functions. For the purposes of this paper, we note that although $W^{1,1}(X)$ and $BV(X)$ have been defined for general metric measure spaces, in what follows we shall only consider $X$  as a closed or open subset of $\R^n$, with the Euclidean distance and the restricted $n$-dimensional Lebesgue measure.

\subsection{Poincaré inequality}

\begin{definition}\label{def:poincare}
    A metric measure space $(X,\sfd,\mm)$ is said to satisfy a weak $(1,p)$-Poincaré inequality for some $1\leq p<\infty$ if there exist constants $C>0$, $\lambda\geq 1$ so that
    \begin{equation}\label{eq:(1,p)-Poinc-ineq.}
    %\dfrac{1}{|B\cap X|}
    \fint_{B}|u(x)-u_{B}|\, d\mm \leq C \diam (B)\left(%\dfrac{1}{| \lambda B \cap X|}
    \fint_{\lambda B}\rho(x)^p\, d\mm\right)^{1/p} 
    \end{equation}
    for every ball $B\subset X$, every integrable function $u\colon\lambda B\to \R$ and every $L^p$-integrable $p$-weak upper gradient $\rho\colon\lambda B\to [0,\infty]$ of $u$. 
\end{definition}
In the case that $X$ satisfies a weak $(1,1)$-Poincaré inequality we also get that there exists $C>0$ and $\lambda \geq 1$ so that 
\begin{equation}\label{eq:BV-PI}
  \fint_{B} |u(x)-u_{B}|\, d\mm\leq C\diam (B) \dfrac{|D u|(\lambda B)}{\mm(\lambda B)}
\end{equation}
for every ball $B$ centered at $X$ and every $u\in L^{1}_{\loc}(X)$. This fact follows by applying \cite[Theorem 4.21]{BB11} and the weak $(1,1)$-Poincaré inequality to approximating Lipschitz functions in the definition of the total variation. Note that for locally Lipschitz functions $f\colon X\to\R$, the slope $\text{lip}f(x)$ acts as a $1$-weak upper gradient.

The next lemma is a straightforward consequence of the weak $(1,1)$-Poincaré inequality (a similar statement can be found in \cite[Theorem 4.5]{Mir03}). This result will be used in the proof of Proposition \ref{prop:(1,1)_to_purely_unrect}. We provide details of the proof for the sake of completeness.

\begin{lemma}\label{lem:Iso_per.}
    Let $\Omega\subset\R^n$ be an open set that satisfies a weak $(1,1)$-Poincaré inequality with constants $(C_p,\lambda_p)$. Also let $x\in \Omega$, some radius $r>0$ and some measurable set $F\subset \Omega$ with finite perimeter and so that 
    $$|B(x,r)\cap F|\geq \delta r^n\quad \text{and}\quad |B(x, r)\cap (\Omega\setminus F)|\geq \delta r^n $$
    for some constant $\delta\in (0,1)$. Then, there exists a constant $C=C(C_p)>0$ such that
    $$P(F, B(x,\lambda_p r)\cap \Omega)\geq C \delta r^{n-1} .$$
\end{lemma}

\begin{proof}
    We start by letting $u=\chi_F\in BV(\Omega)$. Then by the weak $(1,1)$-Poincaré inequality \eqref{eq:BV-PI}
    \begin{equation}\label{eq:lem_key_tool}
    \fint_{B(x,r)\cap \Omega} |u(y)-u_{B(x,r)\cap \Omega}|\, dy\leq C_p r \dfrac{\|Du\|(B(x,\lambda_p r)\cap \Omega)}{|B(x,\lambda_p r)\cap \Omega|}
    \end{equation}
Observe that $\|Du\|(B(x,\lambda_p r)\cap \Omega)=P(F,B(x,\lambda_p r)\cap \Omega)$. Moreover since $\lambda_p\geq 1$ we have $|B(x,\lambda_p r)\cap \Omega|\geq |B(x,r)\cap \Omega|$. Therefore,  inequality \eqref{eq:lem_key_tool} yields 
\begin{equation}\label{eq:lem_key_tool_2}
           \int_{B(x,r)\cap \Omega} |u(y)-u_{B(x,r)\cap \Omega}|\, dy\leq C_p r P(F,B(x,\lambda_p r)\cap \Omega).
\end{equation}
We will now show that the left-hand side term can be bounded from below by some constant times the $n$-th power of the radius $r$. Let us explain this. First
$$u_{B(x,r)\cap \Omega}=\fint_{B(x,r)\cap \Omega}u(y)\, dy=\dfrac{|B(x,r)\cap F|}{|B(x,r)\cap \Omega|}\in (0,1] ,$$
so, noting that we have $u=1$ on $B(x,r)\cap F $ and $u=0$ on $B(x,r)\cap (\Omega\setminus F)$, we can write
\begin{align}
    \int_{B(x,r)\cap \Omega} |u(y)-u_{B(x,r)\cap \Omega}|\, dy&\geq \max\left\{ \int_{B(x,r)\cap F} u_{B(x,r)\cap \Omega},\int_{B(x,r)\cap (\Omega\setminus F)}\left|1-u_{B(x,r)\cap \Omega}\right|  \right\}\\
&\geq \delta r^n\max \left\{ u_{B(x,r)\cap \Omega} , \left|1-u_{B(x,r)\cap \Omega}\right|  \right\}\geq \dfrac{\delta}{2}r^n.
\end{align}
Finally, going back to \eqref{eq:lem_key_tool_2} we have
$$ \dfrac{\delta}{2}r^n \leq C_p r P(F,B(x,\lambda_p r)\cap \Omega)$$
and hence we conclude
$$P(F,B(x,\lambda_p r)\cap \Omega)\geq C(C_p)\delta r^{n-1} .$$
    
\end{proof}

\subsection{Extension sets}

\begin{definition}
Given a metric measure space $X$ and a Borel set $E\subset X$, we say that $E$ has the $W^{1,p}$-extension property if there exists an extension operator $T\colon W^{1,p}(E)\to W^{1,p}(X)$ and a constant $C>0$ such that $Tu|_E=u$ and $\|Tu\|_{W^{1,p}(X)}\leq C\|u\|_{W^{1,p}(E)}$ for all $u\in W^{1,p}(E)$. 
\end{definition}
This definition extends naturally to $L^{1,p}$, $BV$, and $\overset{\circ}{BV}$ by substituting the $W^{1,p}$ norm with the appropriate norm or seminorm for each space.

For the Euclidean versions of the next result (where the converses also hold) we refer to \cite{Koskela} (or \cite[Theorem 4.4]{HerronKoskela}) and \cite[Lemma 2.1]{KMS2010}.
\begin{lemma}\label{lem:hom-full-norm}
 Let $(X,\sfd,\mm)$ be an Ahlfors regular metric measure space that supports a weak $(1,1)$-Poincaré inequality, and let $E\subset X$ be a bounded Borel set.
    \begin{enumerate}
        \item[(a)] If $E$ is a $L^{1,1}$-extension set, then it is a $W^{1,1}$-extension set.
        \item[(b)] If $E$ is a $\overset{\circ}{BV}$-extension set, then it is a $BV$-extension set.
    \end{enumerate}
\end{lemma} \begin{proof}
    The proof of $(b)$ appears in \cite[Proposition 4.5]{CKLR26}. Let us show $(a)$, which may be a well-known fact for most experts in the area, but it seems that a proof is missing in the literature. We follow the ideas of \cite{HerronKoskela} and \cite{CKLR26}.

Let $T\colon L^{1,1}(E)\to L^{1,1}(X)$ be the extension operator, and let $B(x,r)$ be a ball such that $E\subset B(x,r)$, where $r=\text{diam}(E)$. Next, use \cite{Rajala2021} to find a uniform set $B(x,r)\subset \Omega \subset B(x,r+1)$, which we know from \cite[Proposition 5.9]{Bjo:Sha:07} that is a $W^{1,1}$-extension set. Therefore, it is enough to show that $T\colon W^{1,1}(E) \to W^{1,1}(\Omega)$ is also an extension operator. Call $B:=B(x,r+1)$ and let $u\in W^{1,1}(E)$. We know that $Tu\in L^{1,1}(X)$ so it is clear that $Tu\in L^{1,1}(B)$. By the weak $(1,1)$-Poincaré inequality  \eqref{eq:(1,p)-Poinc-ineq.}, and due to the Alhfors regularity, we can write
\begin{align*}
    \int_{B}|Tu-(Tu)_{B}|d\mm &\leq C \text{diam}(B)\int_{\lambda B} \rho(x)\, d\mm. 
\end{align*}
Here, $\rho$ denotes the minimal upper gradient of $Tu$ on $X$, and note that the constant $C$ depends on  the Poincaré constant $\lambda$ and on the Alhfors regularity constant. Therefore, since $Tu|_E=u$ we have
\begin{align*}
    \fint_{B}|Tu-u_{E}|\, d\mm &\leq \fint_{B}|Tu-(Tu)_{B}|\, d\mm+ \mm(B)|(Tu)_{B}-u_{E}|   \\
    &= \fint_{B}|Tu-(Tu)_{B}|d\mm + \left|\dfrac{\mm(B)}{\mm(E)}\int_E (Tu-(Tu)_{B})\, d\mm\right|\\
    &\leq 2\dfrac{\mm(B)}{\mm(E)}\int_B |Tu-(Tu)_{B}|\, d\mm \leq 2C\dfrac{\mm(B)\text{diam(B)}}{\mm(E)}\int_B \rho_u\, d\mm\leq C' \int_B \rho\, d\mm.
\end{align*}
The constant $C'$ depends on $C$, on $\diam(E)$, on $\mm(E)$ and on the Alhfors regularity exponent $s$ (recall also that $\diam (B)\leq 2\diam (E)+2$). This leads to the estimate
\begin{align*}
   \int_\Omega |Tu|\,d\mm \leq \int_B |Tu|\, d\mm \leq \int_B |Tu-u_E|\, d\mm+\int_B |u_E|\, d\mm\leq C'\int_B \rho\, d\mm+ \mm(B) \, |u_E|.
\end{align*}
We conclude that for some constant $C''\geq 1$
\begin{align*}
 \|Tu\|_{W^{1,1}(\Omega)}&\leq  \int_\Omega |Tu|\,d\mm + \int_B \rho\, d\mm \leq (C'+1)\int_B\rho\,d\mm+\mm(B)|u_E|\\
 & \leq (C'+1)\int_X\rho\,d\mm+\dfrac{\mm(B)}{\mm(E)}\int_E |u|\, d\mm \\
 &\leq \left( (C'+1)\|T\|+ \dfrac{\mm(B)}{\mm(E)} \right)\|u\|_{W^{1,1}(E)}\leq C'' \|u\|_{W^{1,1}(E)}.
\end{align*}

\end{proof}

\subsection{Other geometric notions}

\begin{definition}[Quasiconvexity]
    A set $X\subset \R^n$ is quasiconvex (or $C$-quasiconvex) if there is $C\geq 1$ so that every pair of points $x,y\in X$ can be joined by a rectifiable curve $\gamma\subset X$ such that $\ell(\gamma)\leq C|x-y|$.
\end{definition}

\begin{definition}[Uniform sets]\label{def:unif}
\hspace{1cm}
    \begin{enumerate}
        \item (Martio and Sarvas, \cite{MS78}). A domain $\Omega \subset \R^n$ is uniform (or $C$-uniform) if there is $C\geq 1$  so that for every $x,y\in\Omega$ there exists a rectifiable curve $\gamma\subset \Omega$ connecting them such that
        $$ \ell(\gamma)\leq C |x-y| \quad \text{and} \quad \dist(z,\partial \Omega)\geq C^{-1}\min(\ell(\gamma_{x,z}),\ell(\gamma_{y,z}) )\;\text{for all}\; z\in\gamma.$$
        Here $\gamma_{x,y},\gamma_{y,z}\subset \gamma$ are the subcurves of $\gamma$ joining $x$ with $y$ and $y$ with $z$ respectively.
        A planar domain $\Omega\subset\R^2$ that is uniform in the previous sense is a quasidisk (see \cite{Gehring87}).
\item (Eriksson-Bique, \cite[Definition 2.7]{EG22}). A closed set $E\subset \R^n$ is uniform (or $C$-uniform) if there is $C\geq 1$ so that for every $x,y\in E$ there exists a  curve $\gamma\subset E$ connecting them such that
        \begin{equation}\label{eq:uniform_1}
    \diam(\gamma)\leq C |x-y| \end{equation} 
        and  for all $z\in\gamma$
 \begin{equation}\label{eq:uniform_2}   
        \dist(z,\partial E)\geq C^{-1}\min(\diam(\gamma_{x,z}),\diam(\gamma_{y,z}) ).
        \end{equation}
We remark that if $E$ is Ahlfors regular and $E=\overline{\text{int}(E)}$, $E$ is uniform if and only if $\text{int}(E)$ is uniform in the sense of Martio and Sarvas. Moreover, in the case that $E$ is  quasiconvex we have $\ell(\gamma)\sim \diam (\gamma)$ and we could write lengths instead of diameters in \eqref{eq:uniform_2}.
    \end{enumerate}
\end{definition}

\begin{definition}[Purely unrectifiable sets]
    A set $H\subset \R^n$ is called purely $(n-1)$-unrectifiable if for every Lipschitz function $f\colon\R^{n-1}\to \R^n$ we have that $\mathcal H^1(f(\R^{n-1})\cap H)=0$. 
\end{definition}

\section{Fat Sierpi\'nski carpets and Sobolev extensions}\label{sec:Siperpinski carpet}

In \cite{Sierpinski} it was shown that fat enough Sierpi\'nski carpets 
 support a weak $(1,1)$-Poincar\'e inequality. We will verify that such carpets are also $W^{1,1}$-extension sets. Notice that for $p>1$, assuming the Sierpi\'nski carpet supports a weak $(1,p)$-Poincaré inequality, it would be easier to conclude that the carpet is a $W^{1,p}$-extension set thanks to Proposition \ref{prop:plargerthan1}. In the case $p=1$, we cannot use a simple maximal operator argument for an abstract Whitney-type extension. Instead, we will construct an explicit Whitney-type extension operator.

 \medskip

 Let us recall first the construction of a fat Sierpi\'nski carpet. We mostly follow the notation of \cite{Sierpinski}.
 The Sierpi\'nski carpets are parametrized by a sequence $\mathbf{a} = (a_1,a_2,\dots)$ of reciprocals of odd integers that are at least three. We start the iterative construction by defining $T_0 =[0,1]^2$ and $\mathcal S_{\mathbf a, 0} = T_0$. Next we cover $\mathcal S_{\mathbf a, 0}$ by $a_1^{-2}$ essentially disjoint squares of side-length $a_1$. The  open central  square inside $S_{\mathbf a, 0}$ of side-length $a_1$ is called $\mathcal R_{\mathbf a,1}$, which will be removed. The collection of the remaining squares (excluding the central one) is denoted  by $\mathcal T_{\mathbf a,1}$. And the union of all the squares in $\mathcal T_{\mathbf a,1}$ is denoted by $\mathcal S_{\mathbf a,1}$. Each $Q \in \mathcal T_{\mathbf a,1}$ is covered by $a_2^{-2}$ essentially disjoint squares of side-length $a_1a_2$ and the interior of the central one of those is deleted. The collection of all the squares obtained this way is denoted by $\mathcal T_{\mathbf a,2}$ and their union by $\mathcal S_{\mathbf a,2}$. The collection of all removed central open squares in this iteration is named $\mathcal R_{\mathbf a,2} $. The process is then continued iteratively.   
In particular, at step $i$ suppose we have divided 
 every $Q\in \mathcal T_{\mathbf a, i-1}$ into $a^{-2}_{i}$ essentially disjoint squares of side-length $a_1,\dots,a_i$. 
 \begin{itemize}
\item The collection of all removed open central  squares inside every square of $T_{\mathbf a, i-1}$ is named $\mathcal R_{\mathbf a,i}$. 
\item The collection of the remaining squares (excluding the central one) is denoted  by $\mathcal T_{\mathbf a,i}$. The number of essentially disjoint closed squares in $\mathcal T_{\mathbf a,i}$ is $ \prod^{i}_{j=1}(a^{-2}_{j}-1)$.
\item The union of all the squares in $\mathcal T_{\mathbf a,i}$ is denoted by $\mathcal S_{\mathbf a,i}$.
 \end{itemize}
Finally, we define the Sierpi\'nski carpet associated to the sequence $\mathbf a$ by
 \[
  \mathcal S_{\mathbf a} = \bigcap_{n\in\N}\mathcal S_{\mathbf a,n}=[0,1]^2\setminus \bigcup_{n\in\N}
\mathcal \bigcup_{Q\in \mathcal R_{\mathbf a,n}}Q. 
\]
Note that every $Q\in \mathcal R_{\mathbf a,n}$ is considered to be open and we have $\ell(Q)=a_1\cdots a_n$. One may check that $\mathcal S_{\mathbf a} $ is a compact connected set with empty interior such that $\partial \mathcal S_{\mathbf a}= \mathcal S_{\mathbf a}$. Furthermore, it is known from \cite{Sierpinski} that  $\mathcal S_{\mathbf a}$ has Hausdorff dimension equal to $2$ if $\mathbf a\in c_0$ and it has positive measure if and only if $ \mathbf a\in\ell_2$, being moreover Ahlfors regular in the later case. By \cite[Theorem 1.5]{Sierpinski}, the Sierpi\'nski carpet $\mathcal S_{\mathbf a}$ when equipped with the induced Euclidean distance and the restriction of the Lebesgue measure to $\mathcal S_{\mathbf a}$ satisfies a weak $(1,1)$-Poincar\'e inequality if and only if $\mathbf a \in \ell^1$.

\begin{figure}

\begin{center} \includegraphics[width=0.35\linewidth]{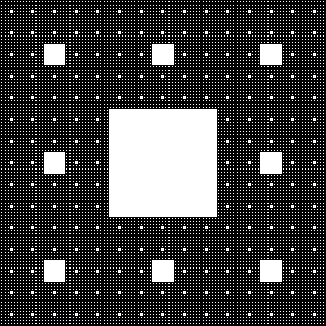}\end{center}\caption{ Sierpi\'nski carpet $\mathcal S_{\mathbf a}$ with $\mathbf a=\left(\frac{1}{3},\frac{1}{5},\frac{1}{7},\dots\right)$}
  \label{carpet:intro}

\end{figure}

\medskip

The main goal of this section is to give the proof of Theorem \ref{thm:Sierp-W11_ext._char}. For that, the difficult implication is to prove that if $\mathcal S_\mathbf a$ satisfies a weak $(1,1)$-Poincaré inequality then it is a $W^{1,1}$-extension set. This is what the next theorem states. As mentioned above, according to \cite[Theorem 1.5]{Sierpinski}, $\mathcal S_{\mathbf a}$ has a weak $(1,1)$-Poincaré inequality if and only if $\mathbf a\in \ell_1$.

\begin{theorem}\label{thm:Sierp-W11_ext.}
    Let $\mathbf{a}=(a_n)_{n\geq 1}
\in\ell_1 $ with $a_n^{-1}\geq 3$ being odd numbers for every $n\in\mathbb N$ and  let $\mathcal S_{\mathbf a}$ be the corresponding fat Sierpi\'nski carpet. Then there exists a linear extension operator $T\colon W^{1,1}(\mathcal S_{\mathbf a})\to W^{1,1}(\R^2)$ and some $C>0$ so that
$\| Tu \|_{W^{1,1}(\R^2)} \leq C\| u \|_{W^{1,1}(\mathcal S_{\mathbf a})} $ for every $ u\in W^{1,1}(\mathcal S_{\mathbf a}).$ 
\end{theorem}
Before proceeding to the proof, we first establish three preliminary lemmas: Lemmas \ref{lem:modif_Sierpinski_preserve_properties}, \ref{lem:sierpinski_ext_op.}, and \ref{lem:sierpinski_ext_op._2}.

\medskip

Let us fix from now on  $\mathbf a=(a_1,a_2,\dots)\in\ell_1$, where the sequence is formed by reciprocals of odd numbers bigger than three.
For our purposes we want to make sure that certain specific subsets of $\mathcal S_{\mathbf a}$ still keep the properties of being Ahlfors regular and satisfying a weak $(1,1)$-Poincaré inequality. Namely, for a given $ \mathbf a\in\ell_1$ take $k_{\mathbf a}\in\N$ so that
\begin{equation}\label{eq:Sierpinki_1}
a_{n}^{-1}\geq 5,\quad \text{for all}\; \; n\geq k_{\mathbf a}. \end{equation}
For any $n\geq k_{\mathbf a}$ and any  square $Q\in \mathcal R_{\mathbf a,n}$ let us define the compact subset $S_Q\subset \mathcal S_{\mathbf a}$ as
\begin{equation}\label{eq:Sierpinski_2}
\mathcal S_Q=(3\overline{Q}\setminus Q)\setminus \left(\bigcup_{m>n}\bigcup_{R\in \mathcal R_{\mathbf a,m}} 3R \right).
\end{equation}
It is important to note that $S_Q\cap S_{Q'}=\emptyset$ whenever $Q,Q'\in \mathcal R_{\mathbf a,n}$, $n\geq k_{\mathbf a}$ and $Q\neq Q'$. Another key point for us is that the sets $S_Q$ are Ahlfors regular and satisfy the $(1,1)$-Poincaré inequality, a fact that will be proved in Lemma \ref{lem:modif_Sierpinski_preserve_properties} with the help of the following result from \cite{EG22}.

\begin{theorem}{\cite[Theorem 1.3]{EG22}}\label{thm:Sylvester_22}
    Let $E\subset\R^2$ be a non-empty compact set and $\mathbf{b}=(b_k)_{k\geq 1}$ a sequence of reciprocals of positive integers. Let also $s_0=\diam (E)$ and $s_k=b_k s_{k-1}$ for $k\geq 1$. Assume $\{\mathcal R_{\mathbf b, k}\}^{\infty}_{k=1} \subset E$ is a collection of domains for which there exists constants $\delta\in (0,1)$, $L>0$ so that for each $R\in\mathcal R_{\mathbf b,k}$:
    \begin{enumerate}
        \item $\partial R$ is connected, and the sets $\R^2\setminus R$ and $E$ are uniform. 
        \item $\diam (R)\leq L s_k$.
        \item $\text{\em dist} (R,E^c)\geq \delta s_{k-1}$.
        \item For all $\ell\leq k$ and $R'\in \mathcal R_{\mathbf b,\ell}$ then $\text{\em dist}(R,R')\geq \delta s_{k-1}$.
    \item For every $k\in\N$, if $r\geq s_{k-1}$ then for all $x\in E$,
    $$\mathcal H^1\left(  \pi_i\left(  B(x,r)\cap \bigcup_{R\in\mathcal R_{\mathbf b,k}}R \right) \right)\leq L rb_k,\quad \text{for all $i=1,2$.} $$
    \end{enumerate}
    Then, if $\mathbf b \in \ell_1$, the set $E\setminus \bigcup_{k\in\N}\bigcup_{R\in \mathcal R_{\mathbf b,k}}R$ satisfies a weak $(1,1)$-Poincaré inequality. 
\end{theorem}

\begin{lemma}\label{lem:modif_Sierpinski_preserve_properties}
   For any fixed $n\geq k_{\mathbf a}$ and $Q\in \mathcal R_{\mathbf a,n}$, the set $S_Q$ is Ahlfors regular and satisfies a weak $(1,1)$-Poincaré inequality (with constants independent of $n$ and $Q$).
\end{lemma}
\begin{proof}

Let us first  verify the weak $(1,1)$-Poincaré inequality using Theorem \ref{thm:Sylvester_22}. In our case we set  $E=3\overline Q\setminus Q$,  define the subsequence $\mathbf b=(b_k)_{k\geq 1}=(a_{n+k})_{k\geq 1}$ and let the collection
$\mathcal R_{\mathbf b,k}=\mathcal R_{\mathbf a,n+k} $ for every $k\in\N$. Then, we can write $S_Q$ as
$$S_Q=(3\overline{Q}\setminus Q)\setminus\left( \bigcup_{m>n}\bigcup_{R\in \mathcal R_{\mathbf a,m}} 3R\right)=E\setminus  \bigcup_{k\in\N}\bigcup_{R\in \mathcal R_{\mathbf b,k}} 3R  .$$
Let us check that the assumptions $(1)-(5)$ of Theorem \ref{thm:Sylvester_22} are satisfied with the constants
$$L=7 \quad;\quad  \delta=\dfrac{1}{12} . $$
First, it is easy to check that $\partial (3R)$ is connected and that both $\R^2\setminus 3R$ and $E$ are uniform sets. Hence $(1)$ is satisfied. Second, we have $s_0=\diam (3\overline{Q})=3\sqrt{2}a_1\cdots a_n$ and $s_k=b_k s_{k-1}=s_0b_1\cdots b_k=3\sqrt{2}a_1\dots a_{n+k}$. 
Therefore, for every $R\in \mathcal R_{\mathbf b,k}=\mathcal R_{\mathbf a,n+k}$
$$\diam(3R)=3\sqrt{2}\ell(R)=3\sqrt{2}a_1\cdots a_{n+k}=s_k ,$$
so we have proved property $(2)$. Moreover, using that  $a^{-1}_{n+k}\geq 5 $ for every $k\geq 1$ we have that $a^{-1}_{n+k}-3\geq a^{-1}_{n+k}/2 $ and then

$$  \dist\left( 3R, 3\overline{Q}^c\right)\geq \dfrac{a^{-1}_{n+k}-3}{2}\sqrt{2}\ell(R) \geq \dfrac{a^{-1}_{n+k}}{4}\sqrt{2}\ell(R)=\dfrac{a_1\cdots a_{n+k-1}}{2\sqrt{2}}=\dfrac{1}{12}s_{k-1}. $$
This shows $(3)$. The same argument shows $(4)$. It only remains to check $(5)$. By symmetry, let us only explain the proof for the projection onto the $x$-coordinate axis. Take $k\in\N$, $r\geq s_{k-1}$ and $x\in 3 \overline Q$. Observe that 
$$\mathcal H^1\left(  \pi_1\left(  B(x,r)\cap \bigcup_{R\in\mathcal R_{\mathbf b,k}}3R \right) \right)\leq \mathcal H^1\left(  \pi_1(B(x,r)) \cap \pi_1\left( \bigcup_{R\in\mathcal R_{\mathbf b,k}}3R \right) \right) .$$
Call $\widetilde s$ the unique positive integer so that 
$$(\widetilde s -1)a_1\cdots a_{n+k-1}\leq 2r\leq \widetilde s a_1\cdots a_{n+k-1}.$$
Then we have 
\begin{equation}\label{eq:siep_modified_poincare_1}
    \widetilde s\leq \dfrac{2r}{a_1\cdots a_{n+k-1}}+1.
\end{equation}
On one hand, note that $\pi_1(B(x,r))$ is just an interval of length $2r$ and so in particular  $ \mathcal H^1(\pi_1(B(x,r)))\leq \widetilde s a_1\cdots a_{n+k-1} $. On the other hand, $ \pi_1(\bigcup_{R\in\mathcal R_{\mathbf b,k}}3R)$ is a disjoint union of intervals of length  $3a_1\cdots a_{n+k}$ centered at some points which are $(a_1\cdots a_{n+k-1})$-separated. We can then  assert that 
$$ \mathcal H^1\left(  \pi_1(B(x,r)) \cap \pi_1\left( \bigcup_{R\in\mathcal R_{\mathbf b,k}}3R \right) \right)\leq 3\widetilde s (a_1\cdots a_{n+k}). $$
By using \eqref{eq:siep_modified_poincare_1} and the fact that $r\geq s_{k-1}=3\sqrt{2}a_1\cdots a_{n+k-1}$ we conclude that
\begin{align*}
\mathcal H^1\left(  \pi_1\left(  B(x,r)\cap \bigcup_{R\in\mathcal R_{\mathbf b,k}}3R \right) \right)&\leq 3\widetilde s (a_1\cdots a_{n+k})\leq 6ra_{n+k}+3a_1\cdots a_{n+k}\\
&\leq 6ra_{n+k}+\dfrac{3r}{3\sqrt{2}}a_{n+k}\leq 7ra_{n+k}=7rb_k. 
\end{align*}
Continuing with the proof of Lemma \ref{lem:modif_Sierpinski_preserve_properties}, for the Ahlfors regularity of $S_Q$ we can argue as in \cite[Lemma 2.22]{EG22}, once we know that $S_Q$ lies under the assumptions of Theorem \ref{thm:Sylvester_22}. The Ahlfors regularity constant of $S_Q$ will depend only on the sequence $\mathbf b$, hence on $\mathbf a$. For completeness we provide the details. Recall that $Q\in \mathcal R_{\mathbf a,n}$ and hence $\ell(Q)=a_1\cdots a_n$. Noting that $\sum^{\infty}_{k=1}a^{2}_{k}<\sum^{\infty}_{k=1}a_k<\infty$
we pick $k_0\geq n\geq  k_{\mathbf a}$ so that 
\begin{equation}\label{eq:Ahlfors-S_Q_0}
    \sum^{\infty}_{k=k_0}a^{2}_{k}< \dfrac{M}{4608}.
\end{equation}
Here $M>0$ is chosen so that for every $x\in S_Q$, $r\in (0,\diam (S_Q))$, $R\in\mathcal R_{\mathbf a,k}$ and $k\geq n$ we have $$|B(x,r)\cap ((3\overline Q \setminus Q)\setminus 3R)|\geq M |B(x,r)|.$$
For instance one may take $M=1/100$.
We will prove that for all $x\in S_Q$ and $r\in (0,a_1\cdots a_{k_0-1}/4\sqrt{2})$  
\begin{equation}\label{eq:Ahlfors-S_Q}
|B(x,r)\cap S_Q|\geq C |B(x,r)|. 
\end{equation}
We assert that this is enough to get the same estimate \eqref{eq:Ahlfors-S_Q} holding for every $x\in S_Q$ and every $r\in (0,\diam (S_Q))$. Indeed, whenever $r\in [a_1\cdots a_{k_0-1}/4\sqrt{2},\diam (S_Q))$ we can write
\begin{align*}
|B(x,r)\cap S_Q| &\geq |B(x,a_1\cdots a_{k_0-1}/8\sqrt{2}  )\cap S_Q|\geq C |B(0,1)| \left(\dfrac{a_1\cdots a_{k_0-1}}{8\sqrt{2}}\right)^2\\
&\geq C|B(0,1)|\dfrac{r^2}{\diam(S_Q)^2} \left(\dfrac{a_1\cdots a_{k_0-1}}{8\sqrt{2}}\right)^2\\
&=C   \left(\dfrac{a_1\cdots a_{k_0-1}}{8\sqrt{2} \diam(S_Q)}\right)^2 |B(x,r)|.
\end{align*}
In order to prove \eqref{eq:Ahlfors-S_Q} fix $x\in S_Q$ and  $r\in (0,a_1\cdots a_{k_0-1}/4\sqrt{2})$. Next we take $j_0\geq k_0$ so that 
$$a_1\cdots a_{j_0}/4\sqrt{2}\leq r<a_1\cdots a_{j_0-1}/4\sqrt{2}.$$ 
Observe that $B(x,r)$ intersects at most one square $3R_0$ from the family $\{ 3R: R\in\mathcal R_{\mathbf a,m}, \, n\leq m\leq j_0  \} $. This is due to the property (4) above which implies that the elements of  $\{ 3R: R\in\mathcal R_{\mathbf a,m}, \, n\leq m\leq j_0  \} $ are $a_1\dots a_{j_0-1}/2\sqrt{2}$-separated and by the choice $r<a_1\cdots a_{j_0-1}/4\sqrt{2}$. Set $$Y=(3\overline Q \setminus Q)\setminus 3R_0$$ and note that $Y$ is $M$-Ahlfors regular, where $M>0$ is independent of the square $R_0$. We can write
\begin{equation}\label{eq: Ahlfors-S_Q-1}
    B(x,r)\cap S_Q=B(x,r)\cap Y\setminus \left(  \bigcup_{k>j_0}\bigcup_{R\in\mathcal R_{\mathbf a,k}} 3R  \right).
\end{equation}
For a given $R\in\mathcal R_{\mathbf a,k}$, by calling $x_R$ the central point of $R$ we have $3R\subset B(x_R, 3a_1\cdots a_{k}/\sqrt{2})$. For a fixed $k\in\N$, the family $\{ B(x_R, a_1\cdots a_{k-1}/4\sqrt{2}):\, R\in \mathcal R_{\mathbf a,k}\} $ is disjoint. Moreover, if we let 
 $\mathcal F_k=\{ R:\,R\in\mathcal R_{\mathbf a,k}, \; 3R\cap B(x,r)\neq\emptyset\} $ for every $k>j_0$ we have that $\bigcup_{R\in\mathcal F_k}3R\subset B(x,4r) $ because
 $$r\geq \dfrac{a_1\cdots a_{j_0}}{4\sqrt{2}}\geq \dfrac{5 a_1\cdots a_k}{4\sqrt{2}}=\left(\dfrac{5}{12}\right) 3\sqrt{2} a_1\cdots a_k\geq\dfrac{1}{3}\diam (3R). $$
In this way, for every $k>j_0$

\begin{equation}\label{eq: Ahlfors-S_Q-2}
\begin{split}
    \left| \bigcup_{R\in \mathcal R_{\mathbf a,k}}B(x,r)\cap 3R \right|&\leq \sum_{R\in\mathcal F_k}|B(0,1)| (9/2)(a_1\cdots a_k)^2=\sum_{R\in\mathcal F_k }|B(0,1)| \left(\dfrac{a_1\cdots a_{k-1}}{4\sqrt{2}}\right)^2   144a^2_k \\
    &\leq |B(0,1)| 16 r^2 \left(144a^2_k  \right).
    \end{split}
\end{equation}
Therefore, using \eqref{eq:Ahlfors-S_Q_0}, \eqref{eq: Ahlfors-S_Q-1} and \eqref{eq: Ahlfors-S_Q-2} we conclude that
\begin{align*}
    |B(x,r)\cap S_Q|&=\left|B(x,r)\cap Y \setminus \bigcup_{k>j_0}\bigcup_{R\in \mathcal R_{\mathbf a,k}}3R\right|\geq |B(x,r)\cap Y|-\left| \bigcup_{k>j_0}\bigcup_{R\in \mathcal R_{\mathbf a,k}}3R \right|\\
    &\geq M |B(x,r)|-\sum^{\infty}_{k=k_0}a^2_k |B(x,r)| 2304 \geq \left(M- \dfrac{M}{4608} 2304\right) |B(x,r)|\\
    &=(M/2)|B(x,r)|.
\end{align*}

\end{proof}

\begin{lemma}\label{lem:sierpinski_ext_op.}
    Let $S\subset \R^2$ be one of the following closed sets
    \begin{enumerate}
        \item[(i)] $\mathcal S_{\mathbf a}\cap 3\overline{Q}$ for some $Q\in\mathcal R_{\mathbf a,n}$ with $n<k_{\mathbf a}$; or
        \item[(ii)] $S_Q$ for some $Q\in\mathcal R_{\mathbf a,n}$ with $n\geq k_{\mathbf a}$.
    \end{enumerate}
 Then there exists a linear continuous operator $T\colon W^{1,1}(S)\to W^{1,1}(Q)$ whose operator norm depends only on the Ahlfors regularity constant and the Poincaré constants of $S$ (and hence on those of $\mathcal S_{\mathbf a}$). 
 Moreover, for every $u\in W^{1,1}(S)\cap \Lip(S)$ and every $x\in\partial Q$ we have that
    \begin{equation}\label{eq:cont_up_to_boundary}
\lim_{y\to x,\, y\in Q}Tu(y)=u(x).  
    \end{equation}

\end{lemma}
\begin{proof}
Observe that in both cases $(i)$ and $(ii)$ (use Lemma \ref{lem:modif_Sierpinski_preserve_properties}) the set $S\subset 3\overline{Q}$ is an Ahlfors regular closed set satisfying the weak $(1,1)$-Poincaré inequality and with $\partial Q\subset S$. Let $C_a>0$ be the Ahlfors regularity constant and $C_p>0$ and $\lambda_p\geq 1$ the Poincaré constants.

First, we are going to define Whitney decompositions of $S\setminus \partial Q$ and of $Q$ that we will call $\widetilde{\mathcal W}=\{ B_i \}$ and $\mathcal W=\{ Q_i \}$ respectively. We start with the open square $Q$ because Whitney decomposition in this situation are quite standard. Namely, following \cite[Chapter 16]{Stein} we set $\mathcal W=\{ Q_i \}_{i\in\N}$ where:
\smallskip
\begin{itemize}
    \item[(W1)] Each $Q_i$ is a closed dyadic cube inside $Q$.
  \item[(W2)] $Q=\bigcup_{i\in\N} Q_i$ and for every $i\neq j$, $\text{int}(Q_i)\cap \text{int}(Q_j)$.
      \item[(W3)] For every $i\in\N$ we have $\sqrt{2}\ell(Q_i)\leq \dist(Q_i,\partial Q)\leq 4 \sqrt{2}\ell(Q_i) $.
        \item[(W4)] If $Q_i\cap Q_j\neq \emptyset$ we have $ (1/4)\ell(Q_i)\leq \ell(Q_j)\leq 4\ell(Q_i)$.
\end{itemize}
\smallskip
For $S\setminus \partial Q$ we will have to follow a more abstract approach to build the Whitney decomposition, which is more typical in the metric measure setting. For every $z \in S \setminus \partial Q$ set $r(z) \coloneqq \dist( z, Q )/(10\lambda_p)$. From the family $\left\{ B(z, r(z)/5) \right\}_{ z \in S\setminus \partial Q }$, we select a maximal subfamily $\left\{ B(z_i,r(z_i)/5) \right\}_{ i \in \N}$ of pairwise disjoint balls. For each $i \in \N$, we refer to $B_{i} \coloneqq B( z_{i}, r(z_i) )$ as a Whitney ball with center $z_{i}$ and radius $r_{i} \coloneqq r(z_i)$. We call $\widetilde{\mathcal{W}} = \left\{ B_i \right\}_{ i \in \N}$ a \emph{Whitney covering} of $S\setminus \partial Q$. The following properties are readily verified (see \cite[Section 4]{HKST15} for details).
\smallskip
\begin{itemize} 
    \item[$(\widetilde{\text{W1}})$] $S\setminus \partial Q =\bigcup_{i\in \N} B_i$ and $5^{-1}B_i \cap 5^{-1}B_j = \emptyset$ whenever $i \neq j$.
   \item[$(\widetilde{\text{W2}})$] If $B_i\cap B_j\neq \emptyset$ then $B_i\cup B_j\subset 4B_i \cap 4B_j$.
    \item[$(\widetilde{\text{W3}})$] For every $i\in\N$, $\dist(B_i,Q)=(50\lambda_p -1) r_i$.
\end{itemize}
\smallskip
Next, we choose a partition of unity $\left\{ \varphi_{i} \right\}_{ i \in \N }$ on $Q$ subordinate to the open cover $\{ (5/4) Q_i:\, i\in\N  \}$  and so that  $|\nabla \varphi_i(x)|\leq 4\ell(Q_i)^{-1}$ for every $x$. We can now go to the definition of the extension operator. For every $Q_i\in \mathcal W$ we choose $B_{s(i)}\in \widetilde W$ so that 
 $$\dist(B_{s(i)},Q_i)\leq 8\ell(Q_i)\leq 16\,\dist(B_{s(i)},Q). $$
 Note that the existence of $s(i)$ holds due to property $(\widetilde{\text{W3}})$ and, in particular,  this yields that  $\dist(B_{s(i)}, Q_i)\sim \ell(Q_i)\sim \ell (B_{s(i)})$ where the constants of comparability may depend on $\lambda_p$. We define for every $u\in W^{1,1}(S)$ and $x\in Q$,
    $$Tu(x)=
    \sum_{i\geq 1}\varphi_i(x)a_{i},\quad \text{where}\quad a_i=\fint_{B_{s(i)}\cap S}u(x)\, dx.$$
Observe that by the local finiteness of the partition of unity we have $Tu\in C^{\infty}(Q)$ and therefore $Tu\in L^1_{\loc}(Q)$. 
In what follows we call $\rho_u\in L^1(S)$ the minimal $p$-weak upper gradient of $u$ on $S$. We show next that $T$ is a bounded operator. That is, we will find $C>0$ such that for all $u\in W^{1,1}(S)$,
\smallskip
\begin{enumerate}
    \item[(a)] $\|Tu\|_{L^1(Q)}\leq C\|u\|_{L^1(S)}$.
    \item[(b)]  $\|\nabla Tu\|_{L^1(Q)}\leq C\|\rho_u\|_{L^1(S)}$.
\end{enumerate}
\smallskip
\noindent (a) Let us first verify that $\|Tu\|_{L^1(Q)}\leq C\|u\|_{L^1(S)}$ holds, for some constant $C>0$.  We have
\begin{equation*}
    \begin{split}
        \|Tu\|_{L^1(Q)}&=\int_Q |\sum_{i\geq 1}\varphi_i(x)a_i|\, dx\leq \int_Q \sum_{i\geq 1}\varphi(x)\left( \fint_{B_{s(i)}\cap S}|u(y)|\,dy \right)\, dx \\
        &= \sum_{j\geq 1}\int_{Q_j}\sum_{\{i:\, (5/4)Q_i \cap Q_j\neq \emptyset\}}\varphi(x)\left( \fint_{B_{s(i)}\cap S}|u(y)|\,dy \right)\, dx \\
        &\leq \sum_{j\geq 1}\sum_{\{i:\, (5/4)Q_i \cap Q_j\neq \emptyset\}} \dfrac{|Q_j|}{|B_{s(i)}\cap S|}\int_{B_{s(i)}\cap S}|u(y)|\, dy.
    \end{split}
\end{equation*}
By using the Ahlfors regularity of $S$ and that $|Q_j|\sim |B_{s(i)}|$ whenever $(5/4)Q_i \cap Q_j\neq \emptyset $ we get
$$\|Tu\|_{L^1(Q)}\lesssim \sum_{j\geq 1}\sum_{\{i:\, (5/4)Q_i \cap Q_j\neq \emptyset\}}\int_{B_{s(i)}\cap S}|u(y)|\, dy .$$
Finally, by changing the order of summation, and using that for a fixed $k\in\N$, $\#\{i:\, s(i)=k\}$ and  $\#\{i:\, B_i\cap B_k\neq \emptyset\} $ are uniformly bounded, we conclude that
\begin{equation}\label{eq:bounded_L^1_ope.}
\begin{split}
    \|Tu\|_{L^1(Q)}&\lesssim \sum_{i\geq 1}\sum_{\{j:\, (5/4)Q_i \cap Q_j\neq \emptyset\}}\int_{B_{s(i)}\cap S}|u(y)|\, dy \lesssim \sum_{i\geq 1}\int_{B_{s(i)}\cap S}|u(y)|\, dy \\
    &\lesssim \sum_{k\geq 1}\int_{B_{k}\cap S}|u(y)|\, dy\lesssim \int_S |u(y)|\,dy =\|u\|_{L^1(S)}.
    \end{split}
\end{equation}

\noindent (b) Next, we want to prove that $\|\nabla Tu\|_{L^1(Q)}\leq C\|\rho_u\|_{L^1(S)}$ for another constant $C>0$. This will require more effort. We begin by writing for every $i\in\N$ the estimate
\begin{equation}\label{eq:example_norm_bound}
\begin{split}
\|\nabla Tu\|_{L^1(Q_i)}&=\int_{Q_i}|\nabla Tu(x)|\, dx=\int_{Q_i}\left| \sum_{Q_j\cap Q_i\neq \emptyset} \nabla \psi_j(x)(a_j-a_i) \right|\, dx\\
&\leq 4 \int_{Q_i} \sum_{Q_j\cap Q_i\neq \emptyset} \ell(Q_j)^{-1}|a_j-a_i| \, dx\leq 8 \sum_{Q_j\cap Q_i\neq \emptyset} \ell(Q_i)|a_j-a_i|,
\end{split}
\end{equation}
where in the last line we are using that $\ell(Q_j)\geq (1/4)\ell(Q_i)$ whenever $Q_j\cap Q_i\neq \emptyset$. In order  to control the term $|a_i-a_j|$ we assert that if $Q_i\cap Q_j\neq \emptyset$  then there is a chain of Whitney balls 
$$C_{s(i),s(j)}=\{B_{i(0)},B_{i(1)}\dots,B_{i(N)}\}\subset \widetilde{\mathcal W}$$ 
so that 
\begin{itemize}
    \item The chain joins $B_{s(i)}$ with $B_{s(j)}$, that is $B_{i(0)}=B_{s(i)}$ and $B_{i(N)}=B_{s(j)}$.
    \item  For every  $j=0,\dots, N-1$ we have $B_{i(j)}\cap B_{i(j+1)}\neq \emptyset $ and also $\diam(B_{s(j)})\sim \ell(Q_i)\sim \ell(Q_j)$.
    \item The number of balls in the chain is uniformly controlled by a constant which does not depend on $i$ or $j$. We write $N\leq N_1$.
   
\end{itemize}
Obtaining such chains of Whitney balls is done by finding a quasiconvex curve in $S$ connecting the first ball $B_{i(0)}$ with the last ball $B_{i(N)}$ and which stays far away from $Q$. Let us prove this.\footnote{Another proof (possibly simpler) of the existence of Whitney chains could be done by using quasiconvex curves that are countable concatenations of vertical and horizontal segments, and that  lie far away from $\partial Q$. This approach relies on the nice structure of the carpet $S$. However, the  method we use is also valid to define chains of Whitney balls in the general setting of closed sets $E\subset \R^2$ that are Ahlfors regular and satisfy the weak $(1,1)$-Poincaré inequality.}

\medskip

\noindent First, call $x_i\in \overline{ B}_{s(i)}\cap S  $ and $x_j\in \overline{ B}_{s(j)}\cap S$ so that $|x_i-x_j|=\dist(B_{s(i)}\cap S_Q, B_{s(j)}\cap S)$. Using that $\R^2\setminus Q$ is uniform there exists a curve $\gamma\subset (\R^2\setminus Q)$ connecting $x_i$ with $x_j$ so that for some $C>0$ (independent of the curve and of the points $x_j,x_i$), 
\begin{equation}\label{eq:chain_balls}
\ell(\gamma)\leq C \dist (B_{s(i)},B_{s(j)})\sim \ell(Q_i) \quad \text{and}\quad \dist (\gamma, Q)\geq C^{-1} \ell(Q_i). 
\end{equation}
 As a consequence we have that for some constant $C_1>0$
 $$\gamma \subset (Q+B(0,C_1\ell(Q_i)))\setminus  (Q+B(0, C^{-1}\ell(Q_i))) . $$
Without loss of generality, by considering $(\gamma \cap 3Q)\cup \partial 3Q$ we can assume that $\gamma \subset 3Q$ still satisfies the previous properties. However, it may happen that $\gamma\not\subset S$, that is, $\gamma$ intersects other obstacles of $S$. These obstacles take the form $\text{int}(3R)$ when $S=S_Q$, or $\text{int}(R)$ when $S=\mathcal S_{\mathbf a}\cap 3\overline{Q}$, for some  $R\in \bigcup_{m>n}\mathcal R_{\mathbf a,m}$. To avoid such situation we need to modify the curve $\gamma$ to obtain a new curve $\gamma'$. We only provide the proof for the case $S=S_Q$, as the other is analogous. If such an intersection occurs, we connect the entrance point $y_1 $ and exit point $y_2$ of $\gamma$ at the obstacle $3R$ using the shorter of the two connected components of $\partial 3R\setminus \{y_1,y_2\}$, denoted by $\gamma_R$. We let $\gamma_R=\emptyset$ whenever $\gamma$ does not enter $\text{int}(3R)$. Obviously, we have $\ell(\gamma\cap 3R)\geq |y_1-y_2|\geq\ell(\gamma_R)/\sqrt{2}$. By modifying $\gamma$ in this manner whenever it enters an obstacle $3R$, we obtain a new curve $\gamma'$ that satisfies the following properties
$$\ell(\gamma')=\ell(\gamma'\cap S)+\sum_{R\in \mathcal{R}_{\mathbf a, m},\, m>n}\ell(\gamma_R)  \leq \sqrt{2}\ell(\gamma\cap S)+\sqrt{2}\sum_{R\in \mathcal{R}_{\mathbf a,m},\, m>n}\ell(\gamma \cap\text{int}(3R))=\sqrt{2}\ell(\gamma)  $$
and
\begin{align*}
\dist (\gamma',Q)&=\min\{\dist(\gamma\cap S,Q),\inf_{R\in \mathcal{R}_{\mathbf a, m},\, m>n}\{\dist(\gamma_R,Q) \} \} \\
&\geq  \min\left\{\dist(\gamma\cap S,Q),\inf_{R\in \mathcal{R}_{\mathbf a, m},\, m>n,}\left\{\dfrac{1}{ 1+3\sqrt{2}}\dist(\gamma \cap 3R,Q) \right\} \right\}\\
&\geq\dfrac{1}{ 1+3\sqrt{2}}\dist(\gamma,Q)\geq \dfrac{1}{ 1+3\sqrt{2}}C^{-1}\ell(Q_i).
\end{align*}
In the second line we are using that if $3R\cap \gamma\neq \emptyset$ then $\dist(\gamma \cap 3R,Q)\leq \dist(\gamma_R,Q)+3\sqrt{2}\ell(R)$ and also that $\ell(R)\leq \dist(3R,Q)\leq \dist(\gamma_R,Q).$
For the last inequality we have used \eqref{eq:chain_balls}. We then have built a curve  $\gamma'$ from $B_{s(i)}$ to $B_{s(j)}$ such that for some constant $C_2>0$
\begin{equation}\label{eq:chain_balls_2}
\gamma' \subset (Q+B(0,C_2\ell(Q_i)))\setminus  (Q+B(0, C^{-1}\ell(Q_i))) . 
\end{equation}
This  is enough to know that if we let
$$C_{s(i),s(j)}=\{B\in\mathcal W:\, B\cap \gamma'\neq \emptyset\}$$
then, by \eqref{eq:chain_balls_2} we have $N:=\# C_{s(i),s(j)}\leq N_1 $. The chain of Whitney balls is now fully constructed.

\medskip

\noindent Going back to \eqref{eq:example_norm_bound} we now can estimate the term $|a_i-a_j|$.  Observe that for every $k=1,\dots, N$ we have $B_{i(k)}\cap B_{i(k-1)}\neq \emptyset$ so by $(\widetilde{\text{W2}})$ the inclusion $B_{i(k)}\cup B_{i(k-1)}\subset 4 B_{i(k)}\cap 4B_{i(k-1)}$ holds. Then

 \begin{equation}\label{eq:example_poinc_ineq}
 \begin{split}
  |a_i-a_j|\leq &\sum^{N}_{k=1}  \left|\fint_{B_{i(k-1)}\cap S} u(x)\,dx-\fint_{B_{i(k)}\cap S} u(x)\,dx\right|\\
         \leq& \sum^{N}_{k=1}   \left|\fint_{B_{i(k-1)}\cap S} u(x)\,dx-\fint_{4 B_{i(k)}\cap S} u(x)\,dx\right| +\left|\fint_{B_{i(k)}\cap S} u(x)\,dx-\fint_{4 B_{i(k)}\cap S} u(x)\,dx\right| \\
\leq & \sum^{N}_{k=1}  \frac{1}{|B_{i(k-1)}\cap S|}\int_{4 B_{i(k)}\cap S} \left| u(x)- \left(\fint_{4B_{i(k)}\cap S} u(y)\,dy \right) \right|\, dx \\
&+  \sum^{N}_{k=1}  \frac{1}{|B_{i(k)}\cap S|}\int_{4 B_{i(k)}\cap S} \left| u(x)- \left( \fint_{4 B_{i(k)}\cap S} u(y)\,dy \right) \right|\, dx \\
\leq & C(C_a) \sum^{N}_{k=1}  \frac{1}{|4 B_{i(k)}\cap S|}\int_{4 B_{i(k)}\cap S} \left| u(x)- \left( \fint_{4B_{i(k)}\cap S} u(y)\,dy \right) \right|\, dx\\
\leq & C(C_a, C_p) \sum^{N}_{ k=1} \ell(B_i)^{-1}\int_{4\lambda_p  B_{i(k)}\cap S} \rho_u(x)\,dx, 
   \end{split}  
 \end{equation}
 where in the last line we are using the weak $(1,1)$-Poincar\'e inequality. Since $\ell(B_i)\sim \ell (Q_i)$ and applying \eqref{eq:example_norm_bound} and \eqref{eq:example_poinc_ineq} we can write 
$$\|\nabla Tu\|_{L^1(Q_i)}\leq 8\sum_{Q_j\cap Q_i\neq \emptyset} \ell(Q_i)|a_j-a_i|\leq  C(C_a,C_p) \sum_{Q_j\cap Q_i\neq \emptyset} \sum_{B\in C_{s(i),s(j)}} \int_{4\lambda_p  B\cap S} |\rho_u(x)|\, dx .$$
Since $\#\{j: Q_j\cap Q_i\neq \emptyset\}\leq C_3 $  and recalling that if $Q_i\cap Q_j\neq \emptyset$ we have $\# (C_{s(i),s(j)})\leq N_1 $ we get
$$\|\nabla Tu\|_{L^1(Q_i)}\leq  C(C_a,C_p)  \sum_{\{B_j\in\widetilde{\mathcal W}:\, \#(C_{s(i),j})\leq N_1\}} \int_{4\lambda_p  B_{j}\cap S} |\rho_u(x)|\, dx . $$
Finally, by changing the order of summation we have 
\begin{equation}\label{eq:bounded_W^1,1_ope.}
\begin{split}
    \|\nabla Tu\|_{L^1(Q)}&=\sum_{i\geq 1}\|\nabla Tu\|_{L^1(Q_i)}\leq \sum_{i\geq 1} C(C_a,C_p)\sum_{\{B_j\in\widetilde{\mathcal W}:\, \#(C_{s(i),j})\leq N_1\}} \int_{4\lambda_p  B_{j}\cap S} |\rho_u(x)|\, dx \\
    &=C(C_a,C_p)\sum_{j\geq 1} \sum_{\{i:\, \# C_{s(i),j}\leq N_1\}} \int_{4\lambda_p  B_j\cap S}|\rho_u(x)|\, dx.
\end{split}
\end{equation}
Given the existence of a constant $N_2>0$ such that $\# \{ i: \# C_{s(i),j}\leq N_1 \} \leq N_2$ for all $j\in\mathbb{N}$, the argument concludes as follows
$$ \|\nabla Tu\|_{L^1(Q)}\leq C(C_a,C_p) N_2\sum_{j\geq 1} \int_{4\lambda_p B_j\cap S} |\rho_u(x)|\,dx\leq  C(C_a,C_p,\lambda_p)  \|\rho_u\|_{L^{1}(S)}.$$
In the last line, by applying the properties of the Whitney decomposition, we also use that for every $i\in\N$,  $\#\{j\in\N:\,4\lambda_p  B_j\cap 4\lambda_p B_i\neq \emptyset\}\leq N_3$ for another absolute constant $N_3>0$.

\medskip

Thanks to (a) and (b) (namely \eqref{eq:bounded_L^1_ope.} and \eqref{eq:bounded_W^1,1_ope.}) we have finally proved that $T\colon W^{1,1}(S)\to W^{1,1}(Q)$ is a bounded operator.

\medskip

To end the proof let us show that in the case $u\in \Lip(S)\cap W^{1,1}(S)$ we have 
\begin{equation}\label{eq:trace}
\lim_{y\to x,\, y\in Q}Tu(y)=u(x)\quad \text{for all}\; x\in \partial Q.
\end{equation}
In particular, this means that  $Tu$, which is smooth on $Q$, extends continuously to the boundary $\partial Q$ where it equals $u|_{\partial Q}$. To do so, let us fix $x\in \partial Q$ and $y\in Q$ and call $L\geq 0$ the Lipschitz constant of $u$ on $S$. Then
\begin{align*}
    |Tu(y)-u(x)|&=\left| \sum_{i\geq 1} \varphi_i(y) a_i-u(x) \right|\leq \sum_{i\geq 1}\varphi_i(y)|a_i-u(x)| \leq \sum_{i \geq 1} \left( \fint_{B_{s(i)}\cap S}|u(z)-u(x)|\,dz  \right) \\
    &\leq \sum_{i\geq 1}\varphi_i(y)L\fint_{B_{s(i)}\cap S}|z-x|\, dz.
\end{align*}
Note that $\varphi_i(y)\neq 0$ whenever $y\in (5/4)Q_i$ so, in particular by (W3), $\dist(y, \partial Q)\sim \ell(Q_i)$. Then 
\begin{align*}
\dist(x, B_{s(i)\cap S})&\leq |x-y|+\dist(y, B_{s(i)}\cap S)\leq |x-y|+\dist(Q_i, B_{s(i)\cap S})+2\ell(Q_i)\\
&\lesssim |x-y|+ C\ell(Q_i)  \lesssim |x-y| + \dist(y,\partial Q) \leq |x-y|.
\end{align*}
We can then conclude that \eqref{eq:trace} holds because 
$$|Tu(y)-u(x)|\leq  \sum_{i\geq 1}\varphi_i(y)L\fint_{B_{s(i)}\cap S}|z-x|\, dz\lesssim   \sum_{i\geq 1}\varphi_i(y)L|y-x|=L|x-y|.$$

\end{proof}

\begin{lemma}\label{lem:sierpinski_ext_op._2}
    There exists a linear continuous operator $T\colon W^{1,1}(\mathcal S_{\mathbf a})\to W^{1,1}(\R^2\setminus [0,1]^2)$ whose operator norm only depends on the Ahlfors regularity constant and the Poincaré constants of $\mathcal S_{\mathbf a}$.  Moreover, for every $u\in W^{1,1}(\mathcal S_{\mathbf a})\cap \Lip(\mathcal S_{\mathbf a})$ and every $x\in\partial [0,1]^2$ we have that
    \begin{equation}\label{eq:cont_up_to_boundary_2}
\lim_{y\to x,\, y\notin [0,1]^2}Tu(y)=u(x).  
    \end{equation}

\end{lemma}
\begin{proof}
 The proof follows the same structure as the one of Lemma \ref{lem:sierpinski_ext_op.}, where $\R^2\setminus [0,1]^2$ plays the role of $Q$. In a similar way as before, we define Whitney decompositions $\{ B_i \}_{i\geq 1}$ and $\{ Q_i \}_{i\geq 1}$ of $\mathcal S_{\mathbf a}\setminus \partial [0,1]^2$ and $\R^2\setminus [0,1]^2$ respectively. We choose a partition of unity $\left\{ \varphi_{i} \right\}_{ i \in \N }$ on $\R^2\setminus [0,1]^2$ subordinate to the open cover $\{ (5/4) Q_i\}_{i\geq 1}$  and so that  $|\nabla \varphi_i|\leq 4\ell(Q_i)^{-1}$. For every $Q_i\in \mathcal W$ we choose $B_{s(i)}\in \widetilde W$ so that
 $\dist(B_{s(i)}, Q_i)\sim \ell(Q_i)\sim \ell (B_{s(i)}$. Call 
 \begin{equation}\label{eq:support_Tu_otside_[0,1]^2}
 U=\bigcup_{i\in\N} \{Q_i:\, \text{there exists an associated $B_{s(i)}$}   \}\subset \R^2\setminus [0,1]^2.
 \end{equation}
We define for every $u\in W^{1,1}(\mathcal S_{\mathbf a})$ the extension $$Tu(x)=
    \sum_{i\geq 1}\varphi_i(x)a_{i},\quad x\in U.$$
Following the same arguments as in Lemma \ref{lem:sierpinski_ext_op.} we can prove that for all $u\in W^{1,1}(\mathcal S_{\mathbf a})$ 
$$\|Tu\|_{W^{1,1}(U)}\leq C(C_a,C_p,\lambda_p) \|u\|_{W^{1,1}(\mathcal S_{\mathbf a})}.$$
Multiplying $Tu$ by a Lipschitz cut-off function $\psi$ that equals one on $[0,1]^2$ and zero on $U^c$, we get $\psi Tu\in W^{1,1}(\R^2\setminus [0,1]^2)$ and the Sobolev norm estimate is preserved. On the other hand, proving \eqref{eq:cont_up_to_boundary_2} follows again the same steps as in Lemma \ref{lem:sierpinski_ext_op.}.
\end{proof}

We are now ready to prove Theorem \ref{thm:Sierp-W11_ext.}, followed by the proof of Theorem \ref{thm:Sierp-W11_ext._char}.

\begin{proof}[Proof of Theorem \ref{thm:Sierp-W11_ext.}]
    First, note that $  \mathcal S_{\mathbf a} $ is a compact set which has the weak $(1,1)$-Poincaré inequality and is Ahlfors regular by \cite{EG21}. Let $(C_p,\lambda_p)\in (0,\infty)\times[1,\infty)$ and $C_a>0$ be the  Poincaré constants and the Ahlfors regularity constant  respectively.

\medskip

In order to define the extension operator $T\colon W^{1,1}(\mathcal S_{\mathbf a})\to W^{1,1}(\R^2)$ observe that by the density of Lipschitz functions on $W^{1,1}(\mathcal S_{\mathbf a})$ (see \cite[Theorem 8.2.1]{HKST15} it is enough to find an extension operator $T\colon W^{1,1}(\mathcal S_{\mathbf a})\cap \Lip(\mathcal S_{\mathbf a})\to W^{1,1}(\R^2)$.
Moreover, since we can write
    $$\R^2\setminus \mathcal S_{\mathbf a}= (\R^2\setminus [0,1]^2)\cup \bigcup_{n\in\N}\bigcup_{Q\in \mathcal R_{\mathbf a,n}} Q$$
for any given $u\in W^{1,1}(\mathcal S_{\mathbf a})$ we only need to define $Tu$ on $\R^2\setminus [0,1]^2 $ and on all squares $Q\in \mathcal R_{\mathbf a,n}$ for every $n\in\N$. Call  first 
$$T_0\colon W^{1,1}(\mathcal S_{\mathbf a})\to W^{1,1}(\R^2\setminus [0,1]^2)$$
the bounded linear operator provided by Lemma \ref{lem:sierpinski_ext_op._2}. Next, we distinguish two types of squares $Q$ to be filled. Let us first fix $k_{\mathbf a}\in\N$ satisfying \eqref{eq:Sierpinki_1} and take 
$Q\in  \mathcal R_{\mathbf a,n}$ with $n<k_{\mathbf a}$. By Lemma \ref{lem:sierpinski_ext_op.} $(i)$ there is a bounded linear operator 
$$\widetilde T_Q\colon W^{1,1}(\mathcal S_{\mathbf a}\cap 3\overline{Q})\to W^{1,1}\left(Q\right). $$
On the other hand, for every $n\geq k_{\mathbf a}$
and every $Q\in\mathcal R_{\mathbf a,n}$, recalling the definition of $S_Q$ given in \eqref{eq:Sierpinski_2}, we  use  Lemma \ref{lem:sierpinski_ext_op.} $(ii)$ to define extension operators 
$$T_{Q}\colon W^{1,1}(S_Q)\to W^{1,1}( Q).  $$
All operators $T_0$, $T_Q$ and $\widetilde T_{Q}$ are norm-bounded by some constant $C>0$, only depending on $C_p$, $\lambda_p$,  $C_a$ and $\mathbf a$, and also satisfying \eqref{eq:cont_up_to_boundary} or \eqref{eq:cont_up_to_boundary_2}. We finally define $T\colon W^{1,1}(\mathcal S_{\mathbf a})\cap \Lip(\mathcal S_{\mathbf a}) \to W^{1,1}(\mathbb R^2)$ as
$$Tu(x)=\begin{cases}
u(x), & \text{if $x\in \mathcal S_{\mathbf a}$}\\
T_0(u) (x), & \text{if $x\notin [0,1]^2$}\\
\widetilde T_Q (u|_{\mathcal S_{\mathbf a}\cap 3\overline{Q}})(x), &\text{if}\; x\in Q \;\text{for some $Q\in \mathcal R_{\mathbf a,n}$ and some $n<k_{\mathbf a}$}\\
T_{Q}(u|_{S_Q})(x),&\text{if}\; x\in Q \;\text{for some $Q\in \mathcal R_{\mathbf a,n}$ and some $n\geq k_{\mathbf a}$.} 
\end{cases}$$
Recall that $S_Q\cap S_{Q'}=\emptyset$ whenever $Q,Q'\in  \mathcal R_{\mathbf a,n}$, $n\geq k_{\mathbf a}$ and $Q\neq Q'$.  This implies that we have a disjoint union
$$\bigcup_{n\geq k_{\mathbf a}}\bigcup_{Q\in \mathcal R_{\mathbf a,n}} S_Q\subset \mathcal S_{\mathbf a}.$$
Moreover $\#\{Q\in \mathcal R_{\mathbf a,n}: 1\leq n\leq k_{\mathbf a}-1\} \leq N_{\mathbf a}$ has a finite number of elements, therefore
\begin{equation}\label{eq:norm-bound-ext-operator}
\begin{split}
    \| Tu\|_{W^{1,1}(\R^2\setminus \mathcal S_{\mathbf a})}&= \|T_0(u)\|_{W^{1,1}(\R^2\setminus [0,1]^2)}+\sum^{k_{\mathbf a}-1}_{n=1}\sum_{Q\in \mathcal R_{\mathbf a,n}}\|Tu\|_{W^{1,1}(Q)} + \sum_{n\geq k_{\mathbf a}}\sum_{Q\in \mathcal R_{\mathbf a,n}}\|Tu\|_{W^{1,1}(Q)} \\
   &\leq C \|u\|_{W^{1,1}(\mathcal S_{\mathbf a})}+ \sum^{k_{\mathbf a}-1}_{n=1}\sum_{Q\in \mathcal R_{\mathbf a,n}}\|\widetilde T_Q (u|_{\mathcal S_{\mathbf a}\cap 3\overline{Q}})\|_{W^{1,1}(Q)} + \sum_{n\geq k_{\mathbf a}}\sum_{Q\in \mathcal R_{\mathbf a,n}}\|T_{Q}(u|_{S_Q})\|_{W^{1,1}(Q)}\\
    &\leq  C \|u\|_{W^{1,1}(\mathcal S_{\mathbf a})}+  N_{\mathbf a} C \|u\|_{W^{1,1}(\mathcal S_{\mathbf a})}+ \sum_{n\geq k_{\mathbf a}}\sum_{Q\in \mathcal R_{\mathbf a,n}}C\|u\|_{W^{1,1}(S_Q)}\\
    &\leq C(2+N_{\mathbf a}) \|u\|_{W^{1,1}(\mathcal S_{\mathbf a})}.
    \end{split}
\end{equation}
In particular we deduce that $Tu\in L^{1}(\R^2)$. 

\medskip

We next prove that $Tu\in W^{1,1}(\R^2)$. To do so we will prove that $Tu$ is absolutely continuous on $\mathcal H^1$-almost every line segment parallel to the coordinate axes, and that those partial derivatives that exist almost everywhere are $L^1$-integrable. Then, using \cite[Theorem 4.21]{EG2015} we will conclude that $Tu\in W^{1,1}(\R^2)$. To prove the absolute continuity, we will only look at vertical line segments, and the horizontal lines case follows then by symmetry.

Namely, our goal is to prove that for $\mathcal H^1$-almost every $t\in\R$ and every $[a,b]\subset \R$,
\begin{equation}\label{eq:abs_cont_a.e._line}
|Tu(t,b)-Tu(t,a)|\leq \int_{\{t\}\times [a,b] } g\, ds, \quad \text{where}\;\; g(x)=\begin{cases}\rho_u(x) &\text{if}\; x\in \mathcal S_{\mathbf a}\\
|\nabla Tu(x)| &\text{if}\;   x\in \R^2\setminus \mathcal S_{\mathbf a}  \end{cases}.
\end{equation}
Once we prove this, it is easy to check that for $\mathcal H^1$-almost every $t\in \R$ the {\em classical} partial derivative  $(\partial (Tu)/\partial x)(t,x)$, that exist on $\mathcal H^1$-almost every point $x\in\R$, belongs to $L^1(\R^2)$. Indeed, using that $Tu=0$ outside the bounded set $U$ defined on \eqref{eq:support_Tu_otside_[0,1]^2} there must exist $l>0$ such that  $U\subset [-l,l]^2$, and by Fubini's theorem
\begin{align*}
\left\| \frac{\partial (Tu)}{\partial x}\right\|_{L^1(\R^2)}&=\int^{l}_{-l} \int^{l}_{-l}\dfrac{\partial (Tu)}{\partial x}(t,s) \, dt\, ds=  \int^{l}_{-l} |Tu(t,l)-Tu(t,-l)|\, dt     \\
&\leq   \int^{l}_{-l} \int^{l}_{-l}g(t,s)\, dt\, ds\leq \|g\|_{L^1(\R^2)}<\infty. 
\end{align*}
It is hence clear that we only need to prove \eqref{eq:abs_cont_a.e._line} to get that $Tu\in W^{1,1}(\R^2)$. In the case $t\notin [0,1]$ we have $u=Tu$ on $\{t\}\times\R$ and the absolutely continuity of $Tu$ on $\R^2\setminus [0,1]^2$ gives the claim. So we may assume that $t\in [0,1]$. In what follows, for explanatory reasons, let us name $Q_0=\R^2\setminus [0,1]^2$ and refer to it as a {\em square}.

\medskip

 We start by recalling that, being  $\rho_u\in L^1(\mathcal S_{\mathbf a})$  a weak upper gradient of $u\in W^{1,1}(\mathcal S_{\mathbf a})\cap\Lip(\mathcal S_{\mathbf a})$, we know that for every $t\in [0,1]$ and every segment $\{t\}\times [a,b]\subset \mathcal S_{\mathbf a}$ we have 
\begin{equation}\label{eq:abs_cont_u^*}
|u(t,a)-u(t,b)|\leq \int_{\{t\}\times [a,b]} \rho_u \, ds. 
\end{equation}
For those vertical lines segments $\{t\}\times [a,b]\subset \{t\}\times\R$ that intersect infinitely many squares $Q\in \bigcup_{n\in\N}\mathcal R_{\mathbf a,n}$, we will prove that they  must form an $\mathcal H^1$-null set. To be precise, let $t\in [0,1]$ and
 $$F_t= \left\{Q\in \{Q_0\}\cup\bigcup_{n\in\N}\mathcal R_{\mathbf a,n}:\, (\{t\}\times [0,1])\cap Q\neq \emptyset \right\}. $$
If we call $D=\{t\in [0,1]:\, \# (F_t)=\infty \}$ then we want to show that $\mathcal H^1(D)=0$. To do so we start by denoting $\widetilde Q_i= \bigcup_{Q\in \mathcal R_{\mathbf a, i}} Q$. We have that 
 $$D\subset \bigcup_{i\geq n}\pi_1(\widetilde Q_i) \quad \text{for every $n\in\N$}.$$
Since $\mathbf a \in \ell_1$ we have $\sum^{\infty}_{j=1} \log(1-a_j)>-\infty$. This means that for any $\varepsilon>0$, there exists $n_0\in\N$ so that $\sum^{\infty}_{j=n_0}\log(1-a_j)\in(-\varepsilon,0).$ Hence
$$\mathcal H^1\left([0,1]\setminus \bigcup_{i\geq n_0}\pi_1(\widetilde Q_i)\right)= \prod^{\infty}_{j=n_0}(1-a_j)=e^{\sum^{\infty}_{j=n_0}\log(1-a_j)}\in (e^{-\varepsilon},1).$$
This means that 
$$\lim_{n\to\infty}\mathcal H^1\left([0,1]\setminus \bigcup_{i\geq n}\pi_1(\widetilde Q_i)\right)=1$$ which leads to
 $\mathcal H^1([0,1]\setminus D)=  1 $. Therefore $\mathcal H^1(D)=0$ as we wanted. This  fact together with \eqref{eq:abs_cont_u^*}, yields that for $\mathcal H^1$-almost every $t\in [0,1]$ we have  $F_t=\{Q^1,\dots, Q^{m}\}\subset\{Q_0\}\cup\bigcup_{n\in\N}\mathcal R_{\mathbf a,n} $, and whenever $F_t=\emptyset$ (that is $\{t\}\times [a,b]\subset \mathcal S_{\mathbf a} $)  the function $u|_{\{t\}\times [a,b]}$ is absolutely continuous. Naming the projection into the $y$-axis
$$\pi_2(\{t\}\times [a,b]\cap Q^i)=(a_i,b_i),\quad \forall i=1,\dots m$$
and setting $b_0=a$ and $a_{m+1}=b$, we may write
\begin{align}\label{eq:abs_cont_1}
    |Tu(t,b)-Tu(t,a)|\leq \sum^{m}_{i=1}|Tu(t,b_i)-Tu(a_i)| +\sum^{m}_{i=0} |Tu(t,a_{i+1})-Tu(t,b_i)|,
\end{align}
where 
$$\bigcup^{m}_{i=0}\{t\}\times [b_i,a_{i+1}]\subset \mathcal S_{\mathbf a} .$$
Since $u$ is absolutely continuous on each of those segments we have
\begin{equation}\label{eq:abs_cont_2}
   |Tu(t,a_{i+1})-Tu(t,b_i)|=|u(t,a_{i+1})-u(t,b_i)|  \leq  \int_{\{t\}\times [b_i,a_{i+1}]}\rho_u \, ds,\quad \forall i=1,\dots, m.
\end{equation}
Moreover, on every $Q\in\{Q^1,\dots Q^m\}$ we have $Tu\in W^{1,1}(Q)$. Since $Tu$  extends continuously to the closure  $\overline Q$ by using \eqref{eq:cont_up_to_boundary} and \eqref{eq:cont_up_to_boundary_2} from Lemmas \ref{lem:sierpinski_ext_op.} and \ref{lem:sierpinski_ext_op._2},  we can say that  $Tu$ is also absolutely continuous on $\mathcal H^1$-almost every vertical line segment contained in $\overline{Q}$. This leads to
\begin{equation}\label{eq:abs_cont_3}
|Tu(t,b_i)-Tu(a_i)| \leq \int_{\{t\}\times [a_i,b_i]}|\nabla (Tu)|\, ds ,\quad \forall i=1,\dots, m.
\end{equation}
By joining \eqref{eq:abs_cont_2} and \eqref{eq:abs_cont_3}, and going back to \eqref{eq:abs_cont_1} we have
$$|Tu(t,b)-Tu(t,a)|\leq \int_{\{t\}\times [0,1] } g\, ds.$$
Recall that $g$ was defined as $\rho_u$ on $\mathcal S_{\mathbf a}$ and as $|\nabla Tu|$ on $[0,1]^2\setminus \mathcal S_{\mathbf a}$ and that $g\in L^1(\R^2)$. This concludes the proof that $Tu\in W^{1,1}(\R^2)$.

\medskip

Finally, once we have proved that $ Tu\in W^{1,1}(\R^2)$, using \eqref{eq:norm-bound-ext-operator}, it is easy to check that $T$ is norm bounded:
\begin{align*}
    \| Tu\|_{W^{1,1}(\R^2)}=\|u\|_{W^{1,1}(\mathcal S_{\mathbf a})}+\|Tu\|_{W^{1,1}(\R^2\setminus \mathcal S_{\mathbf a})}\leq (1+C(1+N_{\mathbf a})) \|u\|_{W^{1,1}(\mathcal S_{\mathbf a})}.
\end{align*}

\end{proof}

\begin{proof}[Proof of Theorem \ref{thm:Sierp-W11_ext._char}]
   One implication is given by Theorem \ref{thm:Sierp-W11_ext.}, together with \cite[Theorem 1.5]{Sierpinski}.

   For the other one we will argue by contradiction and we will assume that $\mathcal S_{\mathbf a}$ does not satisfy a weak $(1,1)$-Poincaré inequality. That is, by using \cite[Theorem 1.5]{Sierpinski}, we have $\mathbf a\notin\ell_1$. In order to prove that $\mathcal S_{\mathbf a}$ is not a $W^{1,1}$-extension set, thanks to \cite[Proposition 3.4]{CKLR25}, it is enough to show instead that $\mathcal S_\mathbf a$ is not a $BV$-extension set. For the later we will make use of Proposition \ref{prop:BV_ext_per_full_norm}, characterizing the $BV$-extension  property through the extension of  sets of finite perimeter with full norm. For this, consider for every $k\in\N$ the sets 
$$F_k=\mathcal S_{\mathbf a}\cap ([0,1]\times [0, a_1\cdots a_{k}/2]).$$ 
Note that, by the definition of perimeter $P_{S_{\mathbf a}}(F_k)=|D\chi_{F_k}|_{\mathcal S_{\mathbf a}}(\mathcal S_{\mathbf a})$, it is easy to check that
$$ P_{S_{\mathbf a}}(F_k)\leq \mathcal H^1\left(\mathcal{S}_{\mathbf a}\cap \left(\left\{[0,1]\times\frac{a_1\dots a_{k}}{2}\right\}   \right).  \right) $$
Using that $\mathbf a\notin \ell_1$ we can prove that the right hand side vanishes. Indeed, fix $k\geq 1$ and note that $$\mathcal H^1\left(\mathcal{S}_{\mathbf a}\cap \left(\left\{[0,1]\times \frac{a_1\cdots a_{k}}{2}\right\}  \right)  \right)=\displaystyle{\lim_{n\to\infty} }\dfrac{1}{a_1\cdots a_{k}}\prod^{n}_{j=k+1}(1-a_j).$$ By calling $x_n:=\prod^{n}_{j=k+1}(1-a_j)>0$,
$$ \lim_{n \to \infty} x_n > 0 \quad \Leftrightarrow \quad \lim_{n \to \infty} \log\bigg(\prod_{j =k+1}^n (1-a_j)\bigg) > -  \infty  \quad \Leftrightarrow \quad  \lim_{n \to \infty} \sum_{j = k+1}^n\big( - \log(1-a_j)\big) < + \infty. $$ 
Since the convergence of the series $\sum_{j\geq k+1} -\log(1-a_j)$ is equivalent to the convergence of $\sum_{j\geq k+1}a_j$, which only occurs when $\mathbf a\in\ell_1$, we conclude that $\displaystyle{\lim_{n\to\infty}}x_n=0$, and hence $P_
{\mathcal S_{\mathbf a}}(F_k)=0$ for every $k\in\N$.

On the other hand, $|F_k|\leq (a_1\dots, a_{n_k})/2$, so $|F_k|\to 0$ as $k\to\infty$. 
Let $\widetilde F_k$ be any extensions of the sets $F_k$. Then, by the Euclidean isoperimetric inequality in $\R^2$ (see \cite[Theorem 5.11]{EG2015}) we have\footnote{Note, that according to the notation of this paper, we have $P_{\R^2}(\widetilde F_k)=|D_{\chi_{\widetilde F_k}}|_{\R^2}(\R^2)=|D_{\chi_{\widetilde F_k}}|(\R^2)=\|D_{\chi_{\widetilde F_k}}\|(\R^2)$.}
\[
|\widetilde F_k| + P_{\R^2}(\widetilde F_k) \ge P_{\R^2}(\widetilde F_k) \ge C \sqrt{|\widetilde F_k|} \ge C \sqrt{| F_k|} = \frac{C}{\sqrt{|F_k|}} (|F_k| + P_
{\mathcal S_{\mathbf a}}(F_k)),
\]
where $C/\sqrt{|F_k|} \to \infty$ as $k \to \infty$. This shows, according to Proposition \ref{prop:BV_ext_per_full_norm}, that $\mathcal S_{\mathbf a}$ cannot be a $BV$-extension set.

\end{proof}

We conclude this section by presenting two examples of both a closed and an open planar sets that act as $W^{1,p}$-extension sets for all $p>1$, but fail to satisfy the $W^{1,1}$-extension property\footnote{The contrary may also occur; that is, there exist open sets that are $W^{1,1}$-extension sets which are not $W^{1,p}$-extension sets for any $p>1$ (consider, e.g., the inward cusps). }.
For the case of open subsets $\Omega\subset\R^2$ similar examples can be constructed using the results on \cite{Koskela1999}. Note that, as explained in \cite{KRZ25}, any planar domain which is a $W^{1,p}$-extension domain for all $p>1$ but fails to be a $W^{1,1}$-extension domain, cannot be simply connected.

\begin{proposition}\label{prop:W1p-ext-closed-NOT-W11}
    Let $\mathbf a\in \ell_2\setminus \ell_1$. Then $\mathcal S_{\mathbf a}$ is a closed $W^{1,p}$-extension set for $p>1$, but it is not a $W^{1,1}$-extension set, nor a $BV$-extension set.
\end{proposition}

\begin{proof}
Since $\mathbf a\in \ell_2\setminus \ell_1$ we know from \cite{Sierpinski} that $\mathcal S_\mathbf a$ is Ahlfors regular and has a weak $(1,p)$-Poincaré inequality. Therefore, by Proposition \ref{prop:plargerthan1} we find that $\mathcal S_{\mathbf a}$ is a $W^{1,p}$-extension set.

On the other hand, since $\mathbf a\notin \ell_1$, following the argument of the second part of the proof of Theorem \ref{thm:Sierp-W11_ext._char} we know that $\mathcal S_{\mathbf a}$ is not a $W^{1,1}$-extension set (nor a $BV$-extension set).
\end{proof}

\begin{proposition}\label{prop:W1p-ext-dom-NOT-W11}
    There exists an Ahlfors regular domain $\Omega\subset \R^2$ that is a $W^{1,p}$-extension domain for every $p>1$, but it is not a $W^{1,1}$-extension domain, nor a $BV$-extension domain.
\end{proposition}

\begin{proof}
    
Let $\mathbf a=(a_n)_{n\geq 1} \in \ell_2\setminus \ell_1$. The idea is to imitate the iterative construction of the Sierpinksi carpet with two differences. First, in order to get a domain, we will remove closed squares at each step instead of open ones. And second, after a finite number of steps $(a_1,\dots,a_{n_1-1})$ we continue the process only in the bottom row of squares from this iteration. That is, only in those squares from $\mathcal T_{\mathbf a,n_1-1}$ which intersect $[0,1]\times \{0\}$, we continue  removing squares following the scales $(a_{n_1},a_{n_1+1},\dots a_{n_2-1})$ until again we stop at iteration $n_2-1$. As before, we continue the process only on the bottom row of squares from iteration $n_2-1$ with the subsequent $(a_{n_2},a_{n_2+1},\dots a_{n_3-1})$, and so on. The key point is that the weak Poincaré inequalities behave the same way as in the well-studied Sierpinksi carpet $\mathcal S_{\mathbf a}$, and by making the holes accumulate on $[0,1]\times \{0\}$ we indeed have an open set. By making an appropriate choice of the sequence $(n_k)_{k\geq 1}$, our domain will have the weak $(1,p)$-Poincaré inequality for all $p>1$ but not the weak $(1,1)$-Poincaré inequality. 

\medskip

We now give the definition of $\Omega$. Recall that the fat Sierpinksi carpet $\mathcal S_{\mathbf a}$ was defined as  \[
  \mathcal S_{\mathbf a} = \bigcap_{n\in\N}\mathcal S_{\mathbf a,n}=[0,1]^2\setminus \bigcup_{n\in\N}
\mathcal \bigcup_{Q\in \mathcal R_{\mathbf a,n}}Q.
\]
Instead, let us define for every $m\geq 2$, the Sierpinksi precarpet domains
$$\Omega^{m}_{\mathbf a}=\bigcap_{n=1}^{m-1} \Omega_{\mathbf a,n}= (0,1)^2\setminus \bigcup_{n\geq 1}^{m-1}\bigcup _{Q\in \mathcal R_{\mathbf a,n}}\overline{Q} \quad ;\quad \Omega^{1}_{\mathbf a}=(0,1)^2 .$$
Pick next a strictly increasing sequence $(n_k)_{k\geq 1}\subset(\N\setminus \{1\})$  so that 
\begin{equation}\label{eq:prop4.7}
\mathcal H^1\left(\Omega^{n_{k}}_{\mathbf a}\cap \left( [0,1] \times \{ a_1\cdots a_{n_{k}-1}/2 \} \right)  \right) =\dfrac{1}{a_1\cdots a_{n_{k}-1}}\prod^{n_{k+1}-1}_{j=n_{k}}(1-a_j)\leq \dfrac{1}{k}.\end{equation}
The existence of such sequence is guaranteed by the fact that $\mathbf a\notin \ell_1$. To simplify notation let $m_k=n_{k+1}-1$ from now on. We define our domain  $\Omega\subset (0,1)^2$ as
$$\Omega=\bigcap_{k\geq 1}U^{k}_{\mathbf a} \quad \text{where} \quad \begin{cases}U^{1}_{\mathbf a}=\Omega^{n_1}_{\mathbf a}\\  U^{k}_{\mathbf a}= \Omega^{n_k}_{\mathbf a}\cup [(0,1)\times (a_1\cdots a_{n_{k-1}},1)],\;\;k\geq 2
\end{cases}. $$

We now show that $\Omega$ is not a $BV$-extension domain (and hence not a $W^{1,1}$-extension domain), whereas it remains a $W^{1,p}$-extension domain for all $p > 1$.

\smallskip

{\em $\bullet$ $\Omega$ is not a $BV$-extension domain}: This can be seen by a similar argument as that of the second part of Theorem \ref{thm:Sierp-W11_ext._char}. Indeed, in this situation, since we work with a domain in $\R^2$, the $BV$-extension property (see \cite{BM1967} and \cite[Lemma 2.1]{KMS2010}), is equivalent to the extension of sets of finite perimeter. Therefore, let $F_k=\Omega\cap ((0,1)\times (0,a_1\cdots a_{n_k -1}/2))$, so that we have $$P(F_k,\Omega))= \mathcal H^1\left(\Omega^{n_{k}}_{\mathbf a}\cap \left( [0,1] \times \{ a_1\cdots a_{n_{k}-1}/2 \} \right)  \right) =\dfrac{1}{a_1\cdots a_{n_{k}-1}}\prod^{n_{k+1}-1}_{j=n_{k}}(1-a_j)\leq \dfrac{1}{k}.$$
Moreover, by imitating the argument and notation from the last part of Theorem \ref{thm:Sierp-W11_ext._char}, any possible extension $\widetilde F_k$ must satisfy
\begin{align}
P(\widetilde F_k,\R^2)&\geq  \sum^{m_k}_{j=n_k}\sum_{Q\in \mathcal F^j_k}P(\widetilde F_k,R_Q)\geq \sum^{m_k}_{j=n_k}\sum_{Q\in \mathcal F^j_k} C (a_1\cdots a_j)\\
&= \sum^{m_k}_{j=n_k}\sum_{Q\in \mathcal F^j_k} C \ell(Q)=C(1-P(F_k,\Omega))\geq C \left( 1-\dfrac{1}{k}\right)
\end{align}
 for some positive constant $C>0$. Consequently, sets of finite perimeter cannot be extended in the sense of conditions (PE1) and (PE2) from the introduction.

\medskip

{\em $\bullet$ $\Omega$ is not a $W^{1,1}$-extension domain}: This follows from the fact that $\Omega$ is not a $BV$-extension domain and by using \cite[Lemma 2.4]{KMS2010}.

\medskip

{\em $\bullet$  $\Omega$ is a $W^{1,p}$-extension domain for every $p>1$}: This is clear because $\Omega$ satisfies a weak $(1,p)$-Poincaré inequality, thanks to the fact that $\mathbf a\in \ell_2$. We omit the details of this proof because it follows the same strategy as \cite[Theorem 1.6]{Sierpinski}. Since we also know that $\Omega$  is Ahlfors regular we can make use of Proposition \ref{prop:plargerthan1} and conclude that $\Omega$ satisfies the  $W^{1,p}$-extension property.

\end{proof}

\section{An Instructive Example}\label{sec:example}

We present an example showing that the method that we used in the proof of Theorem \ref{thm:Sierp-W11_ext.} to construct the extension operator does not work for general Ahlfors-regular compact subsets of the plane satisfying  a weak $(1,1)$-Poincar\'e inequality. The main point is that a parallel version of Lemma \ref{lem:modif_Sierpinski_preserve_properties}, where a modified Sierpi\'nski carpet was built by enlarging the holes at small enough scales while maintaining the weak $(1,1)$-Poincaré inequality, fails to hold for this specific example. This limitation is formalized in Proposition \ref{lem:properties_of_example} below.

\begin{example}\label{ex:method_fails}

We define a  compact set  $E\subset \R^2$  as follows. Pick a decreasing sequence $(t_k)_{k\geq 1}\subset \left(\tfrac{1}{3}, \tfrac{5}{12}\right)$ so that $t_k\to \tfrac{1}{3}$. Then, let $\epsilon_1= \tfrac{1}{2}$ and for every $k\geq 2$ define 

$$ \epsilon_k=\left( t_{k-1}-\dfrac{1}{3} \right)\epsilon_{k-1}.$$
The sequence $(\varepsilon_k)_{k\geq 1}$ is decreasing and converges to zero. Note that $\epsilon_k<\tfrac{1}{12}\epsilon_{k-1}$ for all $k\in\N$.
Let also $n_k\in\N$ be the unique number such that
$$\dfrac{1}{n_{k}}\leq \epsilon_k <\dfrac{1}{n_k-1}.$$
Next, for every $k\in\N$, we consider a family of disjoint $n_k$ equilateral triangles $\{I_{k,j}\}_{j=1}^{n_k}$ whose upper sides are equi-distributed on $[0,1]\times  \{t_k\epsilon_k\}$ and  all of them have their (lower) vertex within the segment $[0,1]\times\{(t_k-1/3)\epsilon_k\}$. Note that the height of each triangle $I_{k,j}$ is $\tfrac{\epsilon_k}{3}$, that $ \ell(I_{k,j})=\tfrac{2\epsilon_k}{3\sqrt{3}}$ and also $\dist(I_{k,j},I_{k,j+1})\sim \epsilon _k$ for all $j=1,\dots, n_k-1$.

\noindent  Finally, we define the compact set $E\subset \R^2$ as
\begin{equation}\label{eq:final_example}
E=([-1,2]\times[-1,1]) \setminus \left( \bigcup^{\infty}_{k=1}\bigcup^{n_k}_{j=1}I_{k,j}\right).\end{equation}
 In what follows, for an open equilateral triangle $I\subset \R^2$ with center $(x_0,y_0)$ and lower vertex  $(x_0,z_0)$, and for a given $\eta>0$, we denote by $\eta I$ the open equilateral triangle with same center $(x_0,y_0)$ and lower vertex $(x_0, y_0-\eta (y_0-z_0)) $. Moreover, define
$$\alpha=\sup \left\{t>0:\, \text{the family $\{tI_{k,j}\}^{n_k}_{j=1}$ is disjoint for every $k\in\N$}  \right\} >1 .$$

\begin{figure}

\begin{center}
\begin{tikzpicture}[scale=3.5]

\fill[black!70, scale=2.3] (0.85,-0.2) rectangle (2.15,0.5);

\filldraw[color=black!70, fill=white, scale=2.3](1.2,0.0625) -- (1.0666,0.28) -- (1.3333,0.28) -- cycle;

\filldraw[color=black!70, fill=white, scale=2.3](1.8,0.065) -- (1.6666,0.28) -- (1.9333,0.28) -- cycle;

\foreach \x in {1,2,3,...,16}{\filldraw[color=black!70, fill=white, thin,  scale=2.3]({1+(\x*0.0588)},0.0052) -- ({1+(\x*0.0588)-0.015},0.024) -- ({1+(\x*0.0588)+0.015},0.024) -- cycle;}

\foreach \x in {1,2,3,...,192}{\filldraw[color=black!70, fill=white, thin,  scale=2.3]({1+(\x*0.0052)},0.0002) -- ({1+(\x*0.0052)-0.0022},0.0032) -- ({1+(\x*0.0052)+0.0022},0.0032) -- cycle;}

\end{tikzpicture}
\end{center} {The set $E\cap \left[-\tfrac{1}{3},\tfrac{4}{3}\right]\times\left[-\tfrac{1}{4},\tfrac{1}{4}\right]$ defined in \eqref{eq:final_example} for $\epsilon_1=1/2$, $n_1=2$ and $(t_k)_{k\geq 1}=(\frac{19}{48},\frac{17}{8},\dots)$.} 
 \end{figure}

\begin{proposition}\label{lem:properties_of_example}
For the compact set $E\subset\R^2$ defined in \eqref{eq:final_example} we have:

\begin{enumerate}
\item $E$ is Ahlfors 2-regular.
\item $E$ satisfies a weak $(1,1)$-Poincar\'e inequality.
\item $E$ is a $W^{1,1}$-extension set.
\item  For any choice of parameters $1 < \eta <\tau< \alpha$, any choice of open neighborhoods $U_{k,j}$ with $\eta I_{k,j} \subset U_{k,j} \subset \tau I_{k,j}$ for all $k,j$ and any choice of $N\in\N$, the set
\[E_{N}=
([-1,2]\times[-1,1])  \setminus \bigcup_{k\ge N}\bigcup^{n_k}_{j=1}U_{k,j}
\]
does not satisfy a weak $(1,1)$-Poincar\'e inequality. Indeed, the Poincaré inequality fails for balls centered at points from $(0,1) \times \{0\}$.
\end{enumerate}
    
\end{proposition}
Observe that by Proposition \ref{lem:properties_of_example} (4), the method from Theorem \ref{thm:Sierp-W11_ext.} used to prove that fat Sierpiński carpets are $W^{1,1}$-extension sets (specifically, Lemma \ref{lem:modif_Sierpinski_preserve_properties}) fails for Example \ref{eq:final_example}.

\begin{remark}
The most technical part of the proof of Proposition \ref{lem:properties_of_example} is showing that $E$ supports a weak $(1,1)$-Poincaré inequality. For this, we cannot rely on Theorem \ref{thm:Sylvester_22} (see \cite[Theorem 1.3]{EG22}), because the set $E$ does not have {\em small projections} as required in Theorem \ref{thm:Sylvester_22} (5). Instead, we prove the  weak $(1,1)$-Poincaré inequality by showing the existence of Semmes  family of curves on $E$.

\end{remark}
\begin{proof}

$(1)$ Let us show that $E$ is Ahlfors regular. We remark that the Alhfors regularity constants that we get are far from being sharp. We start by defining the strips 
$$S_k:=[-1,2]\times[\epsilon_{k},\epsilon_{k-1}],\quad k\geq 1,$$ where $\epsilon_0=1$, and also name $E_l:=[-1,2]\times [-1,0]$ the lower part of $E$. Take a point $z=(z_1,z_2)\in E$, and fix $k=k_z\in \N\cup \{-1,0\}$ so that $z\in S_k$. Then, we distinguish several cases. Suppose first that $z\in S_0\cup E_l$. Then for every $r>0$ we have
    $$|B(z,r)\cap E| \geq \left|Q(z,\sqrt{2}r)\cap E\right|\geq \dfrac{\sqrt{2}}{2}r \sqrt{2}r=r^2= \dfrac{1}{\pi}|B(z,r)|. $$
 Suppose next that $z\in S_k$ for some $k\in\N$ and let $r>0$. In the case  $a_2\leq \tfrac{\sqrt{2}}{4}r$, we have
    $$|B(z,r) \cap E|\geq |Q(z,\sqrt{2}r)\cap E|\geq (\sqrt{2}/4)r\sqrt{2}r =\dfrac{r^2}{2}=\dfrac{1}{2\pi}|B(z,r)|.$$
    Otherwise, 
    $\tfrac{\sqrt{2}}{4}r<z_2\leq \epsilon_{k-1}$ and we get
    $$|B(z,r) \cap E|\geq |Q(z,\sqrt{2}r)\cap E|\geq |Q(z,\sqrt{2}r/2)\cap E\cap S_k|\geq\dfrac{r^2}{16}=\dfrac{1}{16\pi}|B(z,r)|.$$
    Here we used the fact that
 whenever $b\in S_k\cap E$ and $\tfrac{s}{2}\leq \epsilon_{k-1}$ we have  $|Q(b,s)\cap E\cap S_k|\geq \tfrac{s^2}{8}$.

\medskip

$(2)$ We show that $E$ satisfies the weak $(1,1)$-Poincaré inequality by proving that $E$ supports a Semmes family of curves.

\begin{definition}
    A compact set $E\subset \R^2$ is said to support a {\em Semmes family of curves} if there exists a constant $C\geq 1$ so that for each pair of points  $x, y\in E$ there is a family $\Gamma_{xy}$ of $C$-quasiconvex curves (not necessarily pairwise disjoint) connecting $x$
to $y$ together with a  probability measure $\mu_{xy}$ on $\Gamma_{xy}$ satisfying a Riesz-type inequality; that is, for every Borel set $A\subset E$ we have 
\begin{equation}\label{eq:Semmes_curves_inequality}
\int_{\Gamma_{xy}}\mathcal H^1(\gamma\cap A)\, d\mu_{xy}(\gamma)\leq C\int_{CB_{xy}\cap A}\dfrac{|x-w|}{|B(x,|x-w|)\cap E|}+\dfrac{|y-z|}{|B(y,|y-w|)\cap E|}\, dw .
\end{equation}
Here $CB_{xy}=B(x, C|x-y|)\cup B(y,C|x-y|)$.
\end{definition}
In our case, since the set $E$ is doubling,  inequality \eqref{eq:Semmes_curves_inequality} can be rewritten as
\begin{equation}\label{eq:aa}
\int_{\Gamma_{xy}}\mathcal H^1(\gamma\cap A)\, d\mu_{xy}(\gamma)\leq C\int_{CB_{xy}\cap A}\dfrac{1}{|x-w|}+\dfrac{1}{|y-w|}\, dw \quad \text{for every Borel set $A\subset E$}. \end{equation}
Every Ahlfors regular compact set $E\subset \R^2$ supporting a Semmes family of curves satisfies a weak $(1,1)$-Poincaré inequality. We prefer to omit the proof of this essential fact, which is well-known in the literature; instead we refer to \cite[Section 3]{KLS15} and \cite[Theorem 9.5]{Heinonen01} for the interested reader. More interestingly, we point out that for Ahlfors regular sets the weak $(1,1)$-Poincaré inequality also implies the existence of a Semmes family of curves (see \cite{DEKS2021,FO2019}).

\medskip

The next lemma states the main properties that the family $\Gamma_{xy}$ of $C$-quasiconvex curves will satisfy in order for the inequality \eqref{eq:aa} to hold. We first state and prove the lemma, and then we show how to build the family $\Gamma_{xy}$.

\medskip

Let us fix some useful notation in order to work with curve families.
If $\Gamma=\{\gamma_t\}_{t\in J}$ and $\Gamma'=\{\gamma'_t\}_{t\in J}$ are two curve families indexed by the same set of indexes $J$ and such that the end point of each  $\gamma_t\in\Gamma$ corresponds to the initial or end point of $\gamma'_t$ in $\Gamma'$, we write $\Gamma\cup \Gamma'=\{\gamma_t\cup\gamma'_t\}_{t\in J}$ to denote the family of natural concatenations of the curves $\gamma_t$ with $\gamma'_t$ (these ones maybe with reversed direction). Moreover, for a given family of arc-length parametrized curves $\Gamma=\{\gamma_t\}_{t\in J}$, and with no aim of confusion, we denote its image $\bigcup_{t\in J}\{\gamma_t(s):\, s\in [0,\ell(\gamma_t)]\}$ using the same letter $\Gamma$.

\begin{lemma}\label{lem:Semmes_family}
Let $x,y\in E$ be distinct points and let $D=|x-y|>0$. Suppose that there is a closed segment $J\subset E$ such that $\mathcal H^1(J)=\tfrac{D}{32}$ and assume that we have a family of curves  $\Gamma_{xy}=\{\gamma_z\}_{z\in J}\subset E$ connecting $x$ with $y$, parameterized by arc-length and endowed with the normalized $\mathcal H^1$-measure $\mu=\mu_{xy}$. That is, for every $\Delta\subset \Gamma_{xy}$,
$$
\mu(\Delta)=\mathcal H^1(\{z\in J:\, \gamma_z\in \Delta  \})/\mathcal H^1(J).$$
If there is a constant $C\geq 1$, independent of $x$ and $y$, such that the family of curves $\Gamma_{xy}$ satisfies the following three properties $(1)-(3)$, then $E$ supports a Semmes family of curves. 
\begin{enumerate}
    \item Every curve $\gamma\subset \Gamma_{xy}$ is $C$-quasiconvex.
    \item $\Gamma_{xy}\subset 4B_{xy}\cap E$.

    \item  We can write $\Gamma_{xy}=\Gamma_x\cup\Gamma_y$, where $\Gamma_x=\{\gamma_z|_{[0,s_z]} :\, z\in J\} =\{\gamma^{x}_{z}\}_{z\in J}$ and $\Gamma_y=\{\gamma_z|_{{[s_z,1]}} :\, z\in J\}=\{\gamma^{y}_{z}\}_{z\in J} $, and $\bigcup_{z\in J} \gamma_z(s_z)=J$. Moreover, for $a=x$ and $a=y$ we have that for all $r>0$ and all $z,z'\in J$ for which $\gamma^{a}_{z}\setminus B(a,r)\neq\emptyset$ and $\gamma^{a}_{z'}\setminus B(a,r)\neq \emptyset$, then
\begin{equation}\label{eq:key_property_semmes}
 \dfrac{|z-z'|}{\mathcal  H^1(J)}\leq C\dfrac{\dist(\gamma^{a}_{z}\setminus B(a,r),\gamma^{a}_{z'}\setminus B(a,r)}{r}  
  \end{equation}

\end{enumerate}
\end{lemma}
\begin{proof}
It is enough to prove that $(2)$ and $(3)$ imply that $(\Gamma_{xy},\mu_{xy})$ satisfies inequality \eqref{eq:aa} with constant $8C$ for every Borel subset $A\subset E$. We start by showing that \eqref{eq:key_property_semmes} implies that for $x=a$ and $x=y$ the functions 
 \begin{align*} f^a:\Gamma_a&\longrightarrow J\\
     z&\longmapsto f^a(w)=z,\quad \text{in case}\;w\in\gamma^a_z
 \end{align*}
are $( C\mathcal H^1(J)/r)$-Lipschitz on $\Gamma_a\setminus B(a, r)$ for every $r>0$. Note that $(f^{a})^{-1}(z)=\gamma^{a}_{z}$ for every $z\in J$. Let $w,w'\in \Gamma_x\setminus B(a,r)$ such  that $w\in \gamma^{a}_{z}$ and $w'\in \gamma^{a}_{z'}$ for some $z,z'\in I$. Then, using \eqref{eq:key_property_semmes},
$$|f(w)-f(w')|=|z-z'|\leq C\mathcal H^1(J)\dfrac{\dist(\gamma^{a}_{z}\setminus B(a, r),\gamma^{a}_{z'}\setminus B(a,r))}{ r} .$$
Since $w\in \gamma^{a}_{z} \setminus B(a, r)$ and $w'\in \gamma^{a}_{z'} \setminus B(a, r)$ we conclude that
$$|f(w)-f(w')|\leq C\mathcal H^1(J)\dfrac{|w-w'|}{r}. $$

\noindent Let us finally show that $(\Gamma_{xy},\mu_{xy})$ satisfies \eqref{eq:aa}. We take a Borel set $A\subset E$ and we may assume without loss of generality that $A\subset 4B_{xy}$ and, by dividing $A$ into two parts if necessary, that $|z-y|\geq |x-y|/2$ for all $z\in A$. 
%See the Summer version for more details.
We divide $A$ into different `annuli'
\begin{align*}
 A_j=A\cap (B(x,2^{-j}|x-y|)\setminus B(x,2^{-j-1}|x-y|).
\end{align*}
Next, by making use of the classical coarea formula for Lipschitz functions (see \cite[Theorem 3.10]{EG2015}), 
\begin{align}
    \int_{J}\mathcal H^1(\gamma_z\cap A)&\, d\mu(z)\leq \int_{J}\mathcal H^1(\gamma^{y}_{z}\cap A)+\mathcal H^1(\gamma^{x}_{z}\cap A)\, d\mu(z) \notag \\
    &=\dfrac{1}{\mathcal H^1(J)}\int_{J}\mathcal H^1\left((f^y)^{-1}(z)\cap A\right)\, d\mathcal H^1(z) + \sum_{j\geq -2} \int_{J}\mathcal H^1\left((f^x)^{-1}(z)\cap A_j\right)\, d\mathcal H^1(z) \notag \\
    & = \dfrac{1}{\mathcal H^1(J)} \int_{A} |\nabla f^y(w)| \,dw+  \sum_{j\geq -2} \int_{A_j} |\nabla f^x(w)|\, dw \notag \\
    &\leq \int_{A}\dfrac{2C}{|x-y|}\, dw+ \sum_{j\geq -2} \int_{A_j}\dfrac{C}{2^{-j-1}|x-y|}\, dw \notag \\
    &= \dfrac{2C}{|x-y|} \left(|A|+\sum_{j\geq -2} 2^j|A_j| \right). \label{eq:semmes_paso_1}
\end{align}
In the above estimates, note that we are using that $f^y$ is $\tfrac{2C\mathcal H^1(J)}{|x-y|}$-Lipschitz on $A\cap \Gamma_y$, which follows from \eqref{eq:key_property_semmes} and because $\dist(A,y)\geq |x-y|/2$. On the other hand, using the definition of $A_j$ and that for every $w\in A\subset 4B_{xy}$ we have $|y-w|\leq 4|x-y|$,

\begin{align}
    \int_{A}\dfrac{1}{|y-w|}&+\dfrac{1}{|x-w|}\, dw = \int_{A}\dfrac{1}{|y-w|}\, dw+\sum_{j\geq -2} \int_{A_j}\dfrac{1}{|x-w|}\, dw  \notag \\
    &\geq \int_{A}\dfrac{1}{4|x-y|}\, dw+\sum_{j\geq -2} \int_{A_j}\dfrac{1}{2^{-j}|x-y|}\, dw \geq \dfrac{1}{4|x-y|}\left( |A|+\sum_{j\geq -2} \left(2^j|A_j| \right)\right). \label{eq:semmes_paso_2}
\end{align}
It is clear that \eqref{eq:semmes_paso_1} and \eqref{eq:semmes_paso_2} yield to \eqref{eq:aa}, so we are done.
\end{proof}

Next, we explain how to build a curve family $\Gamma_{xy}$ connecting arbitrary  points $x$ and $y$ such that all assumptions $(1)-(3)$ of Lemma \ref{lem:Semmes_family} are satisfied for a constant $C\geq 1$ independent of the points. This will conclude the proof of Proposition  \ref{lem:properties_of_example} (2).

\bigskip

 $\bullet$ \underline{{\bf Construction of the Semmes family of curves $(\Gamma_{xy},\mu_{xy})$.}}

\medskip

\noindent Fix two distinct points $x=(x_1,x_2)$, $y=(y_1,y_2)\in E$, and denote
$$D=|x-y|\quad ;\quad Q=(0,1)\times(0,1/2).$$
In the case  $x,y\notin Q$, the existence of a Semmes family of curves $\Gamma_{xy}\subset E$ is clear because $E\setminus Q=([-1,2]\times [-1,1])\setminus Q$  is a uniform set and thus satisfies the weak $(1,1)-$Poincaré inequality. Hence, we may assume $x\in Q$ and we next distinguish two different situations, depending on whether
\begin{equation}\label{eq:sit_1_semmes}
     y_2>0 \quad \text{and}\quad  D<\frac{1}{2}\min\{x_2, y_2\}
\end{equation}
holds. If \eqref{eq:sit_1_semmes} holds, it means that the removed triangles that are close to both $x$ and $y$ have a large size compared to $D=|y-x|$. This allows for an easy argument in order to define a Semmes family of $2$-quasiconvex curves on $\tfrac{3}{2}B_{xy}$. In fact, in this situation, one can prove that the open set $\tfrac{3}{2}B_{xy}$ is intersected by at most three triangles of the family $\{I_{k,j}\}$, hence a Semmes family of $2$-quasiconvex curves is easily constructed within $\tfrac{3}{2}B_{xy}\cap E$.
% More details inthe Summer version.
    
\medskip

If \eqref{eq:sit_1_semmes} fails, we instead utilize the lower subset $E_l:=[-1,2]\times [-1,0]\subset E$ as a 'mirror', which provides sufficient room to construct the required Semmes families of curves\footnote{We note that the upper subset $E_u:=E\cap [-1,2]\times [0,1]$ does not support any weak $(1,1)$-Poincaré inequality. We leave the details of this fact for the interested reader.}. More precisely, given $x=(x_1,x_2)$ and $y=(y_1,y_2)$, we associate the segments $I_x, I_y\subset E_l$, defined by
$$ 
I_x:=\left[x_1-\tfrac{D}{64},x_1+\tfrac{D}{64}\right]\times \{-\tfrac{D}{2}\}\quad ; \quad I_y:=  \left[y_1-\tfrac{D}{64},y_1+\tfrac{D}{64}\right]\times \{-\tfrac{D}{3}\} .
$$
Without loss of generality, in case $y_2>0$, we assume $\min\{x_2,y_2\}=x_2$. Note that 
$\dist(x,I_x)\sim D$ because
\begin{equation}\label{eq:distance_x_I_x}
 \tfrac{1}{2}D\leq  \dist(x,I_x)\leq x_2+\tfrac{1}{2}D\leq 2D+\tfrac{1}{2}D=\tfrac{5}{2}D 
\end{equation}
and if $y_2>0$ and $D<\frac{1}{2}\min\{x_2, y_2\}=x_2$, we also have $ \dist(y,I_y)\sim D$ because
\begin{equation}\label{eq:distance_y_I_y}
\tfrac{1}{3}D\leq  \dist(y,I_y)\leq   \dist(y,(y_1,x_2))+\dist((y_1,x_2),I_y)=|y_2-x_2|+x_2+ \tfrac{1}{3}D\leq \tfrac{10}{3}D. 
\end{equation}
We let $J:=I_x$ be our set of indexes (note that $\mathcal H^1(J)=D/32$) and we build $\Gamma_{xy}=\{\gamma_z\}_{z\in J=I_x}$ according to steps (A)-(B)-(C) below. But first, we state the next useful fact, that will be needed in our argument below.

\begin{fact}\label{fact} 
Let $[0,L]^2\subset\R^2$ be a square with side-length $L>0$. And let $I$ be an open equilateral triangle with vertexes $v_1=(L/2,0)$, $v_2=(L/2-L/\sqrt{3} ,L/3) $ and $v_3=( L/2+L/\sqrt{3},L/3) $. Then, for some uniform constant $C\geq 1$,
\begin{enumerate}
    \item  If $x\in R$ and $J\subset R$ is a segment with $L/64<\mathcal H^1(J)<L/8$, then there  exists a family of curves $\{\gamma_z\}\subset F$ connecting $x$ with every $z\in J$ such that $\ell(\gamma_z)\leq C \mathcal H^1(J)$.
\item  If $x=(x_1,x_2)\in [0,L]\times[L/3,L]$ and $J\subset [0,L]\times\{0\}$ is a segment with $L/64<\mathcal H^1(J)<L/8$ and mid-point $(x_1,0)$, then  there  exists a family of curves $\{\gamma_z\}\subset R\setminus I$ connecting $x$ with every $z\in J$ such that $\ell(\gamma_z)\leq C |x-z|$.
\end{enumerate}
Moreover, in both situations we can enforce that $$  \dfrac{|z-z'|}{\mathcal  H^1(J)}\leq C\dfrac{\dist(\gamma_{z}\setminus Q(x,r),\gamma_{z'}\setminus Q(x,r))}{r}\quad \forall z,z'\in J,\;\;\forall r\in(0,\dist(x,J)).$$ 
\end{fact}
\begin{proof}
The details of this geometrical proof are left to the reader. We only note that in case $(2)$, if $J\cap \overline I\neq \emptyset$, the family of curves will look like two pencil of curves $P_{1}$ and $P_2$ starting from $x$ and traveling through $R\setminus I$ from the left and right hand sides of $I$ to each of the two connected components of  $J\setminus \overline I$. See Figure \ref{fig:exampe_step_1}.
\end{proof}

\noindent (A)  Connect $x$ with $I_x$ through a family of curves $\Gamma_x=\{\gamma^x_z\}_{z\in I_x}$ whose properties are encoded in the next claim.
\begin{claim}\label{claim:semmes}
{\em There exists a constant $C\geq 1$ and a family $\Gamma_x=\{\gamma^x_z\}_{z\in I_x}$ of unit-speed (injective) curves $\gamma^x_z$ connecting $x$ with every $z\in I_x$ such that \begin{enumerate}
    \item[(i)] $\ell(\gamma^x_z)\leq C|x-z|$ and $\gamma^x_z\subset 4B_{xy}$ for every $z\in I_x$.
    \item[(ii)] The family $\Gamma_x$ satisfies property \eqref{eq:key_property_semmes} from Lemma \ref{lem:Semmes_family} with the constant $C$.
    \end{enumerate}}
\end{claim}
\noindent We note that the family of curves $\Gamma_x$ could be taken to be $2$-quasiconvex, but for the sake of simplicity, we do not aim for the best constant $C$ in Lemma \ref{claim:semmes}. The proof of the claim appears below. We only mention that from Claim \ref{claim:semmes} (i) and \eqref{eq:distance_x_I_x}, 
    \begin{equation}\label{eq:quasiconvexity_Gamma_x}
        \ell(\gamma^x_z)\leq C|x-z|\leq C\left( \dist(x, I_x)+ \dfrac{\text{diam}(I_x)}{2}\right)\leq C\left(\tfrac{5}{2}D+\tfrac{D}{64}\right)\lesssim D\quad \forall z\in I_x.
    \end{equation}

\smallskip

\noindent (B) Connect $y$ with $I_x$ using another family of adequate curves $\Gamma_y=\{\gamma^y_z\}_{z\in I_x}$.

\medskip

\noindent For every $z\in I_x$ we build unit-speed (injective) curves $\gamma^y_z$ connecting $y$ with $z$ such that $\ell(\gamma^y_z)\leq C |y-z|$ for some constant $C\geq 1$, $\gamma^y_z\subset 4B_{xy}$ and such that \eqref{eq:key_property_semmes}  holds for the family $\Gamma_y=\{\gamma^y_z\}_{z\in I_x}$. 

\smallskip

\begin{enumerate}
    \item[(B.1)] In  case  $y_2<0$, note that we have $$\dist(y,I_x)\leq |(y_1,y_2)-(x_1,y_2)|+\tfrac{D}{2} \leq |y-x|+\tfrac{D}{2}\leq \tfrac{3}{2}D.$$
    This is an easy situation (Fact \ref{fact} (1)) where there exists some constant $C\geq 1$ such that we can connect $y$ with 
$I_x$ through a family of curves $\Gamma_y=\{\gamma^y_z\}_{z\in I_x}\subset 4B_{xy}\cap E_l$ satisfying property \eqref{eq:key_property_semmes}  and so that for all $z\in I_x $  
\begin{equation}\label{eq:dist_y_I_x}
\ell(\gamma^y_z)\lesssim D.    
\end{equation}
Note that the previous curves need not be quasiconvex, since $y$ could belong to $I_x$ in some situations. 

\smallskip

\item[(B.2)] In case $y_2>0$ and $D<\frac{1}{2}\min\{x_2, y_2\}$ we let $\Gamma_y=\Gamma^1_y\cup\Gamma^2_y$. The family $\Gamma^1_y=\{\gamma^{1,y}_{z'}\}_{z'\in I_y}$ consists of $C$-quasiconvex curves connecting $y$ with $I_y$, which are constructed following the same method as that of $\Gamma_x$ (see Claim \ref{claim:semmes}). In particular, the family $\Gamma^1_y\subset 4 B_{xy}$ satisfies property \eqref{eq:key_property_semmes}  and for every $z'\in I_y$, by using \eqref{eq:distance_y_I_y},
\begin{equation}\label{eq:est_curve_Gamma'_y}
\ell(\gamma^{1,y}_{z'})\leq C |y-z'| \leq C \left( \dist(y,I_y)+ \dfrac{\text{diam}(I_y)}{2}\right) \leq C\left( \tfrac{10 D}{3}+\tfrac{D}{64}\right)\lesssim 7D.
\end{equation} 
 The family $\Gamma^2_y=\{ \gamma^{2,y}_{z}\}_{z\in I_x}$ consists of straight line-segments connecting every point $z'\in I_y$ with every  $z\in I_x$, following a linear bijection between these two sets $\phi:I_x\to I_y$. That is, $\Gamma^2_y =\{ \overline{\phi(z)z}\}_{z\in I_x}$. Here, for $a,b\in \R^2$, we write $\overline{ab}=\{at+b(1-t):\,t\in [0,1]\}$ to denote the segment between $a$ and $b$, which may also be understood as a curve connecting $a$ with $b$.\\
 Furthermore, for all $z\in I_x$,
\begin{equation}\label{eq:dist_z-phi(z)}
\ell\left(\overline{\phi(z)z}\right)=|\phi(z)-z|\leq |y_1-x_1|+\tfrac{D}{6}\leq \tfrac{7D}{6}.    
\end{equation}
We finally define
$\Gamma_y=\Gamma^1_y\cup \Gamma^2_y=\{\gamma^y_z\}_{z\in J},$
where  $\Gamma_y\subset 4 B_{xy}$. It is also  clear that, thanks to the geometric location of $I_x$ and $I_y$ (this is, $\dist(I_x,I_y)\sim D$, $\pi_2(I_x)=-D/2$ and $\pi_2(I_y)=-D/3$) and using that $\Gamma^1_y$ satisfies \eqref{eq:key_property_semmes}, the family $\Gamma_y$ also satisfies \eqref{eq:key_property_semmes}. Finally, by using \eqref{eq:est_curve_Gamma'_y} and \eqref{eq:dist_z-phi(z)}, for every $z\in I_x$,
\begin{equation}\label{eq:est_curve_Gamma_y}
\ell(\gamma^{y}_{z})=\ell\left(\gamma^{1,y}_{\phi(z)} \right) + |\phi(z)-z|\leq C |y-\phi(z)|  +|\phi(z)-z|\leq CD+\tfrac{7}{6}D\lesssim D.
\end{equation}

\end{enumerate}

\medskip

\noindent (C) We define the family $\Gamma_{xy}=\Gamma_x\cup \Gamma_y=\{\gamma_t\}_{t\in I}=\{\gamma_z=\gamma^x_z\cup \gamma_z^y\}_{z\in I_x}$  as the concatenation of the curves from $\Gamma_x$ with those from $\Gamma_y$, which join $x$ with $y$. Since $\Gamma_x,\Gamma_y\subset 4B_{xy}$ we have $\Gamma_{xy}\subset 4B_{xy}$. Moreover, for every $\gamma_z=\gamma^x_z\cup \gamma_z^y$, using \eqref{eq:quasiconvexity_Gamma_x}, and  \eqref{eq:dist_y_I_x} or \eqref{eq:est_curve_Gamma_y},
\begin{align*}
\ell(\gamma_z)&\leq \ell(\gamma^x_z)+\ell(\gamma^y_z)\lesssim  D= |x-y|.
\end{align*}
Recall also that both $\Gamma_x$ and $\Gamma_y$ satisfy property \eqref{eq:key_property_semmes}.
In summary, all properties $(1)-(3)$ of Lemma \ref{lem:Semmes_family} are satisfied for $\Gamma_{xy}$. It only remains to prove the Claim \ref{claim:semmes}.

\smallskip

\begin{proof}[Proof of Claim \ref{claim:semmes}]
Recall that, for some constant $C\geq 1$, we aim to define a family of $C$-quasiconvex curves $\Gamma_x=\{\gamma_z\}_{z\in I_x}\subset 4B_{xy}$ (hereafter, we omit the subscript $x$) connecting $x$ with every  $z\in I_x$ such that \eqref{eq:key_property_semmes} holds for $C$. In fact, the construction can be tailored so that the curves are  $2$-quasiconvex, which directly gives $\Gamma_x\subset 4B_{xy}$. An illustration of this construction is given in Figure \ref{fig:example_final_illustration}.

  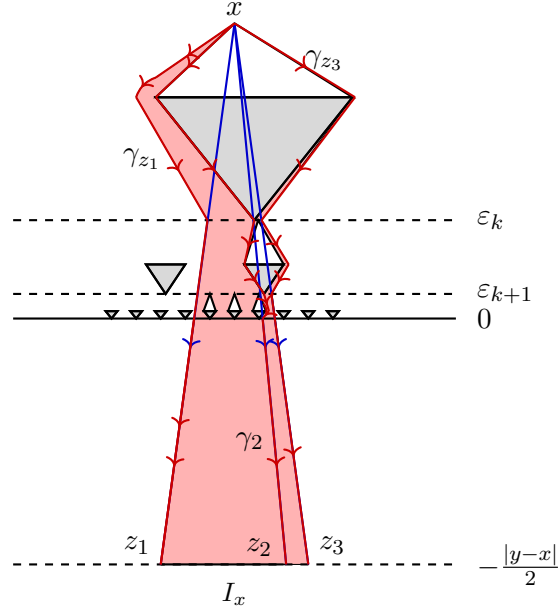
\begin{figure}[htbp]
    \begin{minipage}{0.45\textwidth}
\begin{tikzpicture}[
    scale=0.65,
    % Define styles for the rays to automatically place arrows along the paths
    redray/.style={draw=red!80!black, thick, decoration={markings, mark=at position 0.6 with {\arrow{>}}}, postaction={decorate}},
        blueray/.style={draw=blue!80!black, thick, decoration={markings, mark=at position 0.6 with {\arrow{>}}}, postaction={decorate}},
    blackray/.style={draw=black, thick, decoration={markings, mark=at position 0.6 with {\arrow{>}}}, postaction={decorate}},
    % Style for the arrowheads on the short middle segments
    midarrow/.style={decoration={markings, mark=at position 0.5 with {\arrow{>}}}, postaction={decorate}}
]

    % --- Coordinates & Main Points ---
    \coordinate (X) at (0, 8);
\node at (0,8.3) {$x$};

\draw[fill=red!30] (X) -- (-1.6, 6.9) to[bend left=-25] (-2, 6.5) --  (-0.55, 4) -- (-1.5, -3) -- (1.5,-3) --  (0.76,2.5)   --  (1.1,3.1)   -- (0.55,4)  --  (2.4, 6.5)   --  cycle;

\draw[thick, fill=gray!30] (-1.8, 3.1) -- (-1.0, 3.1) -- (-1.4, 2.5) -- cycle;
\draw[thick, fill=gray!30] (-1.6, 6.5) -- (2.4, 6.5) -- (0.4, 4) -- cycle;

\draw[thick, fill=gray!30]  (0.2, 3.1) -- (1.0, 3.1) -- (0.6, 2.5) -- cycle;

\draw[thick, fill=white]  (X) -- (2.4, 6.5)-- (-1.6,6.5) -- cycle;
\draw[thick, fill=white]  (0.48,4) -- (1,3.1)-- (0.2,3.1) -- cycle;

\draw[thick, fill=white]  (-0.5,2.5) -- (-0.63, 2.15)-- (-0.38, 2.15) -- cycle;
\draw[thick, fill=white]  (0,2.5) -- (-0.13, 2.15) -- (0.13, 2.15) -- cycle;
\draw[thick, fill=white] (0.5,2.5) -- (0.37, 2.15) -- (0.64, 2.15)  --  cycle;

 \foreach \x in {-2.5, -2, -1.5, -1, -0.5, 0, 0.5, 1, 1.5, 2} { \draw[thick, fill=gray!30]     (\x-0.12, 2.15) -- (\x+0.12, 2.15) -- (\x, 2) -- cycle;}

    % --- Horizontal Boundary Lines ---
    \draw[dashed, thick] (-4.5, 4) -- (4.5, 4) node[right, xshift=4pt] {$\varepsilon_k$};
    \draw[dashed, thick] (-4.5, 2.5) -- (4.5, 2.5) node[right, xshift=4pt] {$\varepsilon_{k+1}$};
    \draw[thick] (-4.5, 2) -- (4.5, 2) node[right, xshift=4pt] {$0$};
    \draw[dashed, thick] (-4.5, -3) -- (4.5, -3) node[right, xshift=4pt] {$-\frac{|y-x|}{2}$};

    % --- Triangles ---
    % Large top inverted triangle
    \draw[thick] (-1.6, 6.5) -- (2.4, 6.5) -- (0.4, 4) -- cycle;

    % Two small inverted triangles
    \draw[thick] (-1.8, 3.1) -- (-1.0, 3.1) -- (-1.4, 2.5) -- cycle;
    \draw[thick] (0.2, 3.1) -- (1.0, 3.1) -- (0.6, 2.5) -- cycle;

    % Row of small triangles along the '0' line
    \foreach \x in {-2.5, -2, -1.5, -1, -0.5, 0, 0.5, 1, 1.5, 2} {
        \draw[thick] (\x-0.12, 2.15) -- (\x+0.12, 2.15) -- (\x, 2) -- cycle;
    }

    % --- Segment (I'_x) ---
    % Tilted rectangle to match the sketch
    \coordinate (UL) at (-1.5, -3);
    \coordinate (UR) at (1, -3);
    \draw[thick] (UL) -- (UR) ;

\node at (0,-3.6) {\small{$I_x$}};
\node at (0.5,-2.65) {$z_2$};

 \draw[blueray] (X) -- (-1.5, -3) node[above left, text=black] {$z_1$};
\draw[blueray] (X) -- (1.05, -3);
\draw[blueray] (X) -- (1.5, -3)  node[above right, text=black] {$z_3$};

\draw[redray] (X) -- (-1.6, 6.5);
\draw[redray] (-1.6, 6.5) -- (0.4, 4);
\draw[redray] (0.4,4) -- (0.2, 3.1);
\draw[redray] (0.2,3.1) -- (0.6, 2.5);
\draw[redray] (0.6,2.5) -- (0.7, 2.15);
\draw[redray] (0.7,2.15) -- (0.57, 2);

\draw[redray] (X) -- (-1.6, 6.9);
\draw[redray] (-1.6,6.9) to[bend left=-25] (-2, 6.5);
\draw[redray] (-2, 6.5) -- (-0.55, 4) node[midway, left] {$\gamma_{z_1}$};
\draw[redray] (-0.55,4) -- (-1.5, -3);

%\draw[redray] (-0.5,2.5) -- (-0.63, 2.15);
%\draw[redray] (-0.5,2.5) -- (-0.38, 2.15);

%\draw[redray] (0,2.5) -- (-0.13, 2.15);
%\draw[redray] (0,2.5) -- (0.13, 2.15);

%\draw[redray] (0.5,2.5) -- (0.37, 2.15);
%\draw[redray] (0.5,2.5) -- (0.64, 2.15);

 %\draw[redray] (X) -- (2.4, 6.6)  
  \draw[redray] (X) -- (2.45, 6.5) node[right, midway] {$\gamma_{z_3}$};
  \draw[redray] (2.45, 6.5) -- (0.55, 4);
  \draw[redray] (0.55, 4) -- (1.1, 3.1);
  \draw[redray] (1.1, 3.1) -- (0.75, 2.5);
 \draw[redray] (0.75, 2.5) -- (1.5, -3);

\draw[redray] (-0.82, 2) -- (-1.5, -3);

 \draw[redray] (0.57, 2) -- (1.05, -3) node[midway,left] {$\gamma_2$};

    % Central straight red ray

    % --- Ray Paths ---
    % Left bundle of red rays

\end{tikzpicture}

\caption{Family of curves $\Gamma_x$}
\label{fig:example_final_illustration}
\end{minipage}
\end{figure}

For the given point $x=(x_1,x_2)\in Q$ we fix $k_x\in\N$ so that $x_2\in [\epsilon_{k_x},\epsilon_{k_x-1})$. We begin by employing the standard method for constructing a Semmes pencil of curves in the Euclidean setting (see \cite[Example 6.1]{KLS15}), connecting $x$ to each point in $I_x$ via straight line segments. Call this family 
$$\{\overline{xz}\}_{z\in I_x}.$$
It is not difficult to see that $\{\overline{xz}\}_{z\in I_x}$ satisfies \eqref{eq:key_property_semmes} for some constant $C\geq 1$. However, these curves may intersect some of the removed triangles, so in order to handle this difficulty, we will perturb these segments as follows. Consider the `pencil' $$P_x=\bigcup\{\overline{xz}:\, z\in I_x\},$$
and define the sets
\begin{align*}
\mathcal A_k&=[0,1]\times \{t_k\epsilon_k\} &;\quad \mathcal B_k&=[0,1]\times \{\epsilon_k\}\;\;&k\in\N \\
 A_{k}&=\mathcal A_k\cap P_x  &;\quad   B_{k}&=\mathcal B_k\cap P_x \; &k\geq k_x.  
\end{align*}
 The curves $\Gamma_x=\{\gamma_z:\, z\in I_x\}$ will be defined by modifying the segments $\overline{xz}$ on each strip $S_k=([-1,2]\times [\epsilon_{k},\epsilon_{k-1}])$ for every $k\geq k_x$. We leave these segments unaltered on the lower part $E_l= [-1,2]\times [-1,0]$, as well as on the boundaries $\partial S_k$. That is 
\begin{equation}\label{eq:def_Gamma_x}
\gamma_z=\left(\bigcup_{k\geq k_x}\gamma_{z,k}\right)\cup \left(\overline{xz}\cap E_l\right),\quad z\in I_x,
\end{equation}
where the curves $\{\gamma_{z,k}\}_{k\geq k_x}\subset S_{k}\cap E$ are defined such that, for some uniform constant $C\geq 1$, they satisfy the following properties:

\smallskip

\begin{enumerate}
    \item[(i)] For all $k\geq k_x$ we have $\gamma_z\cap B_{k}=\overline{xz}\cap B_{k}$. 
        \item[(ii)] For all $k\geq k_x$.  $\ell(\gamma_{z,k})\leq C\ell(\,\overline{xz} \cap S_k)$. (One may take $C=2$).
    \item[(iii)] For every $k>k_x$ we have $\gamma_{z,k}=\gamma^{1}_{z,k}\cup \gamma^{2}_{z,k}$ as the union of  two segments with $\gamma^{1}_{z,k}\cap \gamma^{2}_{z,k}\in \mathcal A_k$. 
    \item[(iv)] The family $\Gamma_x=\{\gamma_z\}_{z\in I_x}$ satisfies property \eqref{eq:key_property_semmes} for the constant $C$.
\end{enumerate}

\begin{remark}
    Property \eqref{eq:key_property_semmes} from Lemma \ref{lem:Semmes_family} is easier to check if we use the $\ell_1$-distance, which is bi-Lipschitz equivalent to the euclidean distance and  if we use squares $Q(x,r)$ instead of balls $B(x,r)$ (note that $Q(x,r/2)\subset B_{\ell_1}(x,r)\subset Q(x, 2r)$ for all $r>0$). Hence, throughout the remainder of the proof of Claim \ref{claim:semmes}, we adopt this new metric.
\end{remark}

To define the modified curves $\{\gamma_z\}$ we need to distinguish the case $k=k_x$ from the case $k>k_x$. Abusing of notation, we may write $\gamma_z=\gamma_{z'}$ whenever $z'\in\gamma_z$. Also, if for a given $z\in I_x$ we set $z_k=\gamma_z\cap B_k=[x,z]\cap B_k$, we will write $\gamma_{z,k}=\gamma_{z_k,k}$ to mean $\gamma_z\cap S_k=\gamma_{z_k}\cap S_k$.

\medskip

{\bf  Case $\mathbf{k=k_x}$} (see Figure \ref{fig:exampe_step_1}): Connect $x$ with every point $z\in B_{k_x}$ through a family of $2$-quasiconvex curves $\{\gamma_{z,k_x}\}_{z\in B_{k_x}}\subset S_{k_x}\cap E$   such that $\ell(\gamma_{z,k_x})\leq 2\ell([x,z] \cap S_{k_x})$ and also for every $z,z'\in B_{k_x}$
\begin{equation}\label{eq:key_property_step_1}
   \dfrac{ |z-z'|}{\mathcal H^1(B_{k_x})}\lesssim \dfrac{\dist_{\ell_1}(\gamma_{z,k_x}\setminus Q(x, 2r),\gamma_{z',k_x}\setminus Q(x,2r))}{r}\quad \forall r\in (0,x_2-\epsilon_{k_x}.
   )
\end{equation}
In case $P_x\cap S_{k_x}\subset E$  we do not modify the initial pencil of curves and we let $\gamma_{z,k_x}=\overline{xz}\cap S_{k_x}$. Otherwise, we must have $x_2>t_{k_x-1}\epsilon_{k_x-1}$. Note that by \eqref{eq:distance_x_I_x}, we have $D/2\leq \dist(x,I_x)\leq 5D/2$, so the angle of the pencil $P_x$ at its vertex $x$ must be some value $\alpha\in \left[\pi/80, \pi/16\right]$, and since $x_2\in (t_{k_x-1}\epsilon_{k_x-1},\epsilon_{k_x-1}]$, the set $B_{k_x}$ may contain at most one lower vertex of a unique triangle from the family $\{I_{{k_x-1},j}\}_{j}$.  Call this vertex $v$ and the corresponding triangle $I$. We now use Fact \ref{fact} (2) for the segment $B_{k_x}\subset Q_x$ lying on the lower side of some square $Q_x\subset S_{k_x}$. Note that the square $Q_x$ has side-length $\ell(Q_x)=\epsilon_{k_x-1}-\epsilon_{k_x} $ and $\dist(x, B_{k_x})\geq \mathcal H^1(B_{k_x})\sim \epsilon_{k_x-1}$ because we are assuming $x_2>t_{k_x -1}\epsilon_{k_x-1}$. This way, we construct a  collection of curves
$\{\gamma_{z,k_x}\}_{z\in B_{k_x}}\subset Q_x\setminus I\subset S_{k_x}$ such that 
$$\ell(\gamma_{z,k_x})\leq C\ell(\,\overline{xz}\cap S_{k_x})$$
and \eqref{eq:key_property_step_1} holds for every $r\in (0, x_2-\epsilon_{k_x})$.

\medskip

%And we can also enforce, if necessary, that $$\dfrac{\mathcal H^1(P_1\cap \mathcal A_{k_x -1})}{\mathcal H^1(B_{k_x,1})}\sim \dfrac{\mathcal H^1(P_2\cap \mathcal A_{k_x -1})}{\mathcal H^1(B_{k_x,2})} $$mand that $\gamma_{z,k_x}\cap [-1,2]\times [\epsilon_{k_x}, t_{k_x-1}\epsilon_{k_x-1}] $ are just straight line-segments.

{\bf Case $\mathbf{k>k_x}$} (see Figure \ref{fig:example_step_2}):  For every $k> k_x$, and every $z\in I_x$ we connect $z_{k-1}:=\overline{xz}\cap B_{k-1}$ with $z_k:=\overline{xz}\cap B_{k}$ by means of a curve $\gamma_{z,k}\subset E$ such that  $$\ell(\gamma_{z,k})\leq 2\ell(\overline{xz} \cap S_{k_x})\sim \epsilon_{k-1}$$ and so that for all $z,z'\in B_{k+1}$, 
 $(v_1,v_2)\in\gamma_{z,k}$ and $(v'_1,v_2)\in\gamma_{z',k}$, 
\begin{equation}\label{eq:case_k>kx}
|v_1-v'_1|\geq \dfrac{1}{2}\dist\left(\overline{xz}\cap ([-1,2]\times\{v_2\}),\overline{xz'}\cap ([-1,2]\times \{v_2\})\right) .
\end{equation}
The curves $\{\gamma_{z,k}\}$ can be constructed to be mutually disjoint, each consisting of the union of two segments that intersect on $A_{k-1}$. This modification should be carried out so that $\ell(\gamma_{z,k})\leq 2\ell([x,z] \cap S_{k_x})$ and such that the family $\{\gamma_{z,k}\}$ is not too highly compressed, so that \eqref{eq:case_k>kx} remains valid. The geometric separation of the family of removed triangles $\{I_{k-1,j}\}_{j=1}^{n_{k-1}}$ ensures that such a construction is always feasible.
%More details in summer version

\medskip

After having defined the curves $\gamma_{z,k}$ on each strip $S_k$ for every $k\geq k_x$, recall that we  define the family $\{\gamma_z\}_{z\in I_x}$ as in \eqref{eq:def_Gamma_x}. By construction, properties (i), (ii), (iii) are clear. For example,
$$\ell(\gamma_z)=\sum_{k\geq k_x}\ell(\gamma_z,k)+\ell(\,\overline{xz}\cap E_l)\leq  C\sum_{k\geq k_x}\ell(\,\overline{xz})+\ell(\,\overline{xz}\cap E_l)=C |x-z|$$
It remains to check (iv). We start by noticing that for $z,z'\in I_x$ and some $k\geq k_x$, if we call $z_k=  \gamma_z\cap B_k $ and $z'_k= \gamma_{z'}\cap B_k$, then $\gamma_z=\gamma_{z_k}$,  $\gamma_{z'}=\gamma_{z'_k}$ and 
$$\dfrac{|z-z'|}{\mathcal H^1(J)}=\dfrac{|z_k-z'_k|}{\mathcal H^1(B_k)}. $$
Henceforth, for any $z,z'\in I_x$ and given any $r>0$ we distinguish three cases:

\begin{itemize}
    \item If $r\in (0,x_2-\epsilon_{k_x})$ we have, by using \eqref{eq:key_property_step_1}, %{\color{red} Explain more the first inequality?}
\begin{align*}
\dist_{\ell_1}(\gamma_z\setminus Q(x,2r),\gamma_{z'}\setminus Q(x,2r))&\gtrsim \dist_{\ell_1}(\gamma_{z_{k_x},k_x}\setminus Q(x,2r),\gamma_{z'_{k_x},k_x}\setminus Q(x,2r)) \\
&\gtrsim \dfrac{r }{\mathcal H^1(B_{k_x})}|z_{k_x}-z'_{k_x}| =\dfrac{r}{\mathcal H^1(J)}|z-z'| .
\end{align*}
\item If $r\in [x_2-\epsilon_{k_x},x_2)$, we have
\begin{equation}
\dist_{\ell_1}(\gamma_z\setminus Q(x,2r),\gamma_{z'}\setminus Q(x,2r))=\min_{t\geq r}\{ \dist_{\ell_1}(\gamma_z\cap ([-1,2]\times \{x_2-t\}),\gamma_{z'}\cap ([-1,2]\times \{x_2-t\}))\}, \end{equation}
which  follows because for every $z\in I_x$, the curve $\gamma_z\cap [-1,2]\times [ -1,\epsilon_{k_x}]$ maintains a slope relative to the vertical axis within the range $[-\pi/4,\pi/4]$. Then, by also using
\eqref{eq:case_k>kx},
\begin{align*}
    \dist_{\ell_1}(\gamma_z\setminus Q(x,2r),\gamma_{z'}\setminus Q(x,2r))&\geq \dist\left(\,\overline{xz}\cap ([-1,2]\times\{x_2-r\}),\overline{xz'}\cap ([-1,2]\times \{x_2-r\})\right)\\
    &\gtrsim\dfrac{r}{C\mathcal H^1(J)}|z-z'|.
\end{align*}

 \item If $r\in [x_2,x_2+D/2)$, then 
    $$\dist_{\ell_1}(\gamma_z\setminus Q(x,2r),\gamma_{z'}\setminus Q(x,2r))= \dist_{\ell_1}(\overline{xz}\setminus Q(x,2r),\overline{xz'}\setminus Q(x,2r))\gtrsim \dfrac{r |z-z'|}{C\mathcal H^1(J)}.$$
\end{itemize}

\end{proof}

\begin{figure}[htbp]

\begin{minipage}{0.45\textwidth}
     
    \begin{tikzpicture}[
    scale=0.75,
    % Estilo para colocar una flecha en el centro de la línea
    mid_arrow/.style={
        postaction={decorate},
        decoration={
            markings,
            mark=at position 0.55 with {\arrow{Stealth[scale=1.1]}}
        }
    },
    % Estilo para las líneas rojas (un rojo un poco más oscuro para que se lea bien)
   > = stealth,
    curve_red/.style={thick, red!70!black},
    curve_blue/.style={thick, blue!70!black},
]

% ==========================================
% 1. LÍNEAS BASE (Sólida y punteada)
% ==========================================
\coordinate (L_left) at (1, 0);
\coordinate (L_right) at (7.4, 0);
\coordinate (D_left) at (1, 3.5);
\coordinate (D_right) at (7.4, 3.5);

\draw (L_left) -- (L_right) node[right] {$\epsilon_{k_x}$};
\draw[dashed] (D_left) -- (D_right) node[right] {$t_{k_x-1}\epsilon_{k_x-1}$};

% ==========================================
% 2. COORDENADAS PRINCIPALES
% ==========================================
% Puntos en la línea inferior
\coordinate (V0) at (3.8, 0);
\coordinate (V)  at (5.2, 0);
\coordinate (V1) at (5.0, 0);
\coordinate (V2) at (3,3.5);
\coordinate (V3) at (7,3.5);
\coordinate (V4) at (7.1,3.5);

% Punto superior (x)
\coordinate (X) at (4.5, 5.2);

% Puntos de intersección en la línea punteada para las funciones
\coordinate (P0) at (2.4, 3.5); % Inicio de phi_{v_0}

% ==========================================
% 3. TRIÁNGULOS Y LÍNEAS NEGRAS
% ==========================================

% TRIANGULO
\draw[thick, fill=gray!30] (V1) -- (V2) -- (V3) -- cycle;

\draw[fill=red!30] (X) -- (2.7, 4.1) to[bend left=-25] (P0) -- (V0) -- (V1) -- (V2) -- cycle;
\draw[fill=red!30] (X) -- (V4) -- (V) -- (V1) -- (V3) -- cycle;

\draw[curve_blue, postaction={decorate}, 
    decoration={
        markings,
        mark=at position 0.25 with {\arrow{>}},
        mark=at position 0.75 with {\arrow{>}}
    }
]  (X)  -- (V);
\draw[curve_blue, postaction={decorate}, 
    decoration={
        markings,
        mark=at position 0.25 with {\arrow{>}},
        mark=at position 0.75 with {\arrow{>}}
    }
]  (X)  -- (V0);

% Etiquetas de los nodos inferiores y punto x
\fill (V0) circle (1.5pt) node[below left] {$z_0$};
\fill (V1) circle (1.5pt) node[below] {$v$};
\fill (V) circle (1.5pt) node[below right] {$z_1$};
\node at (4.6,-0.55) {\small{$B_{k_x}$}};

\node at (1.6,3.8) {$\mathcal A_{k_x-1}$};
\node[above right] at (X) {$x$};

\draw[thick] (P0) -- (V2) ;

% ==========================================
% 4. TRAYECTOS ROJOS (con flechas)
% ==========================================
% Caídas principales
\draw[curve_red, mid_arrow] (P0) -- (V0) node[midway, left, text=black] {$\gamma_{z_0}$};
%\draw[curve_red, mid_arrow] (P1) -- (V) node[midway, left, text=black] {$\varphi_{v_1}$};

\draw[curve_red, mid_arrow] (V4) -- (V) node[midway, right, text=black] {$\gamma_{z_1}$};
\draw[curve_red, mid_arrow] (X) -- (V3);
\draw[curve_red, mid_arrow] (X) -- (V4);

% Estructura roja superior (Lift/Mapeo)
\coordinate (MidTop) at (V2);  % Desplazamiento sobre la línea punteada
\coordinate (HighTop) at (2.7, 4.1) ; % Subida vertical

% Línea quebrada base (horizontal y vertical)
\draw[curve_red, dashed]  (MidTop) -- (HighTop);

% Flecha curva desde P0 hasta HighTop
\draw[curve_red, ->, >=Stealth] (HighTop) to[bend left=-25] (P0) ;
\draw[curve_red, ->, >=Stealth] (2.85,3.75) to[bend left=-25] (2.7,3.5) ;

% Flechas hacia y desde X
\draw[curve_red, mid_arrow] (X) -- (HighTop) ;
\draw[curve_red, mid_arrow] (X) -- (2.85,3.75) ;
\draw[curve_red, mid_arrow] (2.7,3.5) -- (4.5,0) ;
\draw[curve_red, mid_arrow] (X) -- (V2);
\draw[curve_red, mid_arrow] (V2) -- (V1) node[midway, right, text=black] {$\gamma_{v}$};

\end{tikzpicture} 
\caption{Case $k=k_x$ when constructing $\Gamma_x$.}
\label{fig:exampe_step_1}
\end{minipage}
%
%\hfill
\begin{minipage}{0.45\textwidth}
%\centering
\begin{tikzpicture}[scale=0.61,
    > = stealth,
    curve_red/.style={thick, red!70!black},
    curve_blue/.style={thick, blue!70!black},
]

% Coordinates for top points
\coordinate (z0p) at (2, 6);
\coordinate (z1p) at (2.95, 6);
\coordinate (zp)  at (4.7, 6);
\coordinate (z2p) at (8.6, 6);
\coordinate (z3p) at (8.75, 6);
% Coordinates for bottom points
\coordinate (z0)  at (1.2, 0);
\coordinate (z1)  at (2.3, 0);
\coordinate (z)   at (4.4, 0);
\coordinate (z2)  at (9.5, 0);
\coordinate (z3) at (9.8, 0);
% Horizontal lines
\draw (0, 6) -- (11.8, 6) node[right] {$\varepsilon_k$};

\draw (0, 0) -- (11.8, 0) node[right] {$\varepsilon_{k+1}$};

% Triangles (cones)
\draw[thick, fill=gray!30] (0.5, 2.5) -- (4.1, 2.5) -- (z1) -- cycle;
\draw[thick, fill=gray!30] (7.7, 2.5) -- (11.3, 2.5) -- (9.5,0) -- cycle;

\draw[fill=red!30] (z0p) -- (-0.2 ,2.5) -- (z0) -- (z1) -- (0.5, 2.5) -- (z1p) -- cycle;

\draw[fill=red!30] (z1p) -- (4.1, 2.5) -- (z1) -- (z2) -- (7.7, 2.5) -- (z2p) -- cycle;

\draw[fill=red!30] (z2p) -- (11.3, 2.5) -- (z2) -- (z3) -- (11.5, 2.5) -- (z3p) -- cycle;

\draw[curve_red, postaction={decorate}, 
    decoration={
        markings,
        mark=at position 0.25 with {\arrow{>}},
        mark=at position 0.75 with {\arrow{>}}
    }
]  (z0p) -- (-0.2 ,2.5) -- (z0);
\node[left] at (0.1, 2.9) {$\gamma_{z_0}$};

% \psi_{z_0} (blue)
\draw[curve_blue, postaction={decorate}, 
    decoration={
        markings,
        mark=at position 0.4 with {\arrow{>}},
        mark=at position 0.75 with {\arrow{>}}
    }
]  (z0p)  -- (z0);
\node[left] at (1.7, 2.85) {\small{$\overline{xz_0}$}};

% \psi_{z_1} (blue)
\draw[curve_blue, postaction={decorate}, 
    decoration={
        markings,
        mark=at position 0.25 with {\arrow{>}},
        mark=at position 0.75 with {\arrow{>}}
    }
]  (z1p)  -- (z1);
\draw[curve_blue, postaction={decorate}, 
    decoration={
        markings,
        mark=at position 0.25 with {\arrow{>}},
        mark=at position 0.75 with {\arrow{>}}
    }
]  (z3p)  -- (z3);

\draw[curve_red] (z1p) -- (0.5, 2.5) -- (z1);

\node[right] at (1.6, 3.4) {\small{$\overline{xz_1}$}};

\draw[curve_red, postaction={decorate}, 
    decoration={
        markings,
        mark=at position 0.25 with {\arrow{>}},
        mark=at position 0.75 with {\arrow{>}}
    }
]  (z1p) --(4.1, 2.5) -- (z1);
\node[right] at (3.3, 4.4) {$\gamma_{z_1}$};

% [x,z] (blue)
\draw[curve_blue, postaction={decorate}, 
    decoration={
        markings,
        mark=at position 0.35 with {\arrow{>}},
        mark=at position 0.75 with {\arrow{>}}
    }
]  (zp) -- (z);
\node[left] at (4.45, 1.1) {\small{$\overline{xz}$}};

\draw[curve_red, postaction={decorate}, 
    decoration={
        markings,
        mark=at position 0.4 with {\arrow{>}},
        mark=at position 0.75 with {\arrow{>}}
    }
]  (zp) -- (4.9, 2.5) --  (z);
\node[right] at (4.8, 1.5) {$\gamma_z$};

\draw[curve_red, postaction={decorate}, 
    decoration={
        markings,
        mark=at position 0.25 with {\arrow{>}},
        mark=at position 0.75 with {\arrow{>}}
    }
]  (z2p) -- (7.7, 2.5) -- (z2);
\node[left] at (8, 3.8) {$\gamma_{z_2}$};

\draw[curve_red, postaction={decorate}, 
    decoration={
        markings,
        mark=at position 0.25 with {\arrow{>}},
        mark=at position 0.75 with {\arrow{>}}
    }
]  (z2p) -- (11.3, 2.5) -- (z2);

\draw[curve_red, postaction={decorate}, 
    decoration={
        markings,
        mark=at position 0.25 with {\arrow{>}},
        mark=at position 0.75 with {\arrow{>}}
    }
]  (z3p) -- (11.5, 2.5) -- (z3);
\node[left] at (11.5, 3.8) {$\gamma_{z_3}$};

% \psi_{z_2} (blue)
\draw[curve_blue, postaction={decorate}, 
    decoration={
        markings,
        mark=at position 0.25 with {\arrow{>}},
        mark=at position 0.75 with {\arrow{>}}
    }
]  (z2p) -- (z2);
\node[right] at (7.8, 3.1) {\small{$\overline{xz_2}$}};

% Draw points (dots)
\fill (z0p) circle (1.5pt) node[above left] {$z_0'$};
\fill (z1p) circle (1.5pt) node[above] {$z_1'$};
\fill (zp)  circle (1.5pt) node[above right] {$z'$};
\fill (z2p) circle (1.5pt) node[above left] {$z_2'$};
\fill (z3p) circle (1.5pt) node[above right] {$z_3'$};

\fill (z0) circle (1.5pt) node[below left] {$z_0$};
\fill (z1) circle (1.5pt) node[below right] {$z_1$};
\fill (z)  circle (1.5pt) node[below] {$z$};
\fill (z2) circle (1.5pt) node[below] {$z_2$};
\fill (z3) circle (1.5pt) node[below right] {$z_3$};

% Crosses and Extra Nodes

\node[right] at (1, -0.8) {\small{$B_{k+1}$}};
\node[right] at (5.4, -0.8) {\small{$B_{k+1}$}};
%\node[below] at (10.5,-0.2) {$B_{k+1}$};
\node at (6.1,6.7) {\small{$B_{k}$}};

\draw[dashed] (0, 2.5) -- (11.8, 2.5) node[right] {$t_k \varepsilon_k$};
\end{tikzpicture}
\caption{Case $k>k_x$ when constructing  $\Gamma_x$.}
\label{fig:example_step_2}
\end{minipage}
\end{figure}

\medskip

$(3)$ Let us show that $E$ is a $W^{1,1}$-extension set. According to  Lemma \ref{lem:hom-full-norm} (a),  since $E$ is bounded, it is enough to prove that $E$ is a $L^{1,1}$-extension set. This way we can take advantage of the scaling invariance property of the homogeneous Sobolev seminorm. The idea is to build extension operators for every $k,j$,
 $$T_{k,j}\colon L^{1,1}( Q_{k,j}\setminus I_{k,j})\to L^{1,1}(Q_{k,j}) $$
 where $Q_{k,j}$ are pairwise disjoint closed squares such that $\overline{I}_{k,j}\subset Q_{k,j}$. Namely, we define
 $$Q_{k,j}=Q\left(c_{k,j}, \tfrac{2\epsilon_k}{3}\right) ,$$
 Next, we define a bounded continuous linear operator, whose existence follows from \cite{G-BRT24},
 $$\widetilde T \colon L^{1,p}(Q\setminus I)\to L^{1,1}(Q)$$
 where $Q=Q(0,2/3)$ and $I$ is the open equilateral triangle with one vertex at the point $(0,-1/3)$ and the opposite side lying on $\R\times \{0\}$.  Next, for every $x\in Q_{k,j}$, $u\in W^{1,1}(Q_{k,j}\setminus I_{k,j})$ and all $k,j$ we let
 $$T_{k,j}(u)(x)=\widetilde T(u_{k,j})\left(\frac{x-c_{k,j}}{\epsilon_k}\right),\quad \text{where} \; u_{k,j}(y)=u(\epsilon_k y+c_{k,j}),\;\forall y\in Q.$$  It is clear that $\|T_{k,j}\|\leq \|\widetilde T\|$ for all $k,j$, thanks to the scaling invariance of the homogeneous Sobolev seminorm. This allow to define an extension operator $T\colon L^{1,1}(E)\to L^{1,1}([-1,2]\times[-1,1])$ as
 $$T(u)(x)=\begin{cases}
     u(x)&\text{if}\; x\notin \bigcup_{k,j}I_{k,j}\\
     T_{k,j}(u)(x)&\text{if}\; x\in I_{k,j}
 \end{cases} .$$

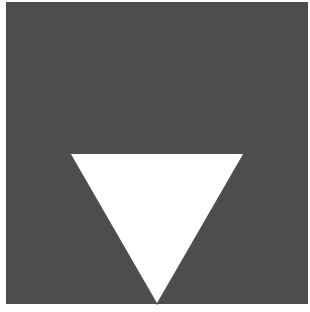
\begin{figure}

\begin{center}
\begin{tikzpicture}[scale=1]

\fill[black!70, scale=2.3] (-0.866,0) rectangle (0.866,1.732);

\filldraw[color=black!70, fill=white, scale=2.3](0,0) -- (-0.5,0.866) -- (0.5,0.866) -- cycle;;

\end{tikzpicture}
\end{center}
\caption{The set $Q\setminus I$ is a $W^{1,1}$-extension set without a weak $(1,1)$-Poincaré inequality.}
   \label{fig:triangle-square}

\end{figure}

\medskip

 $(4)$ For any $1<\eta<\alpha$ we fix $k_0\in\N$ so that for all $k\geq k_0$ we have  $ \tfrac{1}{3}<t_{k}<\tfrac{\eta}{3}$. Therefore, all lower vertexes of the triangles after iteration $k_0$ lie below the segment $[0,1]\times \{0\}$, that is $\left(t_k-\tfrac{\eta}{3}\right)\varepsilon_k<0$ for all $k\geq k_0$. In particular, for every $k\geq k_0$, using that the family of triangles $\{\eta I_{k,j}\}^{n_k}_{j=1}$ is disjoint (by the definition of $\alpha$ and since $\eta<\alpha$),
\[
\mathcal H^1\left(([0,1]\times\{0\}) \cap \bigcup^{n_k}_{j=1}\eta I_{k,j}\right) =\dfrac{2}{\sqrt{3}}\left( \dfrac{\eta}{3}-t_k \right) \epsilon_k n_k\in \left[\dfrac{2}{\sqrt{3}}\left( \dfrac{\eta}{3}-t_k \right), \dfrac{2}{\sqrt{3}}\left( \dfrac{\eta}{3}-t_k \right)(1+\epsilon_k)  \right].
\]
Consequently,  $$\lim_{k\to\infty}\mathcal H^1\left(([0,1]\times\{0\}) \cap \bigcup^{n_k}_{j=1}\eta I_{k,j}\right)=\dfrac{2(\eta-1)}{3\sqrt{3}}>0 .$$ 
A standard application of the Lebesgue differentiation theorem, noting that the triangles $\{I_{k,j}\}^{n_k}_{j=1}$ are equally distributed, yields 
\begin{equation}\label{eq:last_lemma}
\mathcal H^1\left(([0,1]\times\{0\}) \setminus  \bigcup_{k\ge N}\bigcup^{n_k}_{j=1}\eta I_{k,j}\right) =0
\quad \text{for all $N\in\N$}.
\end{equation}
We therefore claim that for any choice of parameters $1 < \eta <\tau\leq 3$ and open neighborhoods $U_{k,j}$ with $\eta I_{k,j} \subset U_{k,j} \subset \tau I_{k,j}$ the set
\[
E_{N}=\R^2 \setminus \bigcup_{k\ge N}\bigcup^{n_k}_{j=1}U_{k,j}
\]
does not locally satisfy a weak $(1,1)$-Poincar\'e inequality, for all $N\in\N$. Indeed, take points from $E_{N}$ that lie as close as we wish to the segment $(0,1) \times \{0\}$ from below and above; recall that by definition of $\alpha$, for a fixed $k\in\N$ the family $\{U_{k,j}\}^{n_k}_{j=1}$ is disjoint. Then, using \eqref{eq:last_lemma} we have that 
\begin{equation}\label{eq:examp_final}
\mathcal H^1(E_{N}\cap ((0,1)\times \{0\}))=0.
\end{equation}
Suppose by contradiction that the set $E_{N}$ satisfies the weak $(1,1)$-Poincaré inequality with constants $(\lambda,C)\in [1,\infty)\times (0,\infty)$. 
In such case, let $\widetilde B=B\left(\left(\tfrac{1}{2},0\right),\tfrac{1}{8}\right)$  and define a function 
$$u\colon\dfrac{1}{\lambda}\widetilde B\cap E_{N}\to\R  \quad ;\quad u(y_1,y_2)=\begin{cases}1 & \text{if}\;y_2\geq 0\\
0 & \text{if}\;y_2< 0
\end{cases}.$$
Thanks to property \eqref{eq:examp_final}, we have $u\in W^{1,1}(\widetilde B\cap E_{N})$ with $1$-weak upper gradient equal to zero, $\rho_u=0$. Note that the set of of curves crossing from $E_{N}\cap \{(y_1,y_2)\in\R^2:\, y_2> 0\}$ to $ E_{N}\cap \{(y_1,y_2)\in\R^2:\, y_2< 0\}$ within $E_{N}$ must intersect the $\mathcal H^1$-zero set $E_{N}\cap ((0,1)\times \{0\})$, so their $1$-modulus  is zero (see \cite[Theorem 5.12]{EG2015} and \cite[Proposition 1.48]{BB11}). Then, it is clear that,
$$  0<\dfrac{1}{|\widetilde B\cap E_{N}|}\int_{\widetilde B\cap E_{N}}|u(x)-u_{\widetilde B\cap E_{N}}|\,dx\leq  \dfrac{C\diam(\widetilde B\cap E_{N})}{| \lambda \widetilde B|}\int_{\lambda\widetilde B\cap E_{N}}\rho_u(x)\, dx=0,$$
and we reached a contradiction.
\end{proof}

\begin{remark}
Interestingly, the set $E$ from Example \ref{eq:final_example} satisfies the weak $(1,1)$-Poincaré inequality because we can construct Semmes families of curves that pass through the lower empty half-space $[-1,2]\times [-1,0]$. However, our proof of the $W^{1,1}$-extension property does not exploit this geometric feature. In fact, the disjoint neighborhoods of every triangle (see Figure \ref{fig:triangle-square}), used to define the extension fail to satisfy the weak $(1,1)$-Poincaré inequality, 
yet the extension remains possible as previously explained. 
Furthermore, this approach shows that $E \cap ([-1,2]\times [0,1])$ is a $W^{1,1}$-extension set even though the weak $(1,1)$-Poincaré inequality fails on it.

We suspect that there exist other methods to prove the $W^{1,1}$-extension property on Example \ref{eq:final_example}, which should rely on the weak $(1,1)$-Poincaré inequality of $E$, and that  such methods could provide valuable insights toward solving Question \ref{question}.
\end{remark}

\end{example}

\section{Equivalence of $BV$ and $W^{1,1}$-extension domains in the plane}\label{sec:equiv_BV_W11}

The main goal of this section is to show that every domain $\Omega\subset\R^2$  for which a weak $(1,1)$-Poincaré inequality holds, the $W^{1,1}$-extension and the $BV$-extension properties are equivalent. This is the content of  Theorem \ref{prop:equiv_of_BV_and W11}. For the proof, we will use some ideas from the characterization of $W^{1,1}$-extension domains in terms of the strong extension of sets of finite perimeter given in \cite{BR21}, together with the next key fact.

\begin{proposition}\label{prop:(1,1)_to_purely_unrect}
    Let $\Omega\subset\R^n$ be a bounded domain\footnote{The result is also true for unbounded domains, but in order to prove Theorem \ref{prop:equiv_of_BV_and W11} we are only interested in the bounded case.} that is Ahlfors regular and satisfies the weak $(1,1)$-Poincaré inequality. Then, if we write $\R^n\setminus \overline \Omega=\bigcup_{i\in I}\Omega_i$ where each $\Omega_i$ is open and connected, we have that $\partial \Omega\setminus \bigcup_{i\in I} \overline{\Omega}_i$ is purely $(n-1)$-unrectifiable. 
\end{proposition}
\begin{proof}
Call $C_p>0$, $\lambda_p\geq 1$ the Poincaré constants and $C_a>0$  the Ahlfors regularity constant of $\Omega$.

We argue by contradiction. Call $ H= \partial \Omega\setminus \bigcup_{i\in I} \overline{\Omega_i}$ and suppose that there exists an $L$-Lipschitz map $f\colon \R^{n-1}\to\R$ so that, after a suitable rotation if necessary, 
$$\mathcal H^{n-1}(\text{Graph}(f)\cap H)>0.$$
The set $\R^n\setminus \text{Graph}(f)$ consists of two open connected components that we denote by
\[
A := \{x = (x_1,\dots,x_n) \in \mathbb R^n \,:\, x_n > f(x_1,\dots,x_{n-1})\}
\]
and 
\[
B := \{x = (x_1,\dots,x_n) \in \mathbb R^n \,:\, x_n < f(x_1,\dots,x_{n-1})\}.
\]
Let us also write $F = \Omega \cap A$ and $G = H \cap \text{Graph}(f)$. Notice that $\partial^M F\cap\Omega=\text{Graph}(f)\cap \Omega$ so, according to \eqref{eq:Per-H^n-1}, $F$ is a set of finite perimeter on $\Omega$. Observe also that for any given  $r>0$ and $x\in G$ we have
\begin{equation}\label{eq:1-unrect_1}
P(F,B(x, r)\cap\Omega)=\mathcal H^{n-1}(\text{Graph}(f)\cap B(x, r)\cap  \Omega).
\end{equation}
We first prove that there exists some constant $C>0$ so that for every $x\in G$
\begin{equation}\label{eq:density1}
\limsup_{r\to 0}\dfrac{|F\cap B(x, r)|}{r^{n}}\geq C
\end{equation}
and
\begin{equation}\label{eq:density2}
\limsup_{r\to 0}\dfrac{|(\Omega \setminus F) \cap B(x, r)|}{r^{n}}\geq C.
\end{equation}
The situation being symmetric, it is enough to prove the first inequality \eqref{eq:density1}. Let us fix $x\in G$.
Since $x\in \text{Graph}(f)$ and $f$ is $L$-Lipschitz, the set 
\[
 R_{x,L} = \{y = (y_1,\dots, y_n) \in \mathbb R^n \,:\, y_n - x_n>L|(x_1,\dots,x_{n-1}) - (y_1,\dots,y_{n-1})|\}
\]
does not intersect $\text{Graph}(f)$. Since also $x \in H$, we have that $x \notin \partial \Omega_i$ for all $i$. Thus, for every $r>0$ we have  
\[
R_{x,2L} \cap B(x,r) \cap \Omega \neq \emptyset.
\]
This means that there exists a sequence of points $(x^i)_{i\geq 1} \subset R_{x,2L} \cap \Omega$ such that $|x^i-x| \to 0$. Let us call $r_i=|x^i-x|$. Since $f$ is $L$-Lipschitz, writing $\delta= \frac{1}{4(L+1)}$, we have 
\[
B\left(x^i,\delta r_i\right) \subset R_{x,L} \subset  A.
\]
Using that $x^i\in\Omega$ and the Ahlfors regularity of $\Omega$ we know that $|B(x^i,\delta r_i)\cap \Omega|\geq C_a C(n,\delta) r^n_i$. This proves \eqref{eq:density1}. The inequality \eqref{eq:density2} follows by repeating the above argument with a cone to the opposite direction from $R_{x,L}$.

By the continuity of the functions
\[
r\longmapsto\dfrac{|F\cap B(x, r)|}{r^{n}} \qquad\text{and}\qquad r\longmapsto\dfrac{|(\Omega \setminus F) \cap B(x, r)|}{r^{n}}
\]
and using the Ahlfors regularity of $\Omega$, that is $|\Omega \cap B(x, r)|>C_a r^n$ for all $0 < r < \diam(\Omega)$, we get
\begin{equation}\label{eq:density3}
\limsup_{r\to 0}\, \left(\min\left(\dfrac{|F\cap B(x, r)|}{r^{n}},\dfrac{|(\Omega \setminus F) \cap B(x, r)|}{r^{n}}\right)\right)\geq C.
\end{equation}
Using the fact that $\Omega$ satisfies a weak $(1,1)$-Poincaré inequality and \eqref{eq:density3} we can apply Lemma \ref{lem:Iso_per.} for the set $F\subset\Omega$ and the ball $B(x,2\delta r_i)\cap \Omega$ in order to get
\begin{equation}\label{eq:1-unrect_2}
\limsup_{r\to 0}\dfrac{P(F,B(x,\lambda_p r)\cap\Omega)}{r^{n-1}}\geq C>0.
\end{equation}
From \eqref{eq:1-unrect_1} and \eqref{eq:1-unrect_2} we get that for every $x\in G$,
\[
\limsup_{r\to 0}\dfrac{\mathcal H^{n-1}(B(x,\lambda_p r)\cap \text{Graph}(f)\cap\Omega)}{r^{n-1}}\geq C>0.
\]
However, by \cite[Theorem 2.6]{EG2015} we know that for $\mathcal H^{n-1}$-almost every point $x \in G$ we must have
\[
\lim_{r\to 0}\dfrac{\mathcal H^{n-1}(B(x,\lambda_p r)\cap \text{Graph}(f)\cap \Omega)}{r^{n-1}}=0.
\]
Since $\mathcal H^{n-1}(G)>0$, we get the desired contradiction.
\end{proof}
\begin{remark}
     Proposition \ref{prop:(1,1)_to_purely_unrect} can fail for closed sets $E\subset \R^2$. One may think about different examples: the fat Sierpi\'nski carpet from Section \ref{sec:Siperpinski carpet}, or a modification of \cite[Example 4.2]{BR21} in order to have the weak $(1,1)$-Poincaré inequality.
\end{remark}

\begin{proof}[Proof of Theorem \ref{prop:equiv_of_BV_and W11}]
The fact that the $W^{1,1}$-extension property implies the $BV$-extension property can be found in \cite[Lemma 2.4]{KMS2010}.
For the other implication assume $\Omega$ is a bounded $BV$-extension domain. By \cite[Proposition 2.3]{BR21}, we know that $\Omega$ is Ahlfors regular. Hence, we can apply Proposition \ref{prop:(1,1)_to_purely_unrect} to see that $\partial \Omega\setminus \bigcup_{i\in I} \overline{\Omega_i}$ is purely $1$-unrectifiable. Finally, by \cite[Theorem 1.4]{BR21}, we conclude that $\Omega$ is a $W^{1,1}$-extension domian.
\end{proof}

\begin{remark}
If $\Omega\subset \R^2$ is a domain with the  weak $(1,1)$-Poincaré inequality that in addition admits a linear $BV$-extension operator we do not know if we can get linearity for the $W^{1,1}$-extension operator as well.  The proof of \cite[Theorem 1.4]{BR21} does not help to give an answer to this question.
\end{remark}

\section*{Acknowledgements}
The authors want to thank  Sylvester Eriksson-Bique for fruitful discussions.

\bibliographystyle{amsplain}
\bibliography{biblio.bib}

\end{document}